\documentclass[10pt,rqno,letterpaper]{article}
\usepackage[utf8]{inputenc}
\usepackage{amsmath}
\usepackage{amsfonts}
\usepackage{amssymb}
\usepackage{graphicx}
\usepackage{mathrsfs}
\usepackage{upref,amsthm,amsxtra,exscale}
\usepackage{cite}
\usepackage[colorlinks=true,urlcolor=blue,
citecolor=red,linkcolor=blue,linktocpage,pdfpagelabels,
bookmarksnumbered,bookmarksopen]{hyperref}
\usepackage{fullpage}

\allowdisplaybreaks 
\usepackage{hyperref}

\usepackage{subcaption}
\usepackage{caption}

\newtheorem{theorem}{Theorem}[section]
\newtheorem{corollary}[theorem]{Corollary}
\newtheorem{remark}[theorem]{Remark}

\newtheorem{lemma}[theorem]{Lemma}
\newtheorem{proposition}[theorem]{Proposition}

\numberwithin{equation}{section}

\usepackage{xcolor}

\def\({\left(}
\def\){\right)}

\def\r{\mathbb{R}}
\def\rm{\mathbb{R}^m}
\def\sm{\mathbb{S}^m}

\def\n{\mathbb{N}}
\def\d{\mathrm{\,d}}
\def\dm{\mathrm{\,d}\mu_g}
\def\eps{\varepsilon}
\def\hook{\hookrightarrow}
\def\rh{\rightharpoonup}
\def\im{\int_{M}}

\def\irm{\int_{\r^m}}
\def\vp{\varphi}
\def\hat{\widehat}
\def\tilde{\widetilde}
\def\cC{\mathcal{C}}
\def\cE{\mathcal{E}}
\def\cH{\mathcal{H}}
\def\cK{\mathcal{K}}
\def\cM{\mathcal{M}}
\def\cN{\mathcal{N}}
\def\cP{\mathcal{P}}
\def\cJ{\mathcal{J}}
\def\cT{\mathcal{T}}
\def\cU{\mathcal{U}}

\def\R{\mathbb{R}}

\def\dist{\mathrm{dist}}

\let\div\undefined
\DeclareMathOperator{\div}{div}
\newcommand{\de}[1] {\mathrm{d} #1}

\def\sideremark#1{\ifvmode\leavevmode\fi\vadjust{\vbox to0pt{\vss
 \hbox to 0pt{\hskip\hsize\hskip1em
 \vbox{\hsize2.1cm\tiny\raggedright\pretolerance10000
  \noindent #1\hfill}\hss}\vbox to15pt{\vfil}\vss}}}%

\author{M\'onica Clapp\footnote{M. Clapp was partially supported by CONACYT (Mexico) through the grant for Project A1-S-10457.}, \ Angela Pistoia\footnote{A. Pistoia was partially supported by Fondi di Ateneo ``Sapienza'' Universit\`a di Roma (Italy). } \ and \ Hugo Tavares\footnote{H. Tavares was partially supported by the Portuguese government through FCT-Funda\c c\~ao para a Ci\^encia e a Tecnologia, I.P., under the projects UID/MAT/04459/2020, PTDC/MAT-PUR/28686/2017 and PTDC/MAT-PUR/1788/2020.} }
\title{Yamabe systems, optimal partitions, and nodal solutions to the Yamabe equation}
\date{\today}

\begin{document}
\maketitle

\begin{abstract}
We give conditions for the existence of regular optimal partitions, with an arbitrary number $\ell\geq 2$ of components, for the Yamabe equation on a closed Riemannian manifold $(M,g)$. 

To this aim, we study a weakly coupled competitive elliptic system of $\ell$ equations, related to the Yamabe equation. We show that this system has a least energy solution with nontrivial components if $\dim M\geq 10$, $(M,g)$ is not locally conformally flat and satisfies an additional geometric assumption whenever $\dim M=10$. Moreover, we show that the limit profiles of the components of the solution separate spatially as the competition parameter goes to $-\infty$, giving rise to an optimal partition. We show that this partition exhausts the whole manifold, and we prove the regularity of both the interfaces and the limit profiles, together with a free boundary condition.

For $\ell=2$ the optimal partition obtained yields a least energy sign-changing solution to the Yamabe equation with precisely two nodal domains. 
\medskip

 \medskip

\end{abstract}

\tableofcontents

\section{Introduction and statement of results}

Consider the Yamabe equation
\begin{equation} \label{eq:y}
\mathscr{L}_gu:=-\Delta_g u + \kappa_mS_gu = |u|^{2^*-2}u\qquad\text{on }M,
\end{equation}
where $(M,g)$ is a closed Riemannian manifold of dimension $m\geq 3$, $S_g$ is its scalar curvature, $\Delta_g:=\mathrm{div}_g\nabla_g$ is the Laplace-Beltrami operator, $\kappa_m:=\frac{m-2}{4(m-1)}$, and $2^*:=\frac{2m}{m-2}$ is the critical Sobolev exponent. We assume that the quadratic form induced by the conformal Laplacian $\mathscr{L}_g$ is coercive.

If $\Omega$ is an open subset of $M$, we consider the Dirichlet problem
\begin{equation} \label{eq:u}
\begin{cases}
\mathscr{L}_gu = |u|^{2^*-2}u &\text{ in }\Omega,\\
u=0 &\text{ on }\partial \Omega.
\end{cases}
\end{equation}
Let $H^1_g(M)$ be the Sobolev space of square integrable functions on $M$ having square integrable first weak derivatives, and let $H^1_{g,0}(\Omega)$ be the closure of $\cC_c^\infty(\Omega)$ in $H^1_g(M)$. The (weak) solutions of \eqref{eq:u} are the critical points of the $\cC^2$-functional $J_\Omega:H^1_{g,0}(\Omega)\to\r$ given by
$$J_\Omega(u):=\frac{1}{2}\int_\Omega(|\nabla_gu|_g^2+ \kappa_mS_gu^2)\dm -\frac{1}{2^*}\int_\Omega|u|^{2^*}\dm.$$
The nontrivial ones belong to the Nehari manifold
$$\cN_\Omega:=\{u\in H^1_{g,0}(\Omega):u\neq 0\text{ and }J'_\Omega(u)u=0\},$$
which is a natural constraint for $J_\Omega$. So, a minimizer for $J_\Omega$ over $\cN_\Omega$ is a nontrivial solution of \eqref{eq:u}, called a \emph{least energy solution}. Such a solution does not always exist. If $\Omega$ is the whole manifold $M$, it provides a solution to the celebrated Yamabe problem. In this case its existence was established thanks to the combined efforts of Yamabe \cite{y}, Trudinger \cite{tru}, Aubin \cite{a1} and Schoen \cite{s}. A detailed account is given in \cite{lp}.  

Set
$$c_\Omega:=\inf_{u\in\cN_\Omega}J_\Omega(u).$$
In this paper, given $\ell\geq 2$, we consider the optimal $\ell$-partition problem
\begin{equation} \label{eq:op}
\inf_{\{\Omega_1,\ldots,\Omega_\ell\}\in\cP_\ell}\;\sum_{i=1}^\ell c_{\Omega_i},
\end{equation}
where $\cP_\ell:=\{\{\Omega_1,\ldots,\Omega_\ell\}:\,\Omega_i\neq\emptyset \text{ is open in }M\text{ and }\Omega_i\cap \Omega_j=\emptyset\text{ if }i\neq j \}$. 
A solution to \eqref{eq:op} is an $\ell$-tuple $\{\Omega_1,\ldots,\Omega_\ell\}\in\cP_\ell$ such that $c_{\Omega_i}$ is attained for every $i=1,\ldots,\ell$, and
$$\sum_{i=1}^\ell c_{\Omega_i}=\inf_{\{\Theta_1,\ldots,\Theta_\ell\}\in\cP_\ell}\;\sum_{i=1}^\ell c_{\Theta_i}.$$
We call it an \emph{optimal $\ell$-partition} for the Yamabe equation on $(M,g)$.

Optimal partitions do not always exist. In fact, there is no optimal $\ell$-partition for the Yamabe equation on the standard sphere $\sm$ for any $\ell\geq 2$. This is because $c_\Omega$ is not attained in any open subset $\Omega$ of $\sm$ whose complement has nonempty interior. Indeed, by means of the stereographic projection $\Sigma:\sm\smallsetminus\{q\}\to\rm$ from a point $q\in\sm\smallsetminus\overline\Omega$, problem \eqref{eq:u} translates into
$$-\Delta u=|u|^{2^*-2}u\;\text{ in }\Sigma(\Omega),\qquad u=0 \;\text{ on }\partial[\Sigma(\Omega)].$$
It is well known that this problem does not have a least energy solution; see, e.g., \cite[Theorem III.1.2]{st}.

Our aim is to give conditions on $(M,g)$ which guarantee the existence of an optimal $\ell$-partition for every $\ell$. To this end, we follow the approach introduced by Conti, Terracini and Verzini in \cite{ctv1,ctv2} and Chang, Lin, Lin and Lin in \cite{CLLL} relating optimal partition problems with variational elliptic systems having large competitive interaction. 

We consider the competitive elliptic system
\begin{equation} \label{eq:s}
\mathscr{L}_g u_i = |u_i|^{2^*-2}u_i + \mathop{\sum_{j=1}^\ell}_{j\neq i} \lambda_{ij}\beta_{ij}|u_j|^{\alpha_{ij}}|u_i|^{\beta_{ij}-2}u_i\quad\text{on }M, \qquad i=1,\ldots,\ell,
\end{equation}
where $\lambda_{ij}=\lambda_{ji}<0$, $\alpha_{ij},\beta_{ij}>1$, $\alpha_{ij}=\beta_{ji}$, and $\alpha_{ij}+\beta_{ij}=2^*$. Firstly, we provide sufficient conditions for \eqref{eq:s} to have a least energy solution with nontrivial components; secondly, in the case $\alpha_{ij}=\beta_{ij}$ and $\lambda_{ij}\equiv \lambda$, we study the asymptotic profiles of such solutions as $\lambda\to -\infty$. As a byproduct, we obtain the existence of a regular optimal $\ell$-partition of \eqref{eq:op}, and the existence of a sign-changing solution of \eqref{eq:y} with two nodal domains. Our results read as follows.

\begin{theorem}\label{thm:existence}
Assume that one of the following two conditions holds true:
\begin{itemize}
\item[$(A1)$] $\dim M=3$, $(M,g)$ is not conformal to the standard $3$-sphere and $2<\alpha_{ij}<4$ for all $i,j=1,\ldots,\ell$.
\item[$(A2)$] $(M,g)$ is not locally conformally flat, $\dim M\geq 9$, and $\frac{8}{m-2}<\alpha_{ij}<\frac{2(m-4)}{m-2}$ for all $i,j=1,\ldots,\ell$ if $m:=\dim M=9$.
\end{itemize}
Then, the system \eqref{eq:s} has a least energy fully nontrivial solution $ (u_1,\ldots,u_\ell)$ such that $u_i\in\cC^2(M)$ and $u_i>0$ for every $i=1,\ldots,\ell$.
\end{theorem}

Note that, as $\alpha_{ij}\in(1,2^*-1)$, it satisfies $\frac{8}{m-2}<\alpha_{ij}<\frac{2(m-4)}{m-2}$ when $m>9$. By a least energy fully nontrivial solution we mean a minimizer of the variational functional for the system \eqref{eq:s} on a suitable constraint that contains only solutions with nonzero components, see Section \ref{sec:existence} below.

\begin{theorem} \label{thm:phase_separation}
Assume that 
\begin{itemize}
\item[$(A3)$] $(M,g)$ is not locally conformally flat and $\dim M\geq 10$. If $\dim M=10$ then
$$|S_g(q)|^2<\frac{5}{28}\,|W_g(q)|^2_g\qquad\forall q\in M,$$
where $W_g(q)$ is the Weyl tensor of $(M,g)$ at $q$.
\end{itemize}
Let $\lambda_{n}<0$ be such that $\lambda_{n}\to -\infty$ and set $\beta:=\frac{2^*}{2}=\frac{m}{m-2}$. For each $n\in\n$, let $(u_{n,1},\ldots,u_{n,\ell})$ be a least energy fully nontrivial solution to the system
\begin{equation} \label{eq:s_0}
\mathscr{L}_g u_i = |u_i|^{2^*-2}u_i + \sum_{\substack{j=1 \\ j\neq i}}^\ell \lambda_{n}\beta|u_j|^{\beta}|u_i|^{\beta-2}u_i\quad\text{on }M,\qquad i=1,\ldots,\ell,
\end{equation}
such that $u_{n,i}\in\cC^2(M)$ and $u_{n,i}>0$ for all $n\in\n$. Then, after passing to a subsequence, we have that
\begin{itemize}
\item[$(i)$]$u_{n,i}\to u_{\infty,i}$ strongly in $H_g^1(M)\cap\cC^{0,\alpha}(M)$ for every $\alpha\in (0,1)$, where  $u_{\infty,i}\geq 0$,\, $u_{\infty,i}\neq 0$,\, and $u_{\infty,i}|_{\Omega_i}$ is a least energy solution to the problem \eqref{eq:u} in $\Omega_i:=\{p\in M:u_{\infty,i}(p)>0\}$ for each $i=1,\ldots,\ell$. Moreover,
\[
\int_M \lambda_n u_{n,i}^\beta u_{n,j}^\beta\to 0 \text{ as } n\to \infty\quad \text{whenever } i\neq j.
\]
\item[$(ii)$] $u_{\infty,i}\in \cC^{0,1}(M)$ for each $i=1,\ldots, \ell$.
\item[$(iii)$]$\{\Omega_1,\ldots,\Omega_\ell\}\in\cP_\ell$ and it is an optimal $\ell$-partition for the Yamabe equation on $(M,g)$. In particular, each $\Omega_i$ is connected.
\item[$(iv)$] $\Gamma:=M\smallsetminus\bigcup_{i=1}^\ell\Omega_i=\mathscr R\cup\mathscr S$, where $\mathscr R\cap\mathscr S=\emptyset$, $\mathscr R$ is an  $(m-1)$-dimensional $\cC^{1,\alpha}$-submanifold of $M$ and $\mathscr S$ is a closed subset of $M$ with Hausdorff measure $\leq m-2$. In particular, $M=\cup_{i=1}^\ell \overline \Omega_i$. Moreover,
\begin{itemize}
\item given $p_0\in\mathscr R$ there exist $i\neq j$ such that
\[
\lim_{p\to p_0^+} |\nabla_g u_i(p)|^2=\lim_{p\to p_0^-}  |\nabla_g u_j(p)|^2\neq 0,
\]
where $p\to p_0^\pm$ are the limits taken from opposite sides of $\mathscr R$,
\item and for $p_0\in\mathscr S$ we have
\[
\lim_{p\to p_0}|\nabla_g u_i(p)|^2=0\quad \text{for every } i=1,\ldots, \ell.
\]
\end{itemize}
\item[$(v)$] If $\ell=2$, then\, $u_{\infty,1}-u_{\infty,2}$\,  is a least energy sign-changing solution to the Yamabe equation \eqref{eq:y}.
\end{itemize}
\end{theorem}

From Theorems \ref{thm:existence} and \ref{thm:phase_separation} we immediately obtain the following results.

\begin{theorem} \label{cor:op}
Assume $(A3)$. Then, for every $\ell\geq 2$ there exists an optimal $\ell$-partition $\{\Omega_1,\ldots,\Omega_\ell\}$ for the Yamabe equation on $(M,g)$ such that each $\Omega_i$ is connected and $M\smallsetminus\bigcup_{i=1}^\ell\Omega_i$ is the union of an $(m-1)$-dimensional $\cC^{1,\alpha}$-submanifold of $M$ and a closed subset whose Hausdorff measure is at most $m-2$.
\end{theorem}

\begin{theorem} \label{cor:nodal}
Assume $(A3)$. Then there exists a least energy sign-changing solution to the Yamabe equation \eqref{eq:y} having precisely two nodal domains.
\end{theorem}

The main difficulty in proving Theorem \ref{thm:existence} lies in the lack of compactness of the variational functional for the system \eqref{eq:s}. Least energy fully nontrivial solutions are given by minimization on a suitable constraint, but minimizing sequences may blow up, as it happens for instance when $(M,g)$ is the standard sphere. To prove Theorem \ref{thm:existence} we establish a compactness criterion (Proposition \ref{prop:compactness}) that generalizes the condition given by Aubin for the Yamabe equation \cite[Th\'eor\`eme 1]{a2}. To verify this criterion we introduce a test function and we make use of fine estimates established in \cite{epv} to show that, under assumptions $(A1)$ and $(A2)$, a minimizer exists.

The components of least energy fully nontrivial solutions to the system  \eqref{eq:s} may also blow up as the parameters $\lambda_{ij}$ go to $-\infty$. The standard sphere is again an example of this behavior. So, to prove Theorem \ref{thm:phase_separation}, we establish a condition that prevents blow-up (see Lemma \ref{lem:weak_partition}). To verify this condition we need to estimate the energy of suitable test functions. Rather delicate estimates are required, particularly in dimension $10$ - where not only the exponents but also the coefficients of the energy expansion play a role - leading to the geometric inequality stated in assumption $(A3)$. These estimates are derived in Appendix \ref{app:B}.

But the occurrence of blow-up is not the only delicate issue in proving Theorem \ref{thm:phase_separation}. To obtain an optimal $\ell$-partition we need that the limit profiles of the components of the solutions to \eqref{eq:s_0} are continuous. To this end, we show that the components $(u_{n,i})$ are uniformly bounded in the $\alpha$-H\"older norm. This requires subtle regularity arguments which are well known in the flat case, see e.g. \cite{cl, nttv, sttz, st}. We adapt some of these arguments (for instance, a priori bounds, blow-up arguments and monotonicity formulas) to obtain uniform H\"older bounds for general systems involving an anisotropic differential operator. This result (Theorem \ref{thm: holder bounds}) is interesting in itself.

In order to prove the optimal regularity of the limiting profiles $u_i$, the regularity of the free boundaries $M\smallsetminus\bigcup_{i=1}^\ell\Omega_i$ and the free boundary condition, we use local coordinates. This reduces the problem to the study of segregated profiles satisfying a system involving divergence type operators with variable coefficients. Using information arising from the variational system \eqref{eq:s_0}, we deduce limiting compatibility conditions between the $u_i$'s which allow to prove an Almgren-type monotonicity formula and to perform a blow-up analysis, combining what is known in case of the pure Laplacian \cite{cl, TavaresTerracini1,sttz} with some ideas from papers dealing with variable coefficient operators \cite{Kukavica,GarofaloGarciaAdv2014,GPGJMPA2016,SWsublinear}. This result (which we collect in a more general setting in Theorem \ref{thm:generaltheorem_Lip_Reg}) is also interesting in its own right.

As we mentioned before, optimal $\ell$-partitions on the standard sphere $\sm$ do not exist. However, if one considers partitions with the additional property that every set $\Omega_i$ is invariant under the action of a suitable group of isometries, then optimal $\ell$-partitions of this kind do exist and they give rise to sign-changing solutions to the Yamabe equation \eqref{eq:y} with precisely $\ell$-nodal domains for every $\ell\geq 2$; as shown in \cite{css}. 

Already in 1986 W.Y. Ding \cite{d} established the existence of infinitely many sign-changing solutions to \eqref{eq:y} on $\sm$, and quite recently Fern\'andez and Petean \cite{fp} showed that there is a solution with precisely $\ell$ nodal domains for each $\ell\geq 2$. These results, as those in \cite{css}, make use of the fact that there are groups of isometries of $\sm$ that do not have finite orbits.  Looking for solutions which are invariant under such isometries allows avoiding blow-up. On the other hand, sign-changing solutions  to \eqref{eq:y}  which blow-up along some special minimal submanifolds of the sphere   $\sm$ have been found by  Del Pino, Musso, Pacard and Pistoia in \cite{dmpp1,dmpp2}. The existence of a prescribed number of nodal solutions on some manifolds $(M,g)$ with symmetries having finite orbits is established in \cite{cf}.

However, the existence of nodal solutions to the Yamabe equation \eqref{eq:y} on an arbitrary manifold $(M,g)$ is largely an open problem. In \cite{ah} Ammann and Humbert established the existence of a least energy sign-changing solution when $(M,g)$ is not locally conformally flat and $\dim M\geq 11$. Theorem \ref{cor:nodal} recovers and extends this result (see Remark \ref{rem:ah}). We also note that an optimal $\ell$-partition $\{\Omega_1,\ldots,\Omega_\ell\}$ gives rise to what in \cite{ah} is called \emph{a generalized metric $\bar g:=\bar u^{2^*-2}g$ conformal to $g$} by taking $\bar u:=u_1+\cdots+u_\ell$ with $u_i$ a positive solution to \eqref{eq:y} in $\Omega_i$. So Theorem \ref{cor:op} may be seen as an extension of the main result in \cite{ah}.

We close this introduction with references to related problems. The study of elliptic systems like \eqref{eq:s} with critical exponents in euclidean spaces has been the subject of intensive research in the past two decades, starting from \cite{CZ1,ChenLin,CLZ}; without being exhaustive, we refer to the recent contributions \cite{tavaresyou,dovettapistoia} for a state of the art and further references. On the other hand, optimal partition problems is another active field of research: see for instance the book \cite{bucurbuttazzo} for an overview on a general theory using quasi-open sets and other relaxed formulations. Particular interest has been shown when the cost involves Dirichlet eigenvalues (leading to spectral optimal partitions) both in Euclidean spaces (see for instance the survey \cite{partition_survey1,partition_survey2} or the recent \cite{alper, rtt,tavareszilio} and references therein), and in the context of metric graphs (see e.g. \cite{kennedy1,kennedy2} and references).

\section{Compactness for the Yamabe system} \label{sec:existence}

We write $\langle\,\cdot\,,\,\cdot\,\rangle$ and $|\,\cdot\,|$ for the Riemannian metric and the norm in $(M,g)$ and for $v,w\in H^1_g(M)$ we define
$$\langle v,w\rangle_g:=\im\left(\langle\nabla_g v,\nabla_g w\rangle + \kappa_m S_g vw\right)\dm \qquad\text{and}\qquad \|v\|_g:=\sqrt{\langle v,v\rangle_g},$$
where $\nabla_g$ denotes the weak gradient. Since we are assuming that the conformal Laplacian $\mathscr{L}_g$ is coercive, $\|\cdot\|_g$ is a norm in $H^1_g(M)$, equivalent to the standard one, and the \emph{Yamabe invariant}
$$Y_g:=\inf_{u\in H^1_g(M)\smallsetminus\{0\}}\frac{\|u\|_g^2}{|u|_{g,2^*}^2}$$
of $(M,g)$ is positive. We write $|u|_{g,r}:=\left(\im|u|^r\dm\right)^{1/r}$ for the norm in $L^r_g(M)$, $r\in[1,\infty)$.

Set $\cH:=(H^1_g(M))^\ell$ and let $\mathcal{J}:\mathcal{H}\to\mathbb{R}$ be given by 
\begin{align*}
\cJ(u_1,\ldots,u_\ell) :=& \frac{1}{2}\sum_{i=1}^\ell\|u_i\|_g^2 - \frac{1}{2^*}\sum_{i=1}^\ell|u_i|_{g,2^*}^{2^*} - \frac{1}{2}\mathop{\sum_{i,j=1}^\ell}_{j\neq i}\im\lambda_{ij}|u_j|^{\alpha_{ij}}|u_i|^{\beta_{ij}}\dm.
\end{align*}
This functional is of class $\cC^1$ and its partial derivatives are
\begin{align*}
\partial_i\cJ(u_1,\ldots,u_\ell)v=&\,\langle u_i,v\rangle_g - \im|u_i|^{2^*-2}u_iv \dm\\
& - \mathop{\sum_{i=1}^\ell}_{j\neq i}\im\lambda_{ij}\beta_{ij}|u_j|^{\alpha_{ij}}|u_i|^{\beta_{ij}-2}u_iv\dm,\quad v\in H^1_g(M).
\end{align*}
Hence, the critical points of $\mathcal{J}$ are the solutions to the system \eqref{eq:s}. 

Note that every solution $u$ to the Yamabe equation \eqref{eq:y} gives rise to a solution of the system \eqref{eq:s} whose $i$-th component is $u$ and all other components are $0$. We are interested in \emph{fully nontrivial solutions}, i.e., solutions $(u_1,\ldots,u_\ell)$ such that every $u_i$ is nontrivial. They belong to the Nehari-type set
$$\cN := \{(u_1,\ldots,u_\ell)\in\cH:u_i\neq 0, \;\partial_i\cJ(u_1,\ldots,u_\ell)u_i=0, \; \forall i=1,\ldots,\ell\}.$$
Define
$$\hat c:=\inf_{(u_1,\ldots,u_\ell)\in\cN}\cJ(u_1,\ldots,u_\ell).$$
A fully nontrivial solution $u$ to \eqref{eq:s} is called a \emph{least energy fully nontrivial solution} if $\cJ(u)=\hat c$. 

\begin{remark}
\emph{
Since $\lambda_{ij}<0$, it is not hard to check that minimization of $\cJ$ on the classical Nehari manifold $\{(u_1,\ldots,u_\ell)\in\cH\smallsetminus \{(0,\ldots,0)\}:\ \sum_{i}\partial_i\cJ(u_1,\ldots,u_\ell)u_i=0\}$ leads necessarily to solutions with only one nonzero component.}
\end{remark}

\begin{proposition} \label{prop:minimum}
If $(u_1,\ldots,u_\ell)\in\cN$, then
$$0<Y_g^{m/2}\leq\|u_i\|_g^2\leq|u_i|_{g,2^*}^{2^*}\qquad\forall i=1,\ldots,\ell,$$
where $Y_g$ is the Yamabe invariant of $(M,g)$. Hence, $\cN$ is a closed subset of $\cH$.
\end{proposition}

\begin{proof}
Since $u_i\neq 0$, $\partial_i\cJ(u_1,\ldots,u_\ell)u_i=0$, and $\lambda_{ij}<0$, we have that
$$\|u_i\|^2_g=|u_i|_{g,2^*}^{2^*} + \sum_{j\neq i}\im\lambda_{ij}\beta_{ij}|u_j|^{\alpha_{ij}}|u_i|^{\beta_{ij}}\dm\leq|u_i|_{g,2^*}^{2^*}\leq Y_g^{-m/(m-2)}\|u_i\|^{2^*}_g.$$
Hence, $Y_g^{m/2}\leq\|u_i\|_g^2\leq|u_i|_{g,2^*}^{2^*}$, as claimed.
\end{proof}

For $u=(u_1,\ldots,u_\ell)\in\cH$ and $s=(s_1,\ldots,s_\ell)\in(0,\infty)^\ell$, we write $su:=(s_1u_1,\ldots,s_\ell u_\ell)$. 

\begin{proposition}\label{prop:mountain_pass}
Let $u=(u_1,\ldots,u_\ell)\in\cH$.
\begin{itemize}
\item[$(i)$] If
$$|u_i|_{g,2^*}^{2^*}>- \sum_{j\neq i}\im\lambda_{ij}\beta_{ij}|u_j|^{\alpha_{ij}}|u_i|^{\beta_{ij}}\dm\quad\forall i=1,\ldots,\ell,$$
then there exists $s_u\in(0,\infty)^\ell$ such that $s_uu\in\cN$.
\item[$(ii)$] If there exists $s_u\in(0,\infty)^\ell$ such that $s_uu\in\cN$, then $s_u$ is unique and
$$\cJ(s_uu)=\max_{s\in(0,\infty)^\ell} \cJ(su).$$
Moreover, $s_u$ depends only on the values
$$a_{u,i}:=\|u_i\|^2_g,\quad b_{u,i}:=|u_i|_{g,2^*}^{2^*},\quad d_{u,ij}:=\sum_{j\neq i}\im\lambda_{ij}\beta_{ij}|u_j|^{\alpha_{ij}}|u_i|^{\beta_{ij}}\dm,$$
$i=1,\ldots,\ell,$ and it depends continuously on them.
\end{itemize}
\end{proposition}

\begin{proof}
Define $J_u:(0,\infty)^\ell \to\r$ by
$$J_u(s):=\mathcal{J}(su) = \sum_{i=1}^\ell\frac{1}{2} a_{u,i}s_i^2 - \sum_{i=1}^\ell\frac{1}{2^*} b_{u,i}s_i^{2^*} -\sum_{i\neq j}\frac{1}{2}d_{u,ij} s_j^{\alpha_{ij}}s_i^{\beta_{ij}}.$$
If $u_i\neq 0$ for all $i=1,\ldots,\ell$, then, as 
$$s_i\,\partial_i J_{u}(s)=\partial_i\mathcal{J}(su)[s_iu_i],\qquad i=1,\ldots,\ell,$$
we have that $su\in\mathcal{N}$ iff  $s$ is a critical point of $J_u$. The statements $(i)$ and $(ii)$ follow immediately from \cite[Lemmas 2.1, 2.2 and 2.3]{cs}.
\end{proof}

\begin{remark}\label{rem:nehari}
\emph{
If $\ell=1$, then $\cN=\{u\in H_g^1(M):u\neq 0,\;\|u\|^2_g=|u|_{g,2^*}^{2^*}\}$ is the usual Nehari manifold for the Yamabe problem \eqref{eq:y} and $s_u\in\r$ is explicitely given by $s_u^{2^*-2}=\frac{\|u\|^2_g}{|u|_{g,2^*}^{2^*}}$. Hence, for every $0\neq u\in H_g^1(M)$,
$$\frac{1}{m}\left(\frac{\|u\|^2_g}{|u|_{g,2^*}^2}\right)^{m/2}=\cJ(s_uu)=\max_{s\in(0,\infty)} \cJ(su),$$
and $\hat c=\frac{1}{m}Y_g^{m/2}$.}
\end{remark}
\medskip

Set $\cT:=\{u\in\cH:\|u_i\|_g=1, \;\forall i=1,\ldots,\ell\}$, and let
$$\cU:=\{u\in\cT:su\in\cN\text{ for some }s\in(0,\infty)^\ell\}.$$
Following \cite[Proposition 3.1]{cs}, it is easy to see that $\cU$ is a nonempty open subset of $\cT$. Define $\Psi:\cU\to\r$ by
$$\Psi(u):=\cJ(s_uu),$$
with $s_u$ as in Proposition \ref{prop:mountain_pass}. This function has the following properties.

\begin{proposition}\label{prop:psi}
\begin{itemize}
\item[$(i)$] $\Psi\in\cC^1(\cU,\r)$.
\item[$(ii)$] Let $u_n\in\cU$. If $(u_n)$ is a Palais-Smale sequence for $\Psi$, then $(s_{u_n}u_n)$ is a Palais-Smale sequence for $\cJ$. Conversely, if $(u_n)$ is a Palais-Smale sequence for $\cJ$ and $u_n\in\cN$ for all $n\in\n$, then $\left(\frac{u_n}{\|u_n\|_g}\right)$ is a Palais-Smale sequence for $\Psi$.
\item[$(iii)$] Let $u\in \cU$. Then, $u$ is a critical point of $\Psi$ if and only if $s_u u$ is a fully nontrivial critical point of $\cJ$.
\item[$(iv)$] If $(u_n)$ is a sequence in $\cU$ such that $u_n\to u\in\partial\cU$, then $\Psi(u_n) \to\infty$.
\end{itemize}
\end{proposition}

\begin{proof}
The proof is identical to that of \cite[Theorem 3.3]{cs}.
\end{proof}

\begin{corollary} \label{cor:minimizer}
If $u\in\cN$ and $\cJ(u)=\hat c$, then $u$ is a fully nontrivial solution to the system \eqref{eq:s}.
\end{corollary}

\begin{proof}
Since  $\Psi\left(\frac{u}{\|u\|_g}\right)=\inf_\cU \Psi$ and $\cU$ is an open subset of the smooth Hilbert manifold $\cT$, we have that $\frac{u}{\|u\|_g}$ is a critical point of $\Psi$. By Proposition \ref{prop:psi}, $u$ is a critical point of $\cJ$.
\end{proof}

Recall that the operator $\mathscr{L}_g$ is conformally invariant, i.e., if $\tilde g=\vp^{2^*-2}g$, $\vp>0$, is a metric conformal to $g$, then
\begin{equation*}\label{confinv}
\mathscr{L}_{\tilde g}(\vp^{-1}u)= \vp^{-(2^*-1)}\mathscr L_g (u)\qquad \forall u\in H^1_g(M).
\end{equation*}
Since $\mathrm{d}\mu_{\tilde g}=\vp^{2^*}\dm$, we have that 
$$\|\vp^{-1}u\|_{\tilde g}=\|u\|_g\qquad\text{and}\qquad |\vp^{-1}u|_{\tilde g,2^*}=|u|_{g,2^*}\qquad \forall u\in H^1_g(M).$$
So, changing the metric within the conformal class of $g$ does not affect our problem.

Let $\sm$ be the standard $m$-sphere and $p\in\sm$. Since the stereographic projection $\sm\smallsetminus\{p\}\to\rm$ is a conformal diffeomorphism, the Yamabe invariant of $\sm$ is the best constant for the Sobolev embedding $D^{1,2}(\rm)\hook L^{2^*}(\rm)$,
$$\sigma_m:=\inf_{w\in D^{1,2}(\rm)\smallsetminus\{0\}}\frac{\|w\|^2}{|w|_{2^*}^2},$$
where $D^{1,2}(\rm):=\{w\in L^{2^*}(\rm):\nabla u\in [L^2(\rm)]^m\}$ equiped with the norm $\|w\|:=\left(\irm|\nabla w|^2\right)^{1/2}$, and $|w|_{2^*}$ is the norm of $w$ in $L^{2^*}(\rm)$. It is well known that
$$\sigma_m=\frac{m(m-2)}{4}\,\omega_m^{2/m},$$
where $\omega_m$ denotes the volume of $\sm$ \cite{a1,Talenti}.

It is shown in \cite[Proposition 4.6]{cs} that $\hat c = \inf_\cN \cJ$ is not attained if $M=\mathbb{R}^m$. Therefore, from the previous paragraph, $\hat c$ is also not attained if $M$ is the standard sphere $\sm$. Our aim is to investigate whether this infimum is attained in some other cases, at least for some values of $\alpha_{ij}$ and $\beta_{ij}$. 

To this end, we establish a compactness criterion for the system \eqref{eq:s} that extends a similar well known criterion for the Yamabe equation \cite[Theorem A]{lp}. The key ingredient is the following result of T. Aubin.

\begin{theorem}[Aubin 1976]\label{thm:aubin}
For every $\eps>0$ there exists $C_\eps>0$ such that
$$\sigma_m |u|_{g,2^*}^2\leq(1+\eps)\im|\nabla_gu|^2\dm + C_\eps\im u^2\dm\qquad\forall u\in H_g^1(M).$$
\end{theorem}
\begin{proof}
See \cite[Th\'eor\`eme 9]{a1} or \cite[Theorem 2.3]{lp}.
\end{proof}

For each $Z\subset \{1,\ldots,\ell\}$, let $(\mathscr{S}_Z)$ be the system of $\ell-|Z|$ equations
\begin{equation*}
(\mathscr{S}_Z)\qquad
\begin{cases}
\mathscr{L}_g u_i = |u_i|^{2^*-2}u_i + \sum\limits_{j\neq i} \lambda_{ij}\beta_{ij}|u_j|^{\alpha_{ij}}|u_i|^{\beta_{ij}-2}u_i\qquad\text{on }M, \\
i,j\in\{1,\ldots,\ell\}\smallsetminus Z,
\end{cases}
\end{equation*}
where $|Z|$ denotes the cardinality of $Z$. The fully nontrivial solutions of $(\mathscr{S}_Z)$ are the solutions $(u_1,\ldots,u_\ell)$ of \eqref{eq:s} which satisfy $u_i=0$ iff $i\in Z$. We write $\cJ_Z$ and $\cN_Z$ for the functional and the Nehari set associated to $(\mathscr{S}_Z)$, and define
\begin{equation}\label{eq:cZ}
\hat c_Z:=\inf_{u\in\cN_Z}\cJ_Z(u).
\end{equation}
The following compactness criterion is inspired by \cite[Lemma 4.10]{cs}.

\begin{proposition} \label{prop:compactness}
Assume that
$$\hat c <\min\left\{\hat c_Z + \frac{|Z|}{m}\,\sigma_m^{m/2}:\emptyset\neq Z\subset\{1,\ldots,\ell\}\right\}.$$
Then $\hat c$ is attained by $\cJ$ on $\cN$.
\end{proposition}

\begin{proof}
By Ekeland's variational principle and Proposition \ref{prop:psi} there is a sequence $(u_n)$ in $\cN$ such that $\cJ(u_n)\to \hat c$ and $\cJ'(u_n)\to 0$. Then, $(u_n)$ is bounded in $\cH$ and, after passing to a subsequence, $u_{n,i}\rh \bar u_i$ weakly in $H^1_g(M)$, $u_{n,i}\to \bar u_i$ strongly in $L^2_g(M)$ and $u_{n,i}\to\bar u_i$ a.e. on $M$, where $u_n=(u_{n,1},\ldots,u_{n,\ell})$. A standard argument shows that $\bar u=(\bar u_1,\ldots,\bar u_\ell)$ is a solution of the system \eqref{eq:s}.

To prove that $\bar u$ is fully nontrivial, set $Z:=\{i\in\{1,\ldots,\ell\}:\bar u_i=0\}$. As $u_n\in\cN$ and $\lambda_{ij}<0$ we have that $\|u_{n,i}\|_g^2\leq|u_{n,i}|_{g,2^*}^{2^*}$ and, as $u_{n,i}\to 0$ strongly in $L^2_g(M)$ for each $i\in Z$, Theorem \ref{thm:aubin} and Proposition \ref{prop:minimum} yield
\begin{align*}
\sigma_m\leq\frac{\im|\nabla_gu_{n,i}|^2\dm + o(1)}{|u_{n,i}|_{g,2^*}^2}=\frac{\|u_{n,i}\|^2_g+o(1)}{|u_{n,i}|_{g,2^*}^2}\leq\left(\|u_{n,i}\|^2_g\right)^{2/m}+o(1).
\end{align*}
So, after passing to a subsequence,
\begin{equation}\label{eq:claim1}
\lim_{n\to\infty}\|u_{n,i}\|_g^2\geq\sigma_m^{m/2} \qquad\forall i\in Z.
\end{equation} 
Therefore,
\begin{align*}
\hat c &=\lim_{n\to\infty}\cJ(u_n)=\lim_{n\to\infty}\frac{1}{m}\left(\sum_{i\not\in Z}\|u_{n,i}\|_g^2+\sum_{i\in Z}\|u_{n,i}\|_g^2\right)\\
&\geq\frac{1}{m}\sum_{i\not\in Z}\|\bar u_i\|_g^2+\frac{|Z|}{m}\,\sigma_m^{m/2}\geq \hat c_Z + \frac{|Z|}{m}\,\sigma_m^{m/2}.
\end{align*}
But then, our assumption implies that $Z=\emptyset$, i.e., $\bar u$ is fully nontrivial. Hence, $\bar u\in\cN$ and $\cJ(\bar u)= \hat c$.
\end{proof}

The proof of the following regularity result is standard; see, e.g., \cite[Appendix B]{st}. We include it here for the sake of completeness.

\begin{proposition} \label{prop:regularity}
Let $u=(u_1,\ldots,u_\ell)\in\cH$ be a solution to the system \eqref{eq:s}. Then, $u_i\in\cC^{2,\gamma}(M)$ for any $\gamma\in(0,1)$ such that $\gamma<\beta_{ij}-1$ for all $i,j=1,\ldots,\ell$.
\end{proposition}

\begin{proof}
Let $s>0$ and assume that $u_i\in L_g^{2(s+1)}(M)$. Note that this is true if $2(1+s)=2^*$. Fix $\eps>0$. For each $L>0$, define $\psi_{iL}:=\min\{u_i^s,L\}$. Then 
$\nabla_g(u_i\psi_{iL})=(1+s1_{\{u_i^s\leq L\}})\,\psi_{iL}\nabla_gu_i$ and $\nabla_g(u_i\psi_{iL}^2)=(1+2s1_{\{u_i^s\leq L\}})\,\psi_{iL}^2\nabla_gu_i$. So, since $\partial_i\cJ(u)[u_i\psi_{iL}]=0$ and $\lambda_{ij}<0$, for any $K>0$ we have that
\begin{align*}
&(1+s)^{-1}\im|\nabla_g(u_i\psi_{iL})|^2\dm\leq\im\psi_{iL}^2|\nabla_gu_i|^2\dm\\
&\leq\im\langle\nabla_gu_i,\nabla_g(u_i\psi_{iL}^2)\rangle\dm\\
&\leq\im|u_i|^{2^*-2}(u_i\psi_{iL})^2\dm - \im \kappa_mS_g(u_i\psi_{iL})^2\dm\\
&\leq K\im |u_i\psi_{iL}|^2\dm +\int_{|u_i|^{2^*-2}\geq K}|u_i|^{2^*-2}|u_i\psi_{iL}|^2\dm + C\im |u_i\psi_{iL}|^2\dm\\
&\leq(K+C_1)\im u_i^{2(s+1)}\dm + \left(\int_{|u_i|^{2^*-2}\geq K}|u_i|^{2^*}\dm\right)^{2/m}|u_i\psi_{iL}|_{g,2^*}^2\\
&\leq (K+C_2)\im u_i^{2(s+1)}\dm + \eta(K)(1+\eps)\sigma_m^{-1}\im|\nabla_g(u_i\psi_{iL})|^2\dm,
\end{align*}
where the last inequality is given by Theorem \ref{thm:aubin} and $C_1,C_2$ are positive constants independent of $L$ and $K$. Since
$$\eta(K):=\left(\int_{|u_i|^{2^*-2}\geq K}|u_i|^{2^*}\dm\right)^{2/m}\to 0\quad\text{as }K\to\infty,$$
we may fix $K$ such that $(1+s)\eta(K)(1+\eps)\sigma_m^{-1}=\frac{1}{2}$. Then, as $u_i\in L_g^{2(s+1)}(M)$, the inequality above yields
\begin{align*}
\im|\nabla_g(u_i\psi_{iL})|^2\dm&\leq(1+s)\eta(K)(1+\eps)\sigma_m^{-1}\im|\nabla_g(u_i\psi_{iL})|^2+\tilde C\\
&\leq\frac{1}{2}\im|\nabla_g(u_i\psi_{iL})|^2+\tilde C,
\end{align*}
with $\tilde C>0$ independent of $L$. Letting $L\to\infty$ we get that $\im|\nabla_g(u_i^{s+1})|^2\dm\leq 2\tilde C$. Hence, $u_i\in L_g^{2^*(s+1)}(M)$. Now, starting with $s$ such that $2(1+s)=2^*$ and iterating this argument, we conclude that $u_i\in L_g^r(M)$ for every $r\geq 1$. 
Since $u_i$ is a weak solution of the equation
$$-\Delta_g u_{i} = -\kappa_mS_gu_i + |u_{i}|^{2^*-1} + \sum\limits_{j\neq i} \lambda_{ij}\beta_{ij}|u_j|^{\alpha_{ij}}|u_i|^{\beta_{ij}-1}=:f_i,$$
from elliptic regularity \cite[Theorem 2.5]{lp} and the Sobolev embedding theorem \cite[Theorem 2.2]{lp} we get that $u_i\in\cC^{0,\gamma}(M)$ for any $\gamma\in(0,1)$. Then, $f_i\in\cC^{0,\gamma}(M)$ for any $\gamma\in(0,1)$ such that $\gamma<\beta_{ij}-1$ for all $j\neq i$, and, by elliptic regularity again, we conclude that $u_i\in\cC^{2,\gamma}(M)$ for any such $\gamma$.
\end{proof}

\section{The choice of the test function}

To prove the strict inequality in Proposition \ref{prop:compactness} we need a suitable test function. We follow the approach of Lee and Parker \cite{lp}.

Fix $N>m$. Given $p\in M$, there is a metric $\tilde g$ on $M$ conformal to $g$ such that 
$$\det \tilde g_{ij}=1+O(|x|^N)$$
in $\tilde g$-normal coordinates at $p$; see \cite[Theorem 5.1]{lp}. These coordinates are called \emph{conformal normal coordinates at $p$}.

Since the Yamabe invariant $Y_g$ is positive, the Green function $G_p$ for the conformal Laplacian $\mathscr{L}_g$ exists at every $p\in M$ and is strictly positive. Fix $p\in M$ and define the metric $\hat g:=G_p^{2^*-2}g$ on $\hat M:=M\smallsetminus\{p\}$. This metric is asymptotically flat of some order $\tau>0$ which depends on $M$. If $m=3,4,5$, or $(M,g)$ is locally conformally flat, the Green function has the asymptotic expansion
\begin{equation*}
G_p(x)=b_m^{-1}|x|^{2-m}+A(p)+O(|x|)
\end{equation*}
in conformal normal coordinates $(x^{i})$ at $p$, where $b_m=(m-2)\omega_{m-1}$ and $\omega_{m-1}$ is the volume of $\mathbb{S}^{m-1}$. The constant $A(p)$ is related to the mass of the manifold $(\hat M,\hat g)$. It follows from the positive mass theorems of Schoen and Yau \cite{SchYau1,SchYau2} that $A(p)>0$ if the manifold $(M,g)$ is not conformal to the standard sphere $\sm$ and, either $m<6$, or $(M,g)$ is locally conformally flat. In the other cases the expansion of the Green function $G_p$ involves the Weyl tensor $W_g(p)$ of $(M,g)$ at $p$; see \cite[Section 6]{lp} for details.

For $\delta>0$, let 
$$U_\delta(x):=[m(m-2)]^{(m-2)/4}\left(\frac{\delta}{\delta^2+|x|^2}\right)^{(m-2)/2},$$
written in conformal normal coordinates $(x^{i})$ at $p$ and, for suitably small $r>0$, define
\begin{equation*}
\hat{V}_{\delta,p}(x):=
\begin{cases}
b_m|x|^{m-2}U_\delta(x) &\text{if }|x|\le r,\\
b_mr^{m-2}U_\delta \left(r\frac{x}{|x|}\right) &\text{otherwise}.
\end{cases}
\end{equation*}
Note that $U_\frac{1}{\delta}(\frac{x}{|x|^2})=|x|^{m-2}U_\delta(x)$. So, up to a constant, $\hat{V}_{\delta,p}$ is the test function defined in \cite[Section 7]{lp}. Now set
\begin{equation}\label{eq:test}
V_{\delta,p}:=G_p\hat V_{\delta,p}.
\end{equation}
The following estimates were proved by Esposito, Pistoia and Vétois in \cite[Proof of Lemma 1]{epv}.

If $m=3,4,5$, or $(M,g)$ is locally conformally flat, then
\begin{align}\label{eq:1_epv}
\|V_{\delta,p}\|_g^2 &=\sigma_m^{m/2}+(m-2)\,\overline c_mA(p)\delta^{m-2}+O(\delta^{m-1}),\\
|V_{\delta,p}|_{g,2^*}^{2^*} &=\sigma_m^{m/2}+m^2\,\overline c_mA(p)\delta^{m-2}+O(\delta^{m-1}).\nonumber
\end{align}
If $(M,g)$ is not locally conformally flat and $m=6$,
\begin{align}\label{eq:2_epv}
\|V_{\delta,p}\|_g^2 &=\sigma_6^{3}+ \overline c_6|W_g(p)|_g^2\,\delta^{4}|\ln\delta| + O(\delta^4),\\
|V_{\delta,p}|_{g,2^*}^{2^*} &=\sigma_6^{3}+ 9\,\overline c_6|W_g(p)|_g^2\,\delta^{4}|\ln\delta| + O(\delta^4).\nonumber
\end{align}
If $(M,g)$ is not locally conformally flat and $m\geq 7$,
\begin{align}\label{eq:3_epv}
\|V_{\delta,p}\|_g^2 &=\sigma_m^{m/2}+\frac{(m-2)^2}{m+2}\,\overline c_m\omega_{m-1}|W_g(p)|_g^2\,\delta^{4}+O(\delta^{5}),\\
|V_{\delta,p}|_{g,2^*}^{2^*} &=\sigma_m^{m/2}+ \frac{m^2}{m-4}\,\overline c_m\omega_{m-1}|W_g(p)|_g^2\,\delta^{4}+O(\delta^{5}).\nonumber
\end{align}
$\overline c_m$ is a positive constant depending only on $m$. In particular,
\begin{equation} \label{eq:c_m}
\overline c_m = \frac{1}{192}\,\frac{(m+2)\left[m(m-2)\right]^\frac{m-2}{2}}{2^{m-1}(m-6)(m-1)}\,\frac{\omega_m}{\omega_{m-1}} \qquad\text{if \ }m\geq 7.
\end{equation}
From these estimates we derive the following result.

\begin{lemma}\label{lem:estimates}
Assume that $(M,g)$ is not conformal to the standard sphere $\sm$. Then, there exist $p\in M$ and $C_0>0$ such that
$$\frac{1}{m}\left(\frac{\|V_{\delta,p}\|^2_g}{|V_{\delta,p}|_{g,2^*}^2}\right)^{m/2}\leq \frac{1}{m}\sigma_m^{m/2} - C_0R(\delta) +o(R(\delta))$$
for all $\delta>0$ sufficiently small, where
\begin{equation*}
R(\delta)=
\begin{cases}
\delta^{m-2} &\text{if either }m<6 \text{ or }(M,g)\text{ is l.c.f.},\\
\delta^4|\ln\delta| &\text{if }m=6 \text{ and }(M,g)\text{ is not l.c.f.},\\
\delta^4 &\text{if }m>6 \text{ and }(M,g)\text{ is not l.c.f.}\\
\end{cases}
\end{equation*}
\end{lemma}

\begin{proof}
By Remark \ref{rem:nehari},
$$\frac{1}{m}\left(\frac{\|V_{\delta,p}\|^2_g}{|V_{\delta,p}|_{g,2^*}^2}\right)^{m/2} =\frac{1}{2}\|s_\delta V_{\delta,p}\|^2_g-\frac{1}{2^*}|s_\delta V_{\delta,p}|_{g,2^*}^{2^*},$$
where $s_\delta^{2^*-2}=\frac{\|V_{\delta,p}\|^2_g}{|V_{\delta,p}|_{g,2^*}^{2^*}}$.

If $m=3,4,5$, or $(M,g)$ is locally conformally flat, the positive mass theorem ensures that $A(p)>0$ for any $p\in M$, and from \eqref{eq:1_epv} we get that
\begin{align*}
\frac{1}{m}\left(\frac{\|V_{\delta,p}\|^2_g}{|V_{\delta,p}|_{g,2^*}^2}\right)^{m/2} &=
\left(\frac{s_\delta^2}{2}-\frac{s_\delta^{2^*}}{2^*}\right)\sigma_m^{m/2} + C\left(\frac{s_\delta^2}{2}(m-2)-\frac{s_\delta^{2^*}}{2^*}m^{2}\right)\delta^{m-2} \\
&\qquad + o(\delta^{m-2})\\
&\leq\frac{1}{m}\sigma_m^{m/2} + \frac{s_\delta^2}{2}C(m-2)(1-ms_\delta^{2^*-2})\delta^{m-2}  + o(\delta^{m-2}),
\end{align*}
where $C$ is a positive constant. 

If $(M,g)$ is not locally conformally flat and $m\geq 6$, we choose $p\in M$ such that $|W_g(p)|_g^2>0$. Then, if $m=6$, estimates \eqref{eq:2_epv} yield
\begin{align*}
&\frac{1}{6}\left(\frac{\|V_{\delta,p}\|^2_g}{|V_{\delta,p}|_{g,3}^2}\right)^{3} \leq\frac{1}{6}\sigma_6^3 + \frac{s_\delta^2}{2}C(1-3s_\delta)\,\delta^{4}|\ln\delta| + o(\delta^{4}|\ln\delta|),
\end{align*}
and, if $m>6$, from \eqref{eq:3_epv} we derive
\begin{align*}
\frac{1}{m}\left(\frac{\|V_{\delta,p}\|^2_g}{|V_{\delta,p}|_{g,2^*}^2}\right)^{m/2} &\leq \frac{1}{m}\sigma_m^{m/2}+ \frac{s_\delta^2}{2}(m-2)C\left(\frac{m-2}{m+2}-s_\delta^{2^*-2}\frac{m}{m-4}\right)\delta^{4}+ o(\delta^{4}),
\end{align*}
for some positive constant $C$. Since $s_\delta\to 1$ as $\delta\to 0$, our claim is proved.
\end{proof}

\begin{lemma}\label{lem:alpha}
Let $R(\delta)$ be as in \emph{Lemma \ref{lem:estimates}} and $\alpha\in[1,\infty)$. Then,
$$\im |V_{\delta,p}|^{\alpha}\dm = o(R(\delta)),$$
if and only if 
\begin{itemize}
\item either $m=3$, $(M,g)$ is not conformal to $\mathbb{S}^3$ and $2<\alpha<4$,
\item or $(M,g)$ is not locally conformally flat, $m\geq 9$, and $\tfrac{8}{m-2}<\alpha<\tfrac{2(m-4)}{m-2}$.
\end{itemize}
\end{lemma}

\begin{proof}
Set $\gamma:=\frac{m-2}{2}\alpha$. From \eqref{eq:test} we deduce
\begin{equation*}
I_\alpha:=\im |V_{\delta,p}|^\alpha \dm=
\begin{cases}
O(\delta^\gamma)&\text{if }\gamma <\frac{m}{2},\\
O(\delta^\frac{m}{2}|\ln\delta|)&\text{if }\gamma =\frac{m}{2},\\
O(\delta^{m-\gamma})&\text{if }\gamma >\frac{m}{2}.\\
\end{cases}
\end{equation*}
Therefore,
\begin{align*}
&I_\alpha=o(\delta^{m-2}) \Longleftrightarrow m-2<\gamma<2\Longleftrightarrow m=3\text{ and }2<\alpha<4,\\
&I_\alpha=o(\delta^4)\Longleftrightarrow 4<\gamma<m-4\Longleftrightarrow m\geq 9\text{ and }\tfrac{8}{m-2}<\alpha<\tfrac{2(m-4)}{m-2},
\end{align*}
and our claim is proved.
\end{proof}

\begin{proposition} \label{prop:existence}
If either one of the assumptions $(A1)$ or $(A2)$ of \emph{Theorem \ref{thm:existence}} holds true, then
\begin{equation} \label{eq:estimate}
\hat c<\min\left\{\hat c_Z + \frac{|Z|}{m}\sigma_m^{m/2}:\emptyset\neq Z\subset\{1,\ldots,\ell\}\right\}.
\end{equation}
\end{proposition}

\begin{proof}
We prove this statement by induction on $\ell$.

If $\ell=1$ the system reduces to the Yamabe equation \eqref{eq:y}, and \eqref{eq:estimate} is equivalent to $Y_g<\sigma_m$. This inequality follows from Lemma \ref{lem:estimates} taking $\delta$ small enough.

Assume that the statement is true for every system $(\mathscr{S}_Z)$ with $|Z|\geq 1$ (i.e., for every system of less than $\ell$ equations). Then, the proof of \eqref{eq:estimate} reduces to showing that
$$\hat c<\min\left\{\hat c_Z + \frac{1}{m}\sigma_m^{m/2}:|Z|=1\right\}.$$
Without loss of generality, we may assume that $Z=\{\ell\}$. By Proposition \ref{prop:compactness} and the induction hypothesis, there exists $(u_1,\ldots,u_{\ell-1})\in \cN_Z$ such that $\cJ_Z(u_1,\ldots,u_{\ell-1})=\hat c_Z$. By Proposition \ref{prop:regularity}, each $u_i\in\cC^0(M)$.

Let $V_{\delta,p}$ be as in Lemma \ref{lem:estimates}. Since $\alpha_{ij}\in(1,\frac{m+2}{m-2})$, we have that
$$\im |V_{\delta,p}|^{\alpha_{ij}}|u_i|^{\beta_{ij}}\dm\leq \max_{q\in M}|u_i(q)|^{\beta_{ij}}\im|V_{\delta,p}|^{\alpha_{ij}}\dm \to 0\quad\text{as }\delta\to 0.$$
Hence, there exists $\delta_0>0$ such that, for every $\delta\in(0,\delta_0)$,
\begin{align*}
&|u_i|_{g,2^*}^{2^*}+\sum_{j\neq i}\beta_{ij}\,\lambda_{ij}\im |u_j|^{\alpha_{ij}}|u_i|^{\beta_{ij}}\dm + \beta_{i\ell}\,\lambda_{i\ell}\im |V_{\delta,p}|^{\alpha_{i\ell}}|u_i|^{\beta_{i\ell}}\dm\\
&\qquad\qquad=\|u_i\|_g^2 + \beta_{i\ell}\,\lambda_{i\ell}\im |V_{\delta,p}|^{\alpha_{i\ell}}|u_i|^{\beta_{i\ell}}\dm >0,\qquad i,j=1,\ldots,\ell-1,
\end{align*}
and
$$|V_{\delta,p}|_{g,2^*}^{2^*}+\sum_{j=1}^{\ell-1}\beta_{\ell j}\,\lambda_{\ell j}\im |u_j|^{\alpha_{\ell j}}|V_{\delta,p}|^{\beta_{\ell j}}\dm > 0.$$
Then, Proposition \ref{prop:mountain_pass} asserts that there are $0<r<R<\infty$ and $s_{\delta,1},\ldots,s_{\delta,\ell}\in [\,r,R\,]$ such that
$$u_\delta =(s_{\delta,1}u_1,\ldots,s_{\delta,\ell-1}u_{\ell-1},s_{\delta,\ell}V_{\delta,p})\in\cN\qquad\forall \delta\in(0,\delta_0).$$
By Proposition \ref{prop:mountain_pass} and Lemmas \ref{lem:estimates} and \ref{lem:alpha},
\begin{align*}
\hat c &\leq\cJ(u_\delta)\\
& = \cJ_Z(s_{\delta,1}u_1,\ldots,s_{\delta,\ell-1}u_{\ell-1}) + \frac{1}{2}s_{\delta,\ell}^2\|V_{\delta,p}\|_g^2 - \frac{1}{2^*}s_{\delta,\ell}^{2^*}|V_{\delta,p}|^{2^*}_{g,2^*}\\
&\qquad - \sum_{i=1}^{\ell-1}s_{\delta,\ell}^{\alpha_{i\ell}} s_{\delta,i}^{\beta_{i\ell}}\lambda_{i\ell}\im |V_{\delta,p}|^{\alpha_{i\ell}}|u_i|^{\beta_{i\ell}}\dm \\
&\leq \hat c_Z + \frac{1}{m}\left(\frac{\|V_{\delta,p}\|_g^{2}}{|V_{\delta,p}|^{2}_{g,2^*}}\right)^{m/2} + \sum_{i=1}^{\ell-1} R^{2^*}|\lambda_{i\ell}|\max_{q\in M}|u_i(q)|^{\beta_{i\ell}}\im|V_{\delta,p}|^{\alpha_{i\ell}}\dm\\
&\leq \hat c_Z+ \frac{1}{m}\sigma_m^{m/2} - C_0R(\delta) +o(R(\delta))\\
&<\hat c_Z+\frac{1}{m}\sigma_m^{m/2},
\end{align*}
for $\delta$ small enough, as claimed.
\end{proof}
\medskip

\begin{proof}[Proof of Theorem \ref{thm:existence}]
By Propositions \ref{prop:existence} and \ref{prop:compactness}, there is $u=(u_1,\ldots,u_\ell)\in\cN$ such that $\cJ(u)=\hat c$. Then, $\bar u=(|u_1|,\ldots,|u_\ell|)\in\cN$ and $\cJ(\bar u)=\hat c$, and the result follows from Corollary \ref{cor:minimizer}, Proposition \ref{prop:regularity} and the strong maximum principle \cite[Theorem 2.6]{lp}.
\end{proof}

\section{Phase separation and optimal partitions}

In this section we restrict to the case $\lambda_{ij}=\lambda$ and $\alpha_{ij}=\beta_{ji}=\frac{2^*}{2}=:\beta$. Our aim is to study the behavior of least energy fully nontrivial solutions to the system \eqref{eq:s_0} as $\lambda\to-\infty$ and to derive the existence and regularity of an optimal partition.

Let $\Omega$ be an open subset of $M$. As mentioned in the introduction, the nontrivial solutions of \eqref{eq:u} are the critical points of the restriction of the functional
$$J_\Omega(u):=\frac{1}{2}\|u\|_g^2-\frac{1}{2^*}|u|_{g,2^*}^{2^*}$$
to the Nehari manifold
$$\cN_\Omega:=\{u\in H^1_{g,0}(\Omega):u\neq 0\text{ and }\|u\|_g^2=|u|_{g,2^*}^{2^*}\}.$$
So, a minimizer for $J_\Omega$ on $\cN_\Omega$ is a solution of \eqref{eq:u}. Note that $J_\Omega(u):=\frac{1}{m}\|u\|_g^2$ if $u\in\cN_\Omega$. Therefore,
$$c_\Omega:=\inf_{\cN_\Omega}J_\Omega=\inf_{u\in\cN_\Omega}\frac{1}{m}\|u\|_g^2.$$
Define
\begin{align*}
\cM_\ell:=&\{(u_1,\ldots,u_\ell)\in\cH:u_i\neq 0,\,\|u_i\|_g^2=|u_i|_{g,2^*}^{2^*},\,u_iu_j=0\text{ on }M\text{ if }i\neq j\},\\
c_\ell^*:=&\inf_{(u_1,\ldots,u_\ell)\in\cM_\ell}\,\frac{1}{m}\sum_{i=1}^\ell\|u_i\|_g^2.
\end{align*}

\begin{lemma} \label{lem:cont->op}
Assume there exists $(u_1,\ldots,u_\ell)\in\cM_\ell$ such that $u_i\in\cC^0(M)$ and 
$$\frac{1}{m}\sum_{i=1}^\ell\|u_i\|_g^2=c_\ell^*.$$
Set $\Omega_i:=\{p\in M:u_i(p)\neq0\}$. Then $\{\Omega_1,\ldots,\Omega_\ell\}$ is an optimal $\ell$-partition for the Yamabe problem on $(M,g)$, each $\Omega_i$ is connected and $J_{\Omega_i}(u_i)=c_{\Omega_i}$ for all $i=1,\ldots,\ell$.
\end{lemma}

\begin{proof}
As $u_i$ is continuous and nontrivial, we have that $\Omega_i$ is an open nonempty subset of $M$ and $u_i\in\cN_{\Omega_i}$. Moreover, $\Omega_i\cap\Omega_j=\emptyset$ for $i\neq j$ because $u_iu_j=0$. Therefore, $\{\Omega_1,\ldots,\Omega_\ell\}\in\cP_\ell$. 

To prove the last two statements of the lemma we argue by contradiction. If, say, $\Omega_1$ were the disjoint union of two nonempty open sets $\Theta_1$ and $\Theta_2$, then, setting $\bar u_1(p):=u_1(p)$ if $p\in\Theta_1$ and $\bar u_1(p):=0$ if $p\in\Theta_2$, we would have that $(\bar u_1,\ldots,u_\ell)\in\cM_\ell$ and 
$$\frac{1}{m}\left(\|\bar u_1\|_g^2 + \sum_{i=2}^\ell\|u_i\|_g^2\right)<c_\ell^*,$$
contradicting the definition of $c_\ell^*$.  

Similarly, if $J_{\Omega_i}(u_i)>c_{\Omega_i}$ for some $i$, then there would exist $v_i\in\cN_{\Omega_i}$ with $c_{\Omega_i}<J_{\Omega_i}(v_i)<J_{\Omega_i}(u_i)$. But then 
$$\frac{1}{m}\sum_{j\neq i}\|u_j\|_g^2+\frac{1}{m}\|u_i\|_g^2<c_\ell^*,$$
which is again a contradiction. Hence, $J_{\Omega_i}(u_i)=c_{\Omega_i}$ for all $i=1,\ldots,\ell$ and
$$\inf_{\{\Theta_1,\ldots,\Theta_\ell\}\in\cP_\ell}\;\sum_{i=1}^\ell c_{\Theta_i}\leq\frac{1}{m}\sum_{i=1}^\ell\|u_i\|_g^2=c_\ell^*\leq\inf_{\{\Theta_1,\ldots,\Theta_\ell\}\in\cP_\ell}\;\sum_{i=1}^\ell c_{\Theta_i}.$$
This shows that $\{\Omega_1,\ldots,\Omega_\ell\}$ is an optimal $\ell$-partition and concludes the proof.
\end{proof}

\begin{lemma} \label{lem:weak_partition}
Let $\lambda_n<0$ and $u_n=(u_{n,1},\ldots,u_{n,\ell})$ be a least energy fully nontrivial solution to the system \eqref{eq:s_0}. Assume that $\lambda_n\to -\infty$ as $n\to\infty$ and that $u_{n,i}\geq 0$ for all $n\in\n$. Assume further that
\begin{equation} \label{eq:assumption}
c_\ell^* <\min\left\{c_k^*+\frac{\ell-k}{m}\,\sigma_m^{m/2}:1\leq k<\ell\right\}.
\end{equation}
Then there exists $(u_{\infty,1},\ldots,u_{\infty,\ell})\in\cM_\ell$ such that, after passing to a subsequence, $u_{n,i}\to u_{\infty,i}$ strongly in $H_g^1(M)$,\, $u_{\infty,i}\geq 0$, and
$$c_\ell^*=\frac{1}{m}\sum_{i=1}^\ell\|u_{\infty,i}\|_g^2.$$ Moreover,
\begin{equation}\label{eq:segregation}
\int_M \lambda_n u_{n,j}^{\beta} u_{n,i}^{\beta}\to 0 \text{ as } n\to \infty\quad \text{whenever } i\neq j.
\end{equation}
\end{lemma}

\begin{proof}
To highlight the role of $\lambda_n$, we write $\cJ_n$ and $\cN_n$ for the functional and the Nehari set associated to the system \eqref{eq:s_0} and we define
$$\hat c_n:=\inf_{\cN_n}\cJ_n.$$
Note that $\cM_\ell\subset\cN_n$ for each $n\in\n$. Therefore, 
$$\frac{1}{m}\sum_{i=1}^\ell\|u_{n,i}\|_g^2=\hat c_n\leq c_\ell^*<\infty\qquad\forall n\in\n.$$
So, after passing to a subsequence, we get that $u_{n,i} \rh u_{\infty,i}$ weakly in $H^1_g(M)$, $u_{n,i} \to u_{\infty,i}$ strongly in $L^2_g(M)$ and $u_{n,i} \to u_{\infty,i}$ a.e. on $M$, for each $i=1,\ldots,\ell$. Hence, $u_{\infty,i} \geq 0$. Moreover, as $\partial_i\cJ_n(u_n)[u_{n,i}]=0$, we have that, for each $j\neq i$,
\begin{align*}
0&\leq\im\beta\,|u_{n,j}|^{\beta}|u_{n,i}|^{\beta}\dm\leq \frac{|u_{n,i}|^{2^*}_{g,2^*}}{-\lambda_n}\leq \frac{C}{-\lambda_n}.
\end{align*}
Fatou's lemma then yields 
$$0 \leq \im |u_{\infty,j}|^{\beta}|u_{\infty,i}|^{\beta}\dm \leq \liminf_{n \to \infty} \im |u_{n,j}|^{\beta}|u_{n,i}|^{\beta}\dm = 0.$$
Hence, $u_{\infty,j} u_{\infty,i} = 0$ a.e. on $M$ whenever $i\neq j$. On the other hand, as $\partial_i\cJ_n(u_n)[u_{\infty,i}]=0$ and $u_{n,i}\geq 0,\,u_{\infty,i}\geq 0$, we have that
$$\langle u_{n,i},u_{\infty,i}\rangle_g\leq\im u_{n,i}^{2^*-1}u_{\infty,i}\dm.$$
So, passing to the limit as $n\to\infty$ we obtain
\begin{equation} \label{eq:comparison}
\|u_{\infty,i}\|_g^2 \leq |u_{\infty,i}|_{g,2^*}^{2^*}\qquad\forall  i=1,\ldots,\ell.
\end{equation}
We claim that
\begin{equation} \label{eq:nontrivial}
u_{\infty,i}\neq 0\qquad\forall i=1,\ldots,\ell.
\end{equation}
To prove this claim, let $Z:=\{i\in\{1,\ldots,\ell\}:u_{\infty,i}=0\}$. After reordering, we may assume that $Z$ is either empty or $Z=\{k+1,\ldots,\ell\}$ for some $0\leq k<\ell$. Then, arguing as we did to prove \eqref{eq:claim1}, we get that
\begin{equation*}
\lim_{n\to\infty}\|u_{n,i}\|_g^2\geq\sigma_m^{m/2} \qquad\forall i\in Z.
\end{equation*}
On the other hand, if $i\not\in Z$, there exists $t_i\in(0,\infty)$ such that $\|t_iu_{\infty,i}\|_g^2 = |t_iu_{\infty,i}|_{g,2^*}^{2^*}$. So we have that $(t_1u_{\infty,1},\ldots,t_k u_{\infty,k})\in \cM_k$. The inequality \eqref{eq:comparison} implies that $t_i\in (0,1]$. Therefore,
\begin{align} \label{eq:convergence}
&c_k^*+\frac{\ell-k}{m}\sigma_m^{m/2} \leq \frac{1}{m}\sum_{i=1}^k\|t_iu_{\infty,i}\|_g^2 +\frac{\ell-k}{m}\sigma_m^{m/2}\\
&\leq \frac{1}{m}\sum_{i=1}^k\|u_{\infty,i}\|_g^2+\frac{\ell-k}{m}\sigma_m^{m/2}\leq \frac{1}{m}\lim_{n\to\infty}\sum_{i=1}^\ell\|u_{n,i}\|_g^2 =\lim_{n\to\infty} \hat c_n \leq c_\ell^*.\nonumber
\end{align}
But then assumption \eqref{eq:assumption} implies that $k=\ell$, i.e., $Z=\emptyset$ and claim \eqref{eq:nontrivial} is proved. Moreover, \eqref{eq:convergence} becomes  
$$c_\ell^*\leq \frac{1}{m}\sum_{i=1}^\ell\|t_iu_{\infty,i}\|_g^2\leq \frac{1}{m}\sum_{i=1}^\ell\|u_{\infty,i}\|_g^2\leq \frac{1}{m}\lim_{n\to\infty}\sum_{i=1}^\ell\|u_{n,i}\|_g^2 \leq c_\ell^*.$$
Hence, $t_i=1$ and, so, $(u_{\infty,1},\ldots,u_{\infty,\ell})\in \cM_\ell$, and
\begin{equation}\label{eq:minimizer}
\frac{1}{m}\sum_{i=1}^\ell\|u_{\infty,i}\|_g^2 =\lim_{n\to\infty} \frac{1}{m}\sum_{i=1}^\ell\|u_{n,i}\|_g^2 = c_\ell^*.
\end{equation}
Consequently, $u_{n,i} \to u_{\infty,i}$ strongly in $H^1_g(M)$. Finally, since 
\begin{align*}
\sum_{i=1}^\ell\|u_{\infty,i}\|_g^2 &=\sum_{i=1}^\ell |u_{\infty,i}|_{g,2^*}^{2^*},\\
\sum_{i=1}^\ell\|u_{n,i}\|_g^2 &=\sum_{i=1}^\ell |u_{n,i}|_{g,2^*}^{2^*}+\mathop{\sum_{i,j=1}^\ell}_{j\neq i} \int_M \lambda_{n} \beta |u_{n,j}|^{\alpha_{ij}}|u_{n,i}|^{\beta_{ij}},
\end{align*}
and $u_{n,i}\to u_{\infty,i}$ strongly in $H^1_g(M)$ and $L^{2^*}_g(M)$, we obtain \eqref{eq:segregation}.
\end{proof}

\begin{lemma} \label{lem:L_infty}
Let $\lambda_n<0$ and $(u_{n,1},\ldots,u_{n,\ell})$ be a solution to the system \eqref{eq:s_0} such that $u_{n,i}\geq0$ and $u_{n,i}\to u_{\infty,i}$ strongly in $H_g^1(M)$ as $n\to\infty$. Then $(u_{n,i})$ is uniformly bounded in $L^\infty(M)$ for all $i=1,\ldots,\ell$. 
\end{lemma}

\begin{proof}
We write again $\cJ_n$ for the functional associated to the system \eqref{eq:s_0}. Note that, by Proposition \ref{prop:regularity}, $u_{n,i}\in L^\infty(M)$ for all $n\in\n$, $i=1,\ldots,\ell$. Fix $i\in\{1,\ldots,\ell\}$.

Let $s\geq0$ and set $w_{n,i}:=u_{n,i}^{1+s}$. Since $\partial_i\cJ_n(u_n)[u_{n,i}^{1+2s}]=0$ and $\lambda_{ij,n}<0$, we get that
\begin{align}\label{eq:L_infty1}
&\im|\nabla_gw_{n,i}|^2\dm=(1+s)^2\im u_{n,i}^{2s}|\nabla_g u_{n,i}|^2\dm \\
&\leq(1+s)^2\im\langle\nabla_g u_{n,i},\nabla_g u_{n,i}^{1+2s}\rangle\dm \nonumber\\
&\leq(1+s)^2\im|u_{n,i}|^{2^*-2}w_{n,i}^2\dm  - (1+s)^2\im \kappa_mS_gw_{n,i}^2\dm.\nonumber
\end{align}
Now, for any $K>0$, we have that
\begin{align}\label{eq:L_infty2}
&\im|u_{n,i}|^{2^*-2}w_{n,i}^2\dm\leq K^{2^*-2}\im w_{n,i}^2\dm \\
&\quad+ \int_{|u_{\infty,i}|\geq K}|u_{\infty,i}|^{2^*-2}w_{n,i}^2\dm + \im\left(|u_{n,i}|^{2^*-2}-|u_{\infty,i}|^{2^*-2}\right)w_{n,i}^2\dm \nonumber\\
&\leq K^{2^*-2}|w_{n,i}|_{g,2}^2 + \eta(K,n)|w_{n,i}|_{g,2^*}^2, \nonumber
\end{align}
where
$$\eta(K,n):=\left[\int_{|u_{\infty,i}|\geq K}|u_{\infty,i}|^{2^*}\dm\right]^\frac{2^*-2}{2^*}+\left||u_{n,i}|^{2^*-2}-|u_{\infty,i}|^{2^*-2}\right|_{g,\frac{2^*}{2^*-2}}.$$
Since $u_{n,i}\to u_{\infty,i}$ in $H_g^1(M)$, we have that $|u_{n,i}|^{2^*-2}\to|u_{\infty,i}|^{2^*-2}$ in $L_g^\frac{2^*}{2^*-2}(M)$. Fix $\eps>0$, and choose $K_s,n_s$ such that $\frac{1+\epsilon}{\sigma_m}(1+s)^2\,\eta(K_s,n)<\frac{1}{2}$ for every $n\geq n_s$. From Theorem \ref{thm:aubin} and inequalities \eqref{eq:L_infty1} and \eqref{eq:L_infty2} we obtain
\begin{align*}
|w_{n,i}|_{g,2^*}^2 &\leq \frac{1+\eps}{\sigma_m}\im|\nabla_g w_{n,i}|^2\dm + C|w_{n,i}|_{g,2}^2\\
&\leq \frac{1+\eps}{\sigma_m}(1+s)^2\,\eta(K_s,n)|w_{n,i}|_{g,2^*}^2+C_s|w_{n,i}|_{g,2}^2\\
&\leq \frac{1}{2}|w_{n,i}|_{g,2^*}^2+C_s|w_{n,i}|_{g,2}^2\qquad\forall n\geq n_s.
\end{align*}
Therefore,
$$
|u_{n,i}|_{g,2^*(1+s)}^{2(1+s)}=|w_{n,i}|_{g,2^*}^2\leq \tilde C_s|w_{n,i}|_{g,2}^2=\tilde C_s|u_{n,i}|_{g,2(1+s)}^{2(1+s)}\quad\forall n\in\n,
$$
whence
\[
|u_{n,i}|_{g,2^*(1+s)}\leq {\tilde{C}}'_s|u_{n,i}|_{g,2(1+s)}\quad\forall n\in\n,
\] 
where $C_s$, $\tilde C_s$ and $\tilde{C}'_s$ are positive constants depending on $s$ but not on $n$,
Iterating this inequality, starting with $s=0$, we conclude that, for any $r\in[2,\infty)$,
$$|u_{n,i}|_{g,r}^2\leq \bar C_r\quad\forall n\in\n,$$
where $\bar C_r$ is a positive constant independent of $n$. Now, we fix $2R>0$ smaller than the injectivity radius of $M$. Since $M$ is covered by a finite number of geodesic balls of radius $R$ and $u_{n,i}$ satisfies
$$\mathscr{L}_gu_{n,i}\leq |u_{n,i}|^{2^*-2}u_{n,i}\qquad\text{on \ }M,$$
we derive from \cite[Theorem 8.17]{gt} that $(u_{n,i})$ is uniformly bounded in $L^\infty(M)$, as claimed.
\end{proof}

\begin{lemma} \label{lem:Holder}
For $\lambda_{n}<0$ such that $\lambda_n\to -\infty$ let $(u_{n,1},\ldots,u_{n,\ell})$ be a solution to the system \eqref{eq:s_0} such that $u_{n,i}\geq0$ and $(u_{n,i})$ is uniformly bounded in $L^\infty(M)$ for each $i=1,\ldots,\ell$. Then, for any $\alpha\in(0,1)$ there exists $C_\alpha>0$ such that
\begin{equation*}
\|u_{n,i}\|_{\cC^{0,\alpha}(M)}\leq C_\alpha\quad\forall n\in\n,\; \forall i=1,\ldots,\ell.
\end{equation*}
\end{lemma}

\begin{proof}
This is a particular case of Theorem \ref{thm:mainappendix}.
\end{proof}

\begin{lemma} \label{lem:estimate2}
Assume that $(M,g)$ is not locally conformally flat, $m\geq 10$, and there exists $(u_1,\ldots,u_{\ell-1})\in\cM_{\ell-1}$ such that $u_i\in\cC^0(M)$, $u_i\geq 0$ and 
$$\frac{1}{m}\sum_{i=1}^{\ell-1}\|u_i\|_g^2=c_{\ell-1}^*.$$
If $m=10$, assume further that there exists $p\in M$ such that
\begin{equation} \label{eq:m=10}
0<u_1(p)<\frac{5}{567}\,|W_g(p)|_g^2.
\end{equation}
Then
\begin{align*}
c_\ell^* &<\min\left\{c_k^*+\frac{\ell-k}{m}\,\sigma_m^{m/2}:1\leq k<\ell\right\}.
\end{align*}
\end{lemma}

\begin{proof}
It suffices to show that
\begin{equation}\label{cc*}
c_\ell^* <c_{\ell-1}^*+\frac{1}{m}\,\sigma_m^{m/2}.
\end{equation}
Set $\Omega_i:=\{q\in M:u_i(q)>0\}$. Then, $\Omega_i$ is open and $\Omega_i\cap\Omega_j=\emptyset$ if $i\neq j$. Since $(M,g)$ is not locally conformally flat and $m\geq 4$, there exists $p\in M$ such that the Weyl tensor $W_g(p)$ at $p$ does not vanish. After reordering, we may assume that, either $p\in\Omega_1$, or $p\in M\smallsetminus \cup_{i=1}^{\ell-1}\overline\Omega_i$.

First, we consider the case where $p\in\Omega_1$. If $m=10$ we take $p$ satisfying \eqref{eq:m=10}. Fix $r>0$ suitably small so that the closed geodesic ball centered at $p$ is contained in $\Omega_1$ and let $\chi:[0,\infty)\to\r$ be a smooth cut-off function such that $0\le\chi\le1$, $\chi\equiv1$ in $[0,\frac{r}{2}]$ and $\chi\equiv0$ in $[r,\infty)$. Define $\tilde V_{\delta,p}$ on $M$ by
\begin{equation*}
\tilde V_{\delta,p}(x)=\chi(|x|)V_{\delta,p}(x)\quad\text{if }|x|\leq r,\qquad  \tilde V_{\delta,p}(x)=0\quad\text{otherwise},
\end{equation*}
written in conformal normal coordinates around $p$, where $V_{\delta,p}$ is the function in \eqref{eq:test}. If $(M,g)$ is not locally conformally flat and $m\geq 7$ the estimates \eqref{eq:3_epv} yield
\begin{align}\label{eq:4_epv}
\|\tilde V_{\delta,p}\|_g^2 &=\sigma_m^{m/2}+\frac{(m-2)^2}{m+2}\,\overline c_m\omega_{m-1}|W_g(p)|_g^2\,\delta^{4}+o(\delta^4),\\
|\tilde V_{\delta,p}|_{g,2^*}^{2^*} &=\sigma_m^{m/2}+ \frac{m^2}{m-4}\,\overline c_m\omega_{m-1}|W_g(p)|_g^2\,\delta^{4}+o(\delta^4),\nonumber
\end{align}
with $\bar c_m$ as in \eqref{eq:c_m}. Now, set
\begin{equation*}
v_1:=(u_1-\tilde V_{\delta,p})^+\qquad\text{and}\qquad v_\ell:=(u_1-\tilde V_{\delta,p})^-.
\end{equation*}
Note that $v_i\neq 0$ and $v_1v_\ell=0$ on $M$, and $v_1=0=v_\ell$ in $M\smallsetminus\Omega_1$. Let $s_i>0$ be such that $\|s_iv_i\|^2_g=|s_iv_i|^{2^*}_{g,2^*}$. Then, $(s_1v_1,u_2,\ldots,u_{\ell-1},s_\ell v_\ell)\in\cM_\ell$ and
\begin{equation} \label{tv}
\|s_iv_i\|^2_g=\left(\frac{\|v_i\|_g^2}{|v_i|_{g,2^*}^2}\right)^{m/2}\qquad\text{for \ }i=1,\ell.
\end{equation}
For $m\geq 10$ from Remark \ref{rem:nehari} and Lemma \ref{lem:appendixB} we derive
\begin{align} \label{eq:v1}
&\frac{1}{m}\left(\frac{\|v_1\|_g^2}{|v_1|_{g,2^*}^2}\right)^{m/2}=\frac{1}{2}\|s_1v_1\|^2_g-\frac{1}{2^*}|s_1v_1|^{2^*}_{g,2^*}\\
&\qquad=\frac{s_1^2}{2}\|u_1\|^2_g-\frac{s_1^{2^*}}{2^*} |u_1|^{2^*}_{g,2^*}-(s_1^2-s_1^{2^*})\mathfrak a_mu_1(p)\delta^\frac{m-2}{2}+o(\delta^4) \nonumber \\
&\qquad\leq \frac{1}{m}\|u_1\|^2_g+o(\delta^4), \nonumber
\end{align}
because $\|u_1\|^2_g=|u_1|^{2^*}_{g,2^*}$ and $s_1^{2^*-2}=\frac{\|v_1\|_g^2}{|v_1|_{g,2^*}^{2^*}}\to 1$ as $\delta\to 0$. Similarly, using \eqref{eq:4_epv}, we obtain
\begin{align*}
&\frac{1}{m}\left(\frac{\|v_\ell\|_g^2}{|v_\ell|_{g,2^*}^2}\right)^{m/2}=\frac{1}{2}\|s_\ell v_\ell\|^2_g-\frac{1}{2^*}|s_\ell v_\ell|^{2^*}_{g,2^*} \\
&\qquad\leq\frac{1}{m}\sigma_m^{m/2}+\left(\frac{s_\ell^2}{2}\frac{(m-2)^2}{m+2}-\frac{s_\ell^{2^*}}{2^*}\frac{m^2}{m-4}\right)\overline c_m\omega_{m-1}|W_g(p)|_g^2\,\delta^{4} \\
&\qquad\quad+\left(s_\ell^2(\mathfrak{a}_m+\mathfrak{b}_m)-s_\ell^{2^*}\mathfrak{b}_m\right)u_1(p)\delta^\frac{m-2}{2}+o(\delta^4).
\end{align*}
Since $s_\ell^{2^*-2}=\frac{\|v_\ell\|_g^2}{|v_\ell|_{g,2^*}^{2^*}}\to 1$ as $\delta\to 0$, and $\frac{1}{2}\frac{(m-2)^2}{m+2}<\frac{1}{2^*}\frac{m^2}{m-4}$ and $\frac{m-2}{2}>4$ when $m\geq 11$, we have that, for $\delta$ small enough,
\begin{equation}\label{eq:v_ell}
\frac{1}{m}\left(\frac{\|v_\ell\|_g^2}{|v_\ell|_{g,2^*}^2}\right)^{m/2}\leq\frac{1}{m}\sigma_m^{m/2}-C\delta^4+o(\delta^4)\qquad\text{if \ }m\geq 11,
\end{equation}
with $C>0$. On the other hand, if $m=10$, then $\frac{m-2}{2}=4$. Recalling that $\omega_m$ is the volume of the standard $m$-sphere $\sm$ and using \eqref{eq:c_m} we obtain
\begin{align*}
&\mathfrak{a}_mu_1(p)+\left(\frac{1}{2}\frac{(m-2)^2}{m+2}-\frac{1}{2^*}\frac{m^2}{m-4}\right)\overline c_m\omega_{m-1}|W_g(p)|_g^2\\
&=\mathfrak{a}_m\left[u_1(p)+\frac{1}{2}\left[\frac{m-2}{m+2}-\frac{m}{m-4}\right]\frac{1}{192}\,\frac{(m+2)\left[m(m-2)\right]^\frac{m-2}{4}}{2^{m-1}(m-6)(m-1)}\,\frac{\omega_m}{\omega_{m-1}}|W_g(p)|_g^2\right]\\
&=\mathfrak{a}_{10}\left[u_1(p)-\frac{5}{567}\,|W_g(p)|_g^2\right]<0
\end{align*}
by assumption \eqref{eq:m=10}. Hence, for $\delta$ small enough,
\begin{equation}\label{eq:v_ell2}
\frac{1}{m}\left(\frac{\|v_\ell\|_g^2}{|v_\ell|_{g,2^*}^2}\right)^{m/2}\leq\frac{1}{m}\sigma_m^{m/2}-C\delta^4+o(\delta^4)\qquad\text{if \ }m=10,
\end{equation}
with $C>0$. From \eqref{tv}, \eqref{eq:v1}, \eqref{eq:v_ell} and \eqref{eq:v_ell2} we derive
\begin{align*}
c_\ell^*&\le \frac1m\left(\|t_1v_1\|^2_g+\|u_2\|^2_g+\dots+\|u_{\ell-1}\|^2_g+\|t_\ell v_\ell\|^2_g\right) \\ 
& =\frac1m\left[\left(\frac{\|v_1\|_g^2}{|v_1|_{g,2^*}^2}\right)^{m/2}+\|u_2\|^2_g+\dots+\|u_{\ell-1}\|^2_g +\left(\frac{\|v_\ell\|_g^2}{|v_\ell|_{g,2^*}^2}\right)^{m/2}\right]\\
&\leq\frac1m\(\|u_1\|^2_g+\|u_2\|^2_g+\dots+\|u_{\ell-1}\|^2_g\right)+\frac{1}{m}\sigma_m^{m/2}-C\delta^4+o(\delta^4)\\
&< c_{\ell-1}^*+\frac 1m \sigma_m^{m/2}
\end{align*}
for $\delta$ small enough. This proves \eqref{cc*} when $p\in\Omega_1$.

If $p\in M\smallsetminus \cup_{i=1}^{\ell-1}\overline\Omega_i$, we fix $r>0$ small enough so that the closed geodesic ball of radius $r$ centered at $p$ is contained in $M\smallsetminus \cup_{i=1}^{\ell-1}\overline\Omega_i$ and define $u_\ell:=t_\ell \tilde V_{\delta,p}$ with $\tilde V_{\delta,p}$ as above and $t_\ell>0$ such that $\|u_\ell\|_g^2=|u_\ell|_{g,2^*}^{2^*}$. Then, $(u_1,\ldots,u_\ell)\in\cM_\ell$ and estimates \eqref{eq:4_epv} yield
$$c_\ell^*\le \frac1m\sum_{i=1}^\ell\|u_i\|^2_g< c_{\ell-1}^*+\frac 1m \sigma_m^{m/2},$$
as claimed.
\end{proof}

\begin{remark} \label{rem:low_dimensions}
\emph{
The argument given above does not carry over to $m<10$ or to the case where $(M,g)$ is locally conformally flat. Indeed, as can be seen from identities \eqref{eq:1_epv}, \eqref{eq:2_epv} and \eqref{eq:3_epv} and Lemma \ref{lem:appendixB}, in these cases
\begin{align*}
\frac{1}{m}\left(\frac{\|v_1\|_g^2}{|v_1|_{g,2^*}^2}\right)^{m/2}&\leq \frac{1}{m}\|u_1\|^2_g+o(\delta^\frac{m-2}{2}),\\
\frac{1}{m}\left(\frac{\|v_\ell\|_g^2}{|v_\ell|_{g,2^*}^2}\right)^{m/2}&\leq\frac{1}{m}\sigma_m^{m/2}+C u_1(p)\delta^\frac{m-2}{2}+o(\delta^\frac{m-2}{2}),
\end{align*}
with $C u_1(p)>0$, for $\delta$ small enough.}
\end{remark}

\begin{remark} \label{rem:m=10}
\emph{
If $m=10$, then the following geometric conditions suffice to guarantee \eqref{eq:m=10}:
\begin{itemize}
\item For $\ell\geq 3$, inequality \eqref{eq:m=10} holds true if $|W_g(q)|_g\neq 0$ for every $q\in M$, because for $p\in\Omega_1:=\{q\in M:u_1(q)>0\}$ close enough to $\partial\Omega_1$ one has that $u_1(p)<\frac{5}{567}\min_{q\in M}|W_g(q)|_g^2$.
\item For $\ell=2$, inequality \eqref{eq:m=10} holds true if 
$$|S_g(q)|^2<\frac{5}{28}\,|W_g(q)|^2_g\qquad\forall q\in M.$$
Indeed, choosing $p$ to be a minimum of $u_1$, since $u_1$ is a positive solution to the Yamabe equation \eqref{eq:y} we have that $\kappa_mS_g(p)u_1=u_1^{2^*-1}+\Delta_gu_1\geq u_1^{2^*-1}$. Setting $m=10$ we get $u_1(p)\leq\frac{4}{81}|S_g(p)|^2<\frac{5}{567}\,|W_g(p)|^2_g$.
\end{itemize}}
\end{remark}
\smallskip

\begin{lemma} \label{lem:weak_partition2}
Assume that $(M,g)$ satisfies the following conditions:
\begin{itemize}
\item[$(A4)$] $(M,g)$ is not locally conformally flat and $\dim M\geq 10$. If $\dim M=10$, then there exist a positive least energy fully nontrivial solution $\bar u$ to the Yamabe equation \eqref{eq:y} and a point $p\in M$ such that $\bar u(p)<\frac{5}{567}\,|W_g(p)|^2_g$ and, in addition, $|W_g(q)|_g\neq 0$ for every $q\in M$ if $\ell\geq 3$.
\end{itemize}
Let $\lambda_{n}<0$ and $u_n=(u_{n,1},\ldots,u_{n,\ell})$ be a least energy fully nontrivial solution to the system \eqref{eq:s_0}. Assume that $\lambda_{n}\to -\infty$ as $n\to\infty$ and that $u_{n,i}\geq 0$ for all $n\in\n$.
Then there exists $(u_{\infty,1},\ldots,u_{\infty,\ell})\in\cM_\ell$ with $u_{\infty,i}\in\cC^{0,\alpha}(M)$ for every $\alpha\in (0,1)$ such that, after passing to a subsequence, $u_{n,i}\to u_{\infty,i}$ strongly in $H_g^1(M)\cap\cC^{0,\alpha}(M)$,\, $u_{\infty,i}\geq 0$, and
$$c_\ell^*=\frac{1}{m}\sum_{i=1}^\ell\|u_{\infty,i}\|_g^2.$$
Moreover,
\[
\int_M \lambda_n u_{i,n}^\beta u_{j,n}^\beta\to 0 \text{ as } n\to \infty\qquad \text{whenever \ } i\neq j.
\]
\end{lemma}

\begin{proof}
The proof is by induction on $\ell$. 

Let $\ell=2$. Then $u_1:=\bar u$ satisfies the hypotheses of Lemma \ref{lem:estimate2}. Therefore, the inequality \eqref{eq:assumption} holds true and Lemma \ref{lem:weak_partition} yields the existence of $(u_{\infty,1},\ldots,u_{\infty,\ell})\in\cM_\ell$ such that, after passing to a subsequence, $u_{n,i}\to u_{\infty,i}$ strongly in $H_g^1(M)$,\, $u_{\infty,i}\geq 0$, and
$$c_\ell^*=\frac{1}{m}\sum_{i=1}^\ell\|u_{\infty,i}\|_g^2.$$
From Lemmas \ref{lem:L_infty} and \ref{lem:Holder} we get that $(u_{n,i})$ is uniformly bounded in $\cC^{0,\alpha}(M)$. Therefore, the family $\{u_{n,i}\}$ is equicontinuous and, as $u_{n,i} \to u_{\infty,i}$ a.e. on $M$, the Arzel\`a-Ascoli theorem yields $u_{n,i}\to u_{\infty,i}$ in $\cC^0(M)$. 

Now, let $\ell\geq 3$ and assume that the statement holds true for $\ell-1$. Then, by Remark \ref{rem:m=10}, the hypotheses of Lemma \ref{lem:estimate2} are satisfied and, consequently, \eqref{eq:assumption} holds true for $\ell$. The same argument we gave for $\ell=2$ yields the result for $\ell$.
\end{proof}

\begin{remark}
\emph{
Observe that to prove the previous lemma for $\ell$, we need it to be true for $\ell-1$, because the inequality \eqref{eq:assumption} must hold true in order to apply Lemma \ref{lem:weak_partition}. Therefore, the inequality $\bar u(p)<\frac{5}{567}\,|W_g(p)|^2_g$ is required for every $\ell\geq 2$.}
\end{remark}
\medskip

\begin{proof}[Proof of Theorem \ref{thm:phase_separation}]
As pointed out in Remark \ref{rem:m=10}, assumption $(A3)$ implies $(A4)$. Statements $(i)$ and $(iii)$ follow immediately from Lemmas \ref{lem:weak_partition2} and \ref{lem:cont->op}.                    
   
Proofs of $(ii)$ and $(iv):$ These statements have a local nature. In local coordinates the system \eqref{eq:s_0} becomes
\[
-\div(A(x)\nabla u_i)=f_i(x,u_i) + a(x) \sum_{\substack{j=1 \\ j\neq i}}^\ell \lambda_n |u_{n,j}|^\beta |u_{n,i}|^{\beta-2}u_{n,i},\qquad x\in\Omega,
\]
where $\Omega$ is an open bounded subset of $\rm$, $a(x)=\sqrt{|g(x)|}$, $A(x)=\sqrt{|g(x)|}(g^{kl}(x))$, and $f_i(x,s):=a(x)(|s|^{2^*-2}s-\kappa_mS_g(x)s)$. As usual, $(g_{kl})$ is the metric $g$ in local coordinates, $(g^{kl})$ is its inverse and $|g|$ its determinant. This system satisfies assumptions $(H1')$, $(H2)$ and $(H3)$ of  Theorem \ref{thm:generaltheorem_Lip_Reg} in Appendix \ref{app:generaltheorem_Lip_Reg}. Statements $(i)$ and $(iii)$, which are already proved, yield assumptions $(H4)$, $(H5)$ and $(H6)$. From Theorem \ref{thm:generaltheorem_Lip_Reg} we obtain that $(ii)$ and $(iv)$ hold true locally on $M$, hence also globally.

Proof of $(v)$: If $u\in H_g^1(M)$ is a sign-changing solution of the Yamabe equation \eqref{eq:y}, then $u^+:=\max\{u,0\}\neq 0$, $u^-:=\min\{u,0\}\neq 0$ and $J_M'(u)[u^\pm]=0$. Hence, $u$ belongs to the set
$$\cE_M:=\{u\in\cN_M:u^+\in\cN_M\text{ and }u^-\in\cN_M\}.$$
Moreover, as shown in \cite[Lemma 2.6]{ccn}, any minimizer of $J_M$ on $\cE_M$ is a sign-changing solution of \eqref{eq:y}. For every $u\in\cE_M$, we have that $(u^+,u^-)\in\cM_2$ and $J_M(u)=\frac{1}{m}(\|u^+\|_g^2+\|u^-\|_g^2)$. Therefore,
$$\inf_{\cE_M}J_M\geq c_2^*=\frac{1}{m}(\|u_{\infty,1}\|_g^2+\|u_{\infty,2}\|_g^2).$$
As $u_{\infty,1}-u_{\infty,2}\in\cE_M$, it is a minimizer of $J_M$ on $\cE_M$. Hence, it is a sign-changing solution of \eqref{eq:y}, as claimed.
\end{proof}

\begin{remark} \label{rem:A4}
\emph{
As can be seen from its proof, Theorem \ref{thm:phase_separation} is true under assumption $(A4)$ and, consequently, so are Corollaries \ref{cor:op} and \ref{cor:nodal}. As noted in Remark \ref{rem:m=10}, $(A4)$ is weaker that $(A3)$, but it requires some knowledge on the least energy solution to the Yamabe equation \eqref{eq:y} having precisely two nodal domains.}
\end{remark}

\begin{remark} \label{rem:ah}
\emph{
In \cite{ah}, Ammann and Humbert defined the second Yamabe invariant of $(M,g)$ as
$$\mu_2(M,g):=\inf_{\tilde{g}\in[g]}\lambda_2(\tilde{g})\mathrm{Vol}(M,\tilde{g})^{2/m},$$
where $\lambda_2(\tilde{g})$ is the second eigenvalue of the operator $\kappa_m^{-1}\mathscr{L}_{\tilde{g}}$ and $[g]$ is the conformal class of $g$. Using the variational characterization in \cite[Proposition 2.1]{ah} one can easily verify that
$$\inf_{\cE_M}J_M=\frac{1}{m}(\kappa_m\,\mu_2(M,g))^{m/2}.$$
The invariant $\mu_2(M,g)$ is not attained at a metric, but it is shown in \cite{ah} that, if $(M,g)$ is not locally conformally flat and $m\geq 11$, this invariant is attained at the \emph{generalized metric} conformal to $g$ which is given by a minimizer of $J_M$ in $\cE_M$. So Corollary \ref{cor:nodal} recovers and extends this result.}
\end{remark}

\begin{remark} \label{rem:rob-vet}
\emph{
It is interesting to compare our result with that proved by Robert and Vétois in \cite{rv} under assumptions which are complementary to ours. In fact, they establish the existence of a sign-changing solution to the subcritical perturbation of Yamabe equation
$$-\Delta_g u + \kappa_mS_gu = |u|^{2^*-2-\epsilon}u\qquad\text{on }M,$$ 
which looks like the difference between a positive solution $u_0$ to the Yamabe equation and a bubble. Their result holds true either in the locally conformally flat case, or in low dimensions $3\le m\le 9$, or if $m=10$ provided $u_0(p)> \frac{5}{567}\,|W_g(p)|^2_g$ for any $p\in M.$}

\emph{An interesting open problem would be to show  that under these assumptions a least energy sign-changing solution to the Yamabe problem \eqref{eq:y} does not exist, as suggested by Remark \ref{rem:low_dimensions}.}
\end{remark}

\appendix

\section{Some estimates}
\label{app:B}

Fix $p\in M$ and $r>0$ suitably small. Let $\chi:[0,\infty)\to\r$ be a smooth cut-off function such that $0\le\chi\le1$, $\chi\equiv1$ in $[0,\frac{r}{2}]$ and $\chi\equiv0$ in $[r,\infty)$, and let $\tilde V_{\delta,p}$ be the function on $M$ given by
\begin{equation}\label{wc}
\tilde V_{\delta,p}(x)=\chi(|x|)V_{\delta,p}(x)\quad\text{if }|x|\leq r,\qquad  \tilde V_{\delta,p}(x)=0\quad\text{otherwise},
\end{equation}
in conformal normal coordinates at $p$, where $V_{\delta,p}$ is the function defined in \eqref{eq:test}. Then, for some positive constant $\mathfrak c_0$,
\begin{equation}\label{w}
0< \tilde V_{\delta,p}(x) \le \mathfrak c_0\left(\frac{\delta}{\delta^2+|x|^2}\right)^{(m-2)/2}\qquad\text{if}\ |x|\le r.
\end{equation}
Let $u\in H^1_g(M)\cap\cC^0(M)$ be such that $u\geq 0$ and $u(x)>0$ if $|x|\leq r$. Then, there are positive constants $\mathfrak c_1,\mathfrak c_2$ such that
\begin{equation}\label{u1}
0<\mathfrak c_1\le u(x)\le \mathfrak c_2\qquad\text{if \ }|x|\leq r.
\end{equation}   
Set
\begin{equation*}
v_1:=(u-\tilde V_{\delta,p})^+\qquad\text{and}\qquad v_\ell:=(u-\tilde V_{\delta,p})^-.
\end{equation*}
Observe that
\begin{equation} \label{Bv1}
v_1  =
\begin{cases}
u_1 &\text{if \ }|x|\geq r, \\
0 &\text{if \ }|x|\le r\text{ \ and \ }u_1\le \tilde V_{\delta,p},\\
u_1-\tilde V_{\delta,p} &\text{if \ }|x|\le r\text{ \ and \ }u_1\ge \tilde V_{\delta,p},\\
\end{cases}
\end{equation}
and
\begin{equation}  \label{Bvell}
v_\ell  =
\begin{cases}
0 &\text{if \ }|x|\geq r, \\
\tilde V_{\delta,p} -u_1&\text{if \ }|x|\le r\text{ \ and \ }u_1\le \tilde V_{\delta,p},\\
0 &\text{if \ }|x|\le r\text{ \ and \ }u_1\ge \tilde V_{\delta,p}.\\
\end{cases}
\end{equation}
By \eqref{wc}, \eqref{w} and \eqref{u1}, there are positive constants $c_1,c_2,c_3$ such that
\begin{equation}\label{crucial}
\begin{cases}
|x|\le r\text{ \ and \ } u_1(x)\le \tilde V_{\delta,p}(x)\ \Rightarrow\ |x|\le c_1\sqrt\delta,\\
|x|\le\tfrac{r}{2}\text{ \ and \ } u_1(x)\ge \tilde V_{\delta,p}(x)\ \Rightarrow\ |x|\ge c_2\sqrt\delta,\\
|x|\le\tfrac{r}{2} \ \Rightarrow \ \tilde V_{\delta,p}(x)=V_{\delta,p}(x),\\
\tfrac{r}{2}\le |x|\le r\  \Rightarrow\ |\tilde V_{\delta,p}(x)|,\ |\nabla \tilde V_{\delta,p}(x)| \le c_3\delta^{(m-2)/2}.
\end{cases}
\end{equation}

\begin{lemma} \label{lem:appendixB}
We have the following estimates:
\begin{itemize}
\item[$(i)$] $\|v_1\|_g^2=\|u\|_g^2-2\mathfrak{a}_mu(p)\delta^\frac{m-2}{2}+o(\delta^{\nu(m)})$,
\item[$(ii)$] $|v_1|_{g,2^*}^{2^*}=|u|_{g,2^*}^{2^*}-2^*\mathfrak{a}_mu(p)\delta^\frac{m-2}{2}+o(\delta^\frac{m-2}{2})$,
\item[$(iii)$] $\|v_\ell\|_g^2=\|\tilde V_{\delta,p}\|_g^2+2(\mathfrak{a}_m+\mathfrak{b}_m)u(p)\delta^\frac{m-2}{2}+o(\delta^{\nu(m)})$,
\item[$(iv)$] $|v_\ell|_{g,2^*}^{2^*}=|\tilde V_{\delta,p}|_{g,2^*}^{2^*}+2^*\mathfrak{b}_mu(p)\delta^\frac{m-2}{2}+o(\delta^\frac{m-2}{2})$,
\end{itemize}
where
\begin{equation*}
\nu(m)=
\begin{cases}
\tfrac{m-2}{2} &\text{if either }m\leq 6\text{ or }(M,g)\text{ is l.c.f.,} \\
4 &\text{if }m\geq 7\text{ and }(M,g)\text{ is not l.c.f.,}
\end{cases}
\end{equation*}
$$\mathfrak{a}_m:=(m-2)(m(m-2))^\frac{m-2}{4}\omega_{m-1}\qquad\text{and}\qquad\mathfrak{b}_m:=\irm U^{2^*-1}.$$
\end{lemma}

\begin{proof}
$(i):$ From \eqref{Bv1} and \eqref{crucial} we obtain
\begin{align*}
\|v_1\|_g^2-\|u\|_g^2
&=   \int\limits_{\{|x|\le r\}\cap \{u \ge \tilde V_{\delta,p}\}}\left[(|\nabla_g (u -\tilde V_{\delta,p})|^2 +\kappa_m S_g (u -\tilde V_{\delta,p})^2)- \left(|\nabla_g u |^2 +\kappa_m S_g u ^2\right)\right]\dm \\ 
&\qquad -\underbrace{\int\limits_{\{|x|\le r\}\cap \{u \le \tilde V_{\delta,p}\}} \left(|\nabla_g u |^2 +\kappa_m S_g u ^2\right)\dm}_{=O(\delta^\frac{m}{2})}\\ 
&=\int\limits_{\{|x|\le r\}\cap \{u \ge \tilde V_{\delta,p}\}}\left(|\nabla_g \tilde V_{\delta,p}|^2 +\kappa_m S_g  \tilde V_{\delta,p}^2\right)\dm\\
&\qquad-2 \int\limits_{\{|x|\le r\}\cap \{u \ge \tilde V_{\delta,p}\}}\left(\langle\nabla_g u ,\nabla_g \tilde V_{\delta,p}\rangle_g+\kappa_mS_g u  \tilde V_{\delta,p} \right) \dm + O\left(\delta^\frac{m}{2}\right)\\
& =\underbrace{O\left(\int\limits_{\{c_2\sqrt\delta\le |x|\le r\} }\left(|\nabla_g \tilde V_{\delta,p}|^2 +\kappa_m S_g  \tilde V_{\delta,p}^2\right)\dm\right)}_{= O(\delta^\frac{m}{2})}\\
&\qquad -2  \underbrace{ \int\limits_{\{|x|\le r\}\cap \{u \ge \tilde V_{\delta,p}\}}\left(\langle\nabla_g u ,\nabla_g \tilde V_{\delta,p}\rangle+\kappa_mS_g u  \tilde V_{\delta,p} \right) \dm}_{\text{see }\eqref{ex5}}   + O\left(\delta^\frac{m}{2}\right),
\end{align*}
and, using \eqref{crucial} again,
\begin{align}\label{ex5}
&\int\limits_{\{|x|\le r\}\cap \{u \ge \tilde V_{\delta,p}\}}\left(\langle\nabla_g u ,\nabla_g \tilde V_{\delta,p}\rangle+\kappa_mS_g u  \tilde V_{\delta,p} \right) \dm\\
&=\int\limits_{\{2c_1\sqrt\delta\le |x|\le r\}}\left(\langle\nabla_g u ,\nabla_g \tilde V_{\delta,p}\rangle +\kappa_mS_g u  \tilde V_{\delta,p} \right)\dm \nonumber \\
&=\int\limits_{\{2c_1\sqrt\delta\le |x|\le r\}}\left(\langle\nabla_g u ,\nabla_g \tilde V_{\delta,p}\rangle+\kappa_mS_g u  \tilde V_{\delta,p} \right) \dm+ O\left(\delta^\frac{m}{2}\right) \nonumber\\ 
&=\underbrace{\int\limits_{\{2c_1\sqrt\delta\le |x|\le r\}}\left(-\Delta _g \tilde V_{\delta,p}+\kappa_mS_g \tilde V_{\delta,p}-\tilde V_{\delta,p}^{2^*-1}\right)u \dm}_{\text{see }\eqref{ex1}} + O\left(\delta^\frac{m}{2}\right) \nonumber \\ 
&\qquad+ \underbrace{\int\limits_{\{2c_1\sqrt\delta\le |x|\le r\}} \tilde V_{\delta,p}^{2^*-1} u \dm}_{= O(\delta^\frac{m}{2})}+ \underbrace{\int\limits_{ \{2c_1\sqrt\delta=|x| \}}\partial_{\nu} \tilde V_{\delta,p}u }_{\text{see }\eqref{ex0}}+\underbrace{\int\limits_{ \{r=|x| \}}\partial_{\nu} \tilde V_{\delta,p}u }_{=0}  \nonumber
\end{align}
where $\partial_\nu$ is the exterior normal derivative,
\begin{equation}\label{ex0}
\int\limits_{ \{2c_1\sqrt\delta=|x| \}}\partial_{\nu} \tilde V_{\delta,p}u=\underbrace{(m-2)\left(m(m-2)\right)^\frac{m-2}{4}\omega_{m-1}}_{=\,\mathfrak a_m}u(p)\delta^\frac{m-2}{2}+o(\delta^\frac{m-2}{2})
\end{equation}
and
\begin{align}\label{ex1}
&\left(\int\limits_{\{2c_1\sqrt\delta\le |x|\le r\}}|-\Delta _g \tilde V_{\delta,p}+\kappa_mS_g  \tilde V_{\delta,p}-\tilde V_{\delta,p}^{2^*-1}|^\frac{2m}{m+2} \dm\right)^\frac{m+2}{2m}\\
&\qquad\qquad=\begin{cases}
O\left(\delta^\frac{m-1}{2}\right)& \text{if }m=4,5, \text{ or }M \text{ is l.c.f.},\\
O\left(\delta^4|\ln\delta|^\frac{2}{3}\right)& \text{if }m=6\text{ and }M \text{ is not l.c.f.},\\
O\left(\delta^\frac{m+10}{4}\right)& \text{if }m\ge7\text{ and }M \text{ is not l.c.f.}\\
\end{cases} \nonumber
\end{align}
Indeed, arguing as in \cite{epv} we obtain
\begin{align*}
&\int\limits_{\{2c_1\sqrt\delta\le |x|\le r\}}|-\Delta _g \tilde V_{\delta,p}+\kappa_mS_g  \tilde V_{\delta,p}-\tilde V_{\delta,p}^{2^*-1}|^\frac{2m}{m+2} \ \dm\\
&=
\begin{cases}
O\left({\displaystyle \int_{2c_1\sqrt\delta}^r } \ \frac{\delta^m}{(\delta^2+s^2)^\frac{m^2}{m+2}} s^{\frac{2m^2}{m+2}-1+\frac{m(m-6)}{m+2}}\d s\right)& \text{if }m=4,5, \text{ or }M \text{ is l.c.f.},\\
O\left({\displaystyle \int_{2c_1\sqrt\delta}^r } \ \frac{\delta^m}{(\delta^2+s^2)^\frac{m^2}{m+2}}s^{8}|\ln s|\d s\right)& \text{if }m=6\text{ and }M \text{ is not l.c.f.},\\
O\left({\displaystyle \int_{ 2c_1\sqrt\delta}^r } \ \frac{\delta^m}{(\delta^2+s^2)^\frac{m^2}{m+2}}s^{\frac{2m^2}{m+2}-1+\frac{m(m-6)}{m+2}}\d s\right)& \text{if }m\ge7\text{ and }M \text{ is not l.c.f.}
\end{cases} \\
&=
\begin{cases}
O\left(\delta^\frac{2m(m-2)}{m+2} {\displaystyle \int_\frac{1}{\sqrt\delta}^\infty } \ s^{-1+\frac{m(m-6)}{m+2}}\d s\right)& \text{if }m=4,5, \text{ or }M \text{ is l.c.f.},\\
O\left(\delta^6{\displaystyle \int_\frac{2c_1}{\sqrt\delta}^\frac{r}{\sqrt \delta}} \ \frac{|\ln \delta s|}{(1+s^2)^\frac{9}{2}}s^{8}\d s\right)& \text{if }m=6\text{ and }M \text{ is not l.c.f.},\\
O\left(\delta^\frac{8m}{m+2} {\displaystyle \int_\frac{ 1}{\sqrt\delta}^\infty } \ s^{ -1-\frac{m(m-6)}{m+2}}\d s\right)& \text{if }m\ge7\text{ and }M \text{ is not l.c.f.}
\end{cases} \\
&=
\begin{cases}
O\left(\delta^{\frac{2m(m-2)}{m+2}-\frac{m(m-6)}{2(m+2)}}\right)& \text{if }m=4,5, \text{ or }M \text{ is l.c.f.},\\
O\left(\delta^8|\ln\delta|\right)& \text{if }m=6\text{ and }M \text{ is not l.c.f.},\\
O\left(\delta^{\frac{8m}{m+2}+\frac{m(m-6)}{2(m+2)}}\right)& \text{if }m\ge7\text{ and }M \text{ is not l.c.f.}
\end{cases}
\end{align*}
This concludes the proof of statement $(i)$.

$(ii):$ Using the inequalities
$$||a+b|^{2^*}-|a|^{2^*}|\le c(|a|^{2^*-1}|b|+|b|^{2^*})\qquad\forall a,b\in\r,$$
$$||a+b|^{2^*}-|a|^{2^*}-2^*a|a|^{2^*-2}|b||\le c(|a|^{2^*-2}|b|^2+|b|^{2^*})\qquad\forall a,b\in\r,$$
we obtain
\begin{align*}
|v_1|_{g,2^*}^{2^*}-|u|_{g,2^*}^{2^*}
&=\int\limits_{\{|x|\ge r\}}|u|^{2^*}\dm + \int\limits_{\{|x|\le r\}}|(u-\tilde V_{\delta,p})^+|^{2^*}\dm-\im|u|^{2^*}\dm\\
&= \int\limits_{\{|x|\le r\}\cap \{u \ge \tilde V_{\delta,p}\}}\left(|u -\tilde V_{\delta,p}|^{2^*}-|u |^{2^*}+2^*u^{2^*-1}\tilde V_{\delta,p} \right)\dm \\
&\qquad-2^*\int\limits_{\{|x|\le r\}\cap \{u \ge \tilde V_{\delta,p}\}} u^{2^*-1}\tilde V_{\delta,p}\dm -\int\limits_{\{|x|\le r\}\cap \{u \le \tilde V_{\delta,p}\}}|u |^{2^*}\dm\\
&=\underbrace{O\left(\int\limits_{\{c_2\sqrt\delta\le |x|\le r\}}\left(u^{2^*-2}\tilde V_{\delta,p}^2 +\tilde V_{\delta,p}^{2^*}\right)\dm\right)}_{O(\delta^\frac{m}{2})\ \text{ if }m\ge5} -2^*\underbrace{\int\limits_{\{2c_1\sqrt\delta\le |x|\le r\}}u^{2^*-1}\tilde V_{\delta,p}\dm}_{\text{ see }\eqref{ex6}}\\  
&\qquad +\underbrace{\int\limits_{\{2c_1\sqrt\delta\le |x|\le r\}\cap \{u \le \tilde V_{\delta,p}\}}u^{2^*-1}\tilde V_{\delta,p}  \dm}_{=0\ \text{ see }\eqref{crucial}} \ -
\underbrace{\int\limits_{\{  |x|\le 2c_1\sqrt\delta\}\cap \{u \ge \tilde V_{\delta,p}\}}u^{2^*-1}\tilde V_{\delta,p}  \dm}_{=O(\delta^\frac{m}{2})}\\
&\qquad+\underbrace{O\left(\int\limits_{\{|x|\le c_1\sqrt\delta\}}u^{2^*} \dm\right)}_{=O(\delta^\frac{m}{2})},
\end{align*}
where
\begin{align}\label{ex6}
&\int\limits_{\{2c_1\sqrt\delta\le |x|\le r\}}u^{2^*-1}\tilde V_{\delta,p}\dm = \int\limits_{\{2c_1\sqrt\delta\le |x|\le r\}}\left(-\Delta _gu+\kappa_mS_g u\right) \tilde V_{\delta,p} \dm \\ 
&=  \int\limits_{\{2c_1\sqrt\delta\le |x|\le r\}}\left(-\Delta _g\tilde V_{\delta,p}+\kappa_mS_g \tilde V_{\delta,p}\right) u \dm  -\underbrace{\int\limits_{ \{2c_1\sqrt\delta=|x| \}} \tilde V_{\delta,p}\partial_{\nu}u }_{=O(\delta^\frac{m-1}{2})}\nonumber \\ 
&\qquad -\underbrace{\int\limits_{ \{r=|x| \}}  \tilde V_{\delta,p}\partial_{\nu}u }_{=0} + \underbrace{\int\limits_{ \{2c_1\sqrt\delta=|x| \}}\partial_{\nu} \tilde V_{\delta,p}u }_{\text{see } \eqref{ex0}}+\underbrace{\int\limits_{ \{r=|x| \}}\partial_{\nu} \tilde V_{\delta,p}u }_{=0}. \nonumber
\end{align}
This concludes the proof of statement $(ii)$.

$(iii):$ Using \eqref{Bvell} and \eqref{crucial} we obtain
\begin{align*}
\|v_\ell\|_g^2-\|\tilde V_{\delta,p}\|_g^2 
&= \int\limits_{\{|x|\le r\}\cap \{u \le \tilde V_{\delta,p}\}}\left[|\nabla_g (u -\tilde V_{\delta,p})|^2 +\kappa_m S_g (u -\tilde V_{\delta,p})^2- |\nabla_g \tilde V_{\delta,p}|^2 +\kappa_m S_g \tilde V_{\delta,p}^2\right]\dm \\ 
& \qquad - \int\limits_{\{|x|\le r\}\cap \{u \ge \tilde V_{\delta,p}\}} \left(|\nabla_g \tilde V_{\delta,p}|^2 +\kappa_m S_g \tilde V_{\delta,p}^2\right)\dm \\ 
&=\int\limits_{\{|x|\le r\}\cap \{u \le \tilde V_{\delta,p}\}}\left(|\nabla_g u |^2 +\kappa_m S_g  u ^2\right)\dm\\
&\qquad-2\int\limits_{\{|x|\le r\}\cap \{u \le \tilde V_{\delta,p}\}}\left(\langle\nabla_g u ,\nabla_g \tilde V_{\delta,p}\rangle+\kappa_mS_g u  \tilde V_{\delta,p} \right) \dm \\ 
&\qquad+\underbrace{O\left(\int\limits_{\{c_2\sqrt\delta\le |x|\le r\} }\left(|\nabla_g \tilde V_{\delta,p}|^2 +\kappa_m S_g  \tilde V_{\delta,p}^2\right)\dm\right)}_{= O(\delta^\frac{m}{2})}\\
&=\underbrace{O\left(\int\limits_{\{|x|\le c_1\sqrt\delta \} }\left(|\nabla_g u |^2 +\kappa_m S_g  u ^2\right)\dm\right)}_{= O(\delta^\frac{m}{2})}\\
&\qquad -2\int\limits_{\{|x|\le r\}\cap \{u \le \tilde V_{\delta,p}\}}\left(\langle\nabla_g u ,\nabla_g \tilde V_{\delta,p}\rangle+\kappa_mS_g u  \tilde V_{\delta,p} \right) \dm+ O\left(\delta^\frac{m}{2}\right)\\
&=-2\underbrace{\int\limits_{\{|x|\le r\}\cap \{u \le \tilde V_{\delta,p}\}}\left(\langle\nabla_g u ,\nabla_g \tilde V_{\delta,p}\rangle+\kappa_mS_g u  \tilde V_{\delta,p} \right) \dm}_{\text{ see }\eqref{ex7}}    +O\left(\delta^\frac{m}{2}\right)
\end{align*}
and
\begin{align}\label{ex7}
& \int\limits_{\{|x|\le r\}\cap\{u \le \tilde V_{\delta,p}\}}(\langle\nabla_g u ,\nabla_g \tilde V_{\delta,p}\rangle+\kappa_mS_g u  \tilde V_{\delta,p}) \dm\\ 
&= \int\limits_{\{ |x|\le\frac {c_2}2\sqrt\delta\}}(\langle\nabla_g u ,\nabla_g \tilde V_{\delta,p}\rangle+\kappa_mS_g u  \tilde V_{\delta,p})\dm \nonumber \\
&\qquad+\underbrace{\int\limits_{\{ \frac {c_2}2\sqrt\delta\le |x|\le r\}\cap\{u \le \tilde V_{\delta,p}\}}(\langle\nabla_g u ,\nabla_g \tilde V_{\delta,p}\rangle+\kappa_mS_g u  \tilde V_{\delta,p})\dm}_{=O(\delta^\frac{m }{2})} \nonumber \\
&= \int\limits_{\{ |x|\le\frac {c_2}2\sqrt\delta\}}(\langle\nabla_g u ,\nabla_g \tilde V_{\delta,p}\rangle+\kappa_mS_g u  \tilde V_{\delta,p}) \dm+O\left(\delta^\frac{m}{2}\right) \nonumber \\
& =\underbrace{\int\limits_{\{ |x|\le\frac {c_2}2\sqrt\delta\}}(-\Delta _g \tilde V_{\delta,p}+\kappa_mS_g  \tilde V_{\delta,p}-\tilde V_{\delta,p}^{2^*-1})u \dm}_{\text{see }\eqref{ex9}}+
\underbrace{\int\limits_{\{ |x|\le\frac {c_2}2\sqrt\delta\}} \tilde V_{\delta,p}^{2^*-1} u \dm}_{\text{see } \eqref{ex8}} \nonumber \\ 
&\qquad+ \underbrace{\int\limits_{ \{2c_1\sqrt\delta=|x| \}}\partial_{\nu} \tilde V_{\delta,p}u }_{\text{see } \eqref{ex0}}+ O(\delta^\frac{m}{2}), \nonumber
\end{align}
where
\begin{equation}\label{ex8}
\int\limits_{\{ |x|\le\frac {c_2}2\sqrt\delta\}} \tilde V_{\delta,p}^{2^*-1} u \dm=u(p)\underbrace{\left(\int\limits_{\mathbb R^m} U^{2^*-1}dx\right) }_{=\mathfrak b_m}\delta^\frac{m-2}{2}+O\left(\delta^\frac{m}{2}\right),
\end{equation}
and, arguing as in \cite{epv},
\begin{align} \label{ex9}
&\int\limits_{\{ |x|\le\frac {c_2}2\sqrt\delta\}}|-\Delta _g \tilde V_{\delta,p}+\kappa_mS_g  \tilde V_{\delta,p}-\tilde V_{\delta,p}^{2^*-1}|^\frac{2m}{m+2} \ \dm\\
&=
\begin{cases}
O\left({\displaystyle \int_0^\frac {c_2}2} \ \sqrt\delta\, \frac{\delta^m }{(\delta^2+s^2)^\frac{m^2}{m+2}}s^{\frac{2m^2}{m+2}-1+\frac{m(m-6)}{m+2}} \d s\right) & \text{if }m= 4,5, \text{ or }M \text{ is l.c.f.},\\
O\left({\displaystyle \int_0^{\frac {c_2}2\sqrt\delta} } \ \frac{\delta^m}{(\delta^2+s^2)^\frac{m^2}{m+2}}s^{8}|\ln s| \d s\right) & \text{if }m=6 \text{ and }M \text{ is not l.c.f.},\\
O\left({\displaystyle \int_0^{\frac {c_2}2\sqrt\delta}} \  \frac{\delta^m}{(\delta^2+s^2)^\frac{m^2}{m+2}}s^{\frac{2m^2}{m+2}-1-\frac{m(m-6)}{m+2}} \d s\right) & \text{if }m\geq 7 \text{ and }M \text{ is not l.c.f.}
\end{cases} \nonumber \\
&=
\begin{cases}
O\left(\delta^{m} {\displaystyle \int_0^{\frac {c_2}2\sqrt\delta} } \ s^{ -1+\frac{m(m-6)}{m+2}} \d s\right) & \text{if }m= 4,5, \text{ or }M \text{ is l.c.f.},\\
O\left(\delta^6  {\displaystyle \int_0^{\frac {c_2}2\sqrt\delta} } \ \frac{|\ln s|}{s} \d s\right) & \text{if }m=6 \text{ and }M \text{ is not l.c.f.},\\
O\left(\ \delta^{m}  {\displaystyle \int_0^{\frac {c_2}2\sqrt\delta}} \  s^{ -1-\frac{m(m-6)}{m+2}} \d s\right) & \text{if }m\geq 7 \text{ and }M \text{ is not l.c.f.}
\end{cases} \nonumber \\
&= 
\begin{cases}
O\left(\delta^{\frac{m(3m-2)}{2(m+2)}}\right) & \text{if }m= 4,5, \text{ or }M \text{ is l.c.f.},\\
O\left(\delta^6 |\ln\delta|^2\right) & \text{if }m=6 \text{ and }M \text{ is not l.c.f.},\\
O\left(\delta^{\frac{m(m+10)}{2(m+2)}}\right) & \text{if }m\geq 7 \text{ and }M \text{ is not l.c.f.}
\end{cases} \nonumber
\end{align}
This concludes the proof of statement $(iii)$.

$(iv):$ Using \eqref{Bvell} and \eqref{crucial} we obtain
\begin{align*}
|v_\ell|_{g,2^*}^{2^*}-|\tilde V_{\delta,p}|_{g,2^*}^{2^*}
&= \int\limits_{\{|x|\le r\}\cap\{u \le \tilde V_{\delta,p}\}}\left(|\tilde V_{\delta,p}-u |^{2^*}-|\tilde V_{\delta,p}|^{2^*}+2^*\tilde V_{\delta,p} ^{2^*-1}u\right)\dm \\
&\qquad-2^* \int\limits_{\{|x|\le r\}\cap\{u \le \tilde V_{\delta,p}\}}\tilde V_{\delta,p} ^{2^*-1}u\dm - \int\limits_{\{|x|\le r\}\cap \{u \ge \tilde V_{\delta,p}\}}\tilde V_{\delta,p}^{2^*}\dm \\
& =\underbrace{O\left(\int\limits_{\{|x|\le c_1\sqrt\delta \}}\left(\tilde V_{\delta,p}^{2^*-2}u^2 +u ^{2^*}\right)\dm\right)}_{=O(\delta^\frac{m}{2})}-2^* \underbrace{\int\limits_{\{|x|\le r\}\cap\{u \le \tilde V_{\delta,p}\}}\tilde V_{\delta,p} ^{2^*-1}u\dm}_{\text{ see }\eqref{ex3}}\\
&\qquad + \underbrace{O\left(\int\limits_{\{c_2\sqrt\delta\le |x|\le r\}} \tilde V_{\delta,p}^{2^*} \dm\right)}_{=O(\delta^\frac{m}{2})}.
\end{align*}
and
\begin{align}\label{ex3}
&\int\limits_{\{|x|\le r\}\cap\{u \le \tilde V_{\delta,p}\}}\tilde V_{\delta,p}^{2^*-1}u  \dm = \underbrace{\int\limits_{\{ |x|\le\frac {c_2}2\sqrt\delta\}}\tilde V_{\delta,p}^{2^*-1}u  \dm}_{\text{see } \eqref{ex8}} \\ 
&\qquad-
 \underbrace{\int\limits_{\{ |x|\le\frac {c_2}2\sqrt\delta\}\cap\{u \ge \tilde V_{\delta,p}\}}\tilde V_{\delta,p}^{2^*-1}u  \dm}_{=0\text{ see }\eqref{crucial}} +\underbrace{\int\limits_{\{ \frac {c_2}2\sqrt\delta\le |x|\le r\}\cap\{u \le \tilde V_{\delta,p}\}}\tilde V_{\delta,p}^{2^*-1}u \dm}_{=O(\delta^\frac{m }{2})} \nonumber \\
&= \mathfrak b_m u(p) \delta^\frac{m-2}{2}+O\left(\delta^\frac{m }{2}\right) \nonumber
\end{align}
This concludes the proof of statement $(iv)$.
\end{proof}

\section{Uniform bounds in H\"older spaces}
\label{app:A}

In this appendix we prove Lemma \ref{lem:Holder}. Since it does not require additional effort, we consider the more general system
\begin{equation} \label{eq:s_app}
\mathscr{L}_g u_i = h_i(p,u_i) + \sum_{\substack{j=1 \\ j\neq i}}^\ell \lambda |u_j|^{\gamma+1}|u_i|^{\gamma-1}u_i \qquad \text{ in } M,\quad i=1,\ldots,\ell,
\end{equation}
where $(M,g)$ is a closed Riemannian manifold of dimension $m\geq 1$, $\lambda<0$, $\gamma>0$,  and $h_i:M\times \R\to \R$ is a continuous function satisfying $|h_i(p,s)|\leq C|s|$ for every $p\in M, \ |s|\leq 1$.
 
Lemma \ref{lem:Holder} is a particular case of the following result.

\begin{theorem}\label{thm:mainappendix}
For each $\lambda<0$ let $(u_{\lambda,1},\ldots, u_{\lambda,\ell})$ be a nonnegative solution to \eqref{eq:s_app} such that $\{u_{\lambda,i}:\lambda<0\}$ is uniformly bounded in $L^\infty(M)$ for every $i=1,\ldots,\ell$. Then, for any $\alpha\in(0,1)$, there exists $C_\alpha>0$ such that
\begin{equation*}
\|u_{\lambda,i}\|_{\cC^{0,\alpha}(M)}\leq C_\alpha\qquad\text{ \ for every \ } \lambda<0,\ i=1,\ldots,\ell.
\end{equation*}
\end{theorem}

In local coordinates, the system \eqref{eq:s_app} becomes 
\begin{equation*}
\begin{cases}
-\frac{1}{a(x)}\div(A(x)\nabla u_{i})=-\kappa_mS_g(x)u_i+h_i(x,u_{i})+\sum\limits_{\substack{j=1 \\ j\neq i}}^\ell\lambda |u_{j}|^{\gamma+1}|u_{i}|^{\gamma-1}u_{i}\qquad\text{in }\Omega, \\
i=1,\ldots,\ell,
\end{cases}
\end{equation*}
where $\Omega$ is an open subset of $\rm$, $a(x):=\sqrt{|g(x)|}$, $A(x):=\sqrt{|g(x)|}(g^{kl}(x))$, $(g_{kl})$ is the metric written in local coordinates, $(g^{kl})$ is its inverse and $|g|$ is the determinant of $(g_{kl})$. Observe that the second order differential operator is uniformly elliptic and, since $M$ is compact, $a$ is bounded away from $0$. Therefore, we end up with a system of the form
\begin{equation}\label{eq:s2}
-\div(A(x)\nabla u_{i})=f_i(x,u_{i})+ a(x)\sum_{\substack{j=1 \\ j\neq i}}^\ell\lambda  |u_{j}|^{\gamma+1}|u_{i}|^{\gamma-1}u_{i} \quad \text{in }\Omega,\quad i=1,\ldots, \ell.
\end{equation}
Let $\mathrm{Sym}_m\equiv\r^\frac{m(m+1)}{2}$ be the space of real symmetric $m\times m$ matrices. For the system \eqref{eq:s2} we prove the following result.

\begin{theorem}\label{thm: holder bounds}
Let $\Omega$ be an open subset of $\R^m$, $\gamma>0$. Assume that
\begin{itemize}
\item[$(H1)$] $a\in \cC^0(\Omega)$ and $a>0$ in $\Omega$,
\item[$(H2)$] $A\in \cC^1(\Omega,\mathrm{Sym}_m)$ and there exists $\theta>0$ such that $\langle A(x) \xi,\xi\rangle\geq\theta |\xi|^2$ for every $x\in \Omega$, $\xi\in \R^m$,
\item[$(H3)$] $f_i:\Omega \times \R\to \R$ is continuous and there exists $\bar c>0$ such that $|f_i(x,s)|\leq \bar c\,|s|$ for every $x\in \Omega, \ |s|\leq 1, \ i=1,\ldots,\ell$.
\end{itemize}
For each $\lambda<0$ let $(u_{\lambda,1},\ldots,u_{\lambda,\ell})$ be a nonnegative solution to the system \eqref{eq:s2}, such that $\{u_{\lambda,i}:\lambda<0\}$ is uniformly bounded in $L^\infty(\Omega)$ for every $i=1,\ldots,\ell$. Then, given a compact subset $\mathcal{K}$ of $\Omega$ and $\alpha\in (0,1)$, there exists $C>0$ such that
\[
\|u_{\lambda,i}\|_{\cC^{0,\alpha}(\mathcal{K})}\leq C \qquad \text{for every \ } \lambda<0, \ i=1,\ldots,\ell.
\]
\end{theorem}

We now show that Theorem \ref{eq:s} follows from Theorem \ref{thm: holder bounds}.

\begin{proof}[Proof of Theorem \ref{thm:mainappendix}]
Arguing by contradiction, assume that $\{u_{\lambda,i}:\lambda<0\}$ is unbounded in $\cC^{0,\alpha}(M)$ for some $\alpha\in (0,1)$ and some $i=1,\ldots,\ell$. Since, by assumption, this set is uniformly bounded in $L^\infty(M)$, there exist $\lambda_n\to -\infty$ and $p_n\neq q_n$ in $M$ such that $u_{n,i}:=u_{\lambda_n,i}$ satisfies
$$\frac{|u_{n,i}(p_n)-u_{n,i}(q_n)|}{[d_g(p_n,q_n)]^\alpha}\to\infty,$$
where $d_g$ is the geodesic distance in $(M,g)$. As $(u_{n,i})$ is uniformly bounded in $L^\infty(M)$, this implies that $d_g(p_n,q_n)\to 0$. Moreover, since $M$ is compact, a subsequence satisfies $p_n\to \bar p$ in $M$. Hence, $q_n\to \bar p$. Now, in local coordinates around $\bar p$ the system \eqref{eq:s_app} becomes \eqref{eq:s2} with $f_i(x,s):=a(x)(-\kappa_mS_g(x)s+h_i(x,s))$. This is a contradiction to Theorem \ref{thm: holder bounds}.
\end{proof}

Therefore, for the remainder of the appendix, our goal is to prove Theorem \ref{thm: holder bounds}. We follow very closely the proof of \cite[Theorem 1.2]{sttz}, where the case $A(x)=I$ was treated, mainly highlighting the differences that arise from having a divergence-type operator instead of the Laplacian. 

We use the following notation for the seminorm in H\"older spaces:
\[
[u]_{\cC^{0,\alpha}(\Omega)}:=\mathop{\sup_{x,y\in \Omega}}_{x\neq y} \frac{|u(x)-u(y)|}{|x-y|^\alpha}.
\]

\subsection{Auxiliary lemmas}

We present the following generalization of \cite[Lemma 5.2]{cpq}, \cite[Lemma 4.1]{st1} and \cite[Lemma 2.2]{sz} to our setting. The first part of the lemma is required to treat the case $\gamma\leq 1$, while the second part is needed for $\gamma>1$ (see the upcoming proof of Lemma \ref{lem: bound in 0} for more details).

\begin{lemma}[Decay estimates]\label{lemma:exp} 
Let $\tilde\Omega$ be an open subset of $\R^m$ and $\tilde A\in\cC^1(\tilde\Omega,\mathrm{Sym}_m)$ be bounded in the $\cC^1$-norm and such that there are $0<\theta<\Theta$ with $\theta |\xi|^2\leq \langle \tilde A(x) \xi,\xi\rangle \leq \Theta |\xi|^2$ for all $x\in\tilde\Omega$ and $\xi\in \R^m$. Let $a_0\geq\|\tilde A\|_{\cC^1(\tilde\Omega,\mathrm{Sym}_m)}$. For any $R>0$ satisfying $\overline{B_{2R}(0)}\subset\tilde\Omega$ we have the following results:
\begin{enumerate}
\item Take $C\geq 1$ and let $u\in H^1(B_{2R}(0))\cap \cC^0(\overline{B_{2R}(0)})$ be a nonnegative solution of
\[
-\div(\tilde A(x)\nabla u)\leq -C u \quad\text{ in } B_{2R}(0).
\]
Then there exist constants $c_1,c_2>0$, depending only on $m$, $\Theta$ and $a_0$, such that.
\[
\|u\|_{L^\infty(B_R(0))} \leq c_1\|u\|_{L^\infty(B_{2R}(0))} \, e^{-c_2 R \sqrt{C}}.
\]
\item Take $\delta>0$, $\gamma\geq 1$, $C\geq 1$ and let $u\in H^1(B_{2R}(0))\cap \cC^0(\overline{B_{2R}(0)})$ be a nonnegative solution of
\[
-\div(\tilde A(x)\nabla u)\leq -Cu^\gamma +\delta \text{ in } B_{2R}(0).
\]
Then there exists a constant $c>0$, depending only on $m$, $\Theta$ and $a_0$, such that
\[
C\|u\|_{L^\infty(B_R(0))}^\gamma \leq \tfrac{c}{R+R^2}\|u\|_{L^\infty(B_{2R}(0))}+\delta.
\]
\end{enumerate}
\end{lemma}

\begin{proof}
1. For the first statement we follow closely the proof of \cite[Lemma 4.1]{st1}, which considers the case of a constant matrix. Define
\[
z(x):=\sum_{i=1}^m \cosh\left(\sqrt{\tfrac{C}{L}}x_i\right)\quad \text{ with } \quad L:=\max\{1,(a_0 m+\Theta)^2\}.
\]
Observe that, for $x\in B_{2R}(0)$,
\begin{align*}
\div(\tilde A(x)\nabla z)&=\sum_{i,j=1}^m \left(\frac{\partial\tilde A_{ij}}{\partial x_i}(x) \frac{\partial z}{\partial x_j}+\tilde A_{ij}(x)\frac{\partial^2 z}{\partial x_i \partial x_j}\right)\\
				&=\sqrt{\tfrac{C}{L}}\sum_{i,j}^m \frac{\partial\tilde A_{ij}}{\partial x_i}(x) \sinh\left(\sqrt{\tfrac{C}{L}}x_j\right)+ \tfrac{C}{L} \sum_{i=1}^m \tilde A_{ii}(x)\cosh\left(\sqrt{\tfrac{C}{L}}x_i\right)\\
				&\leq \sqrt{\tfrac{C}{L}}\,a_0\sum_{j=1}^m \left|\sinh\left(\sqrt{\tfrac{C}{L}}x_j\right)\right| + \tfrac{C\Theta}{L}z(x)\\
				&\leq Cz(x) \left( \tfrac{a_0}{\sqrt{CL}}+\tfrac{\Theta}{L}\right)\leq Cz(x),
\end{align*}
where in the last inequality we used $C\geq 1$ and the definition of $L$. Moreover, observe that there exist $c_1,c_2>0$, depending on $L$ (that is, on $m$, $a_0$ and $\Theta$), such that
\[
z(x)\geq c_1e^{c_2|x|\sqrt{C}} \qquad \text{ for every } x\in \R^m.
\]
Therefore, given $x_0\in B_R(0)$, by the comparison principle (which we can apply because $C>0$) we have
\[
\frac{u(x)}{\|u\|_{L^\infty(B_{2R}(0))}}\leq \frac{z(x-x_0)}{c_1 e^{c_2R\sqrt{C}}} \quad \text{ in } B_R(x_0).
\]
Evaluating the previous inequality at $x=x_0$ yields
\[
u(x_0)\leq \tfrac{m}{c_1}e^{-c_2R\sqrt{C}}\|u\|_{L^\infty(B_{2R}(0))},
\]
and the conclusion follows.

2. We follow the proof of \cite[Lemma 2.2]{sz} (which deals with the the Laplace operator). Our main addition is the use of the mean value theorem for divergence operators, which reads as follows: Given $\Omega \subset \R^m$ there exist $k,K>0$, only depending on $\theta,\Theta>0$, such that for $y\in \Omega$ there exists an increasing family of sets $D_r(y)\subset \Omega$ such that $B_{kr}(y)\subset D_r(y)\subset B_{Kr}(y)$ and, for every solution $w$ of $-\div (A(x)\nabla w)\leq 0$ in $\Omega$, we have
\[
r\mapsto \frac{1}{|D_r(y)|}\int_{D_r(y)} w \quad \text{ is increasing, \ and } \quad w(y)\leq \frac{1}{|D_r(x_0)|}\int_{D_r(y)} w.
\]
(See \cite[Theorem 6.3]{bh} for the proof of this result, which was previously stated in \cite{c, cr}). Now take a nonnegative solution $v\in H^1(B_{2R}(0))$ of 
\[
-\div(A(x)\nabla w)+Cw^\gamma=0 \text{ in } B_{2R}(0),\quad v=\|u\|_{L^\infty(B_{2R}(0))} \text{ on } \partial B_{2R}(0).
\]
Using the uniform ellipticity and since $C>0$, we can apply the maximum principle to deduce that $v\leq \|u\|_{L^\infty(B_{2R}(0))}$ in $B_{2R}(0)$. Let $\eta\in \cC^\infty_c(B_{2R}(0))$, $0\leq \eta\leq 1$ be a cut-off function such that $\eta=1$ in $B_{3R/2}(0)$, and take $\eta_R(x):=\eta(x/R)$. Then,
\begin{align*}
&\int_{B_{3R/2}(0)}C v^\gamma \leq \int_{B_{2R}(0)} C v^\gamma \eta_R = \int_{B_{2R}(0)} \div (\tilde A(x)\nabla v)\eta_R\\
&\qquad=\sum_{i,j=1}^\ell \int_{B_{2R}(0)} v \left(\tilde A_{ij}(x)\frac{\partial^2 \eta}{\partial x_j \partial x_i} \left(\frac{x}{R}\right)\frac{1}{R^2}+\frac{\partial\tilde A_{ij}}{\partial x_j}(x)\frac{\partial \eta}{\partial x_i}\left(\frac{x}{R}\right)\frac{1}{R}\right)\\
&\qquad\leq a_0 (R^{m-1}+R^{m-2}) \|v\|_{L^\infty(B_{2R}(0))}.
\end{align*}
Now let $y\in B_R(0)$. Since $\gamma\geq 1$, by the mean value theorem presented above and Jensen's inequality,
\begin{align*}
Cv(y)^\gamma &\leq C\left(\frac{1}{|D_{R/(2K)}(y)|} \int_{D_{R/(2K)}(y)} v \right)^\gamma\leq \frac{1}{|D_{R/(2K)}(y)|}\int_{|D_{R/(2K)}(y)|} C v^\gamma\\
	     &\leq \frac{1}{|B_\frac{k R}{2K}(y)|}\int_{B_{R/2}(y)} C v^\gamma\leq \left(\frac{2K}{k R}\right)^m \frac{1}{|B_1(0)|}\int_{B_{3R/2}(0)} C v^\gamma \\
	     &\leq \frac{a_0}{R+R^2} \|v\|_{L^\infty(B_{2R}(0))}.
\end{align*}
By the maximum principle we have that $u\leq v+(\delta/C)^{1/\gamma}$, from which the conclusion follows.
\end{proof}

\begin{lemma}[Liouville-type results]\label{lem:liouville2}
Let $A\in \mathrm{Sym}_m$ be a constant matrix and $\alpha\in (0,1)$.
\begin{enumerate}
\item  Let $u,v\in H^1_{loc}(\R^m)\cap \cC^0(\R^m)$ be nonnegative functions satisfying $uv\equiv 0$ and
\[
-\div(A\nabla u)\leq 0,\quad -\div(A\nabla v)\leq 0 \qquad \text{ in } \R^m.
\]
If $[u]_{\cC^{0,\alpha}(\R^m)}<\infty$ and $[v]_{\cC^{0,\alpha}(\R^m)}<\infty$ then, either $u\equiv 0$, or $v\equiv 0$.
\item Let and $u,v\in H^1_{loc} (\R^m)\cap \cC^0(\R^m)$ be nonnegative solutions of
\begin{equation*}\label{eq:Liouville_systeminequalities}
-\div(A\nabla u)\leq -a u^\gamma v^{\gamma+1},\quad  -\div(A\nabla u)\leq -a v^\gamma u^{\gamma+1} \qquad \text{ in } \R^m,
\end{equation*}
with $a>0$ and $\gamma>0$. If $[u]_{\cC^{0,\alpha}(\R^m)}<\infty$ and $[v]_{\cC^{0,\alpha}(\R^m)}<\infty$ then, either $u\equiv 0$, or $v\equiv 0$.

\item Let $u$ be a solution of $-\div (A\nabla u)=0$ in $\R^m$ such that $[u]_{C^{0,\alpha}(\R^m)}<\infty$. Then u is constant.
\end{enumerate}
\end{lemma}

\begin{proof}
Inspired by the proof of \cite[Theorem 3.1]{st1}, since $A$ is symmetric and positive definite, there exist an orthogonal matrix $O$ and a diagonal matrix $D=\text{diag}(d_1,\ldots, d_\ell)$ with $d_i>0$ such that $O^t AO=D$. Then, for $\sqrt{D}:=\text{diag}(\sqrt{d_1},\ldots, \sqrt{d_\ell})$ and $\bar u(x):=u(O\sqrt{D}x)$, we have
\[
\Delta \bar u(x)=\div(A\nabla u)(O\sqrt{D}x).
\]
Under this change of variables, we reduce the proof of this lemma to the case of the Laplace operator. Therefore, parts 1 and 3 follow from \cite[Proposition 2.2 and Corollary 2.3]{nttv}, while part 2 follows from \cite[Lemma A.3]{sttz} (see also \cite[Corollary 1.14-(ii)]{st2} for the case $\gamma\geq 1$).
\end{proof}

\subsection{A contradiction argument and a blow-up analysis}

Fix $\alpha\in(0,1)$. Without loss of generality, we assume that $\overline{B_3(0)}\subset \Omega$. Under the assumptions of Theorem \ref{thm: holder bounds}, we aim at proving the uniform H\"older bound in $B_1(0)$. Fix $\Lambda>0$ such that
\begin{equation} \label{eq:Linfty_bound}
\|u_{\lambda,i}\|_{L^\infty(B_3(0))} \le\Lambda\qquad\forall \ \lambda<0, \ i=1,\dots,\ell.
\end{equation}
Let $\eta \in \mathcal{C}^1_c(\R^m)$ be a radially decreasing cut-off function such that 
\begin{equation*}\label{def eta}
\begin{cases}
\eta(x) =1 & \text{for }x \in B_1(0),\\
\eta(x)= 0 & \text{for }x \in \R^m \smallsetminus B_2(0), \\
\eta(x) = (2-|x|)^2 & \text{for }x \in B_2(0) \smallsetminus B_{3/2}(0).
\end{cases}
\end{equation*}
For $x \in B_2(0)$, let $d_x:= \dist(x,\partial B_2(0))$. It is shown in \cite[Remark 2.1]{sttz} that 
\begin{equation}\label{eq:osc_eta}
\sup_{x \in B_2(0)} \ \sup_{\rho \in (0,d_x/2)}  \frac{\sup_{B_{\rho}(x)} \eta    }{\inf_{B_\rho(x)} \eta} \le 16.  
\end{equation}
Our goal is to prove that there exists $C > 0$ such that
\begin{equation}\label{lip scaling}
\sup_{\substack{x \neq y \\ x,y \in \overline{B_2(0)}}}    \frac{ |(\eta u_{n,i})(x)-(\eta u_{n,i})(y)|}{|x-y|^{\alpha}} \leq C\qquad\forall \ \lambda<0, \ i=1,\dots,\ell.
\end{equation}
Since $\eta=1$ in $B_1(0)$, Theorem \ref{thm: holder bounds} follows readily from \eqref{lip scaling}. 

To prove \eqref{lip scaling} we argue by contradiction. Assume that \eqref{lip scaling} is false. Then there exist $\lambda_n\to -\infty$ such that $u_{n,i}:=u_{\lambda_n,i}$ satisfies
\begin{equation}\label{absurd assumption}
    L_n:=\max_{i=1,\ldots,\ell}\sup_{\substack{x \neq y \\ x,y \in \overline{B_2(0)}} } \frac{ |(\eta u_{n,i})(x)-(\eta u_{n,i})(y)|}{|x-y|^{\alpha}} \longrightarrow \infty \quad \text{as }n \to \infty.
\end{equation}
We may assume that the maximum is attained at $i=1$. Then, for each $n$, we fix a pair of points $x_n, y_n \in B_{2}(0)$ with  $x_n \neq y_n$ such that
\begin{equation*}
    L_n=\frac{ |(\eta u_{n,1})(x_n)-(\eta u_{n,1})(y_n)|}{|x_n-y_n|^{\alpha}}.
\end{equation*}
As $(u_n)$ is uniformly bounded in $L^\infty(B_2(0))$, this implies that $|x_n-y_n| \to 0$. So $(x_n)$ and $(y_n)$ converge to the same point. We denote
\begin{equation}\label{eq:xinfty}
x_\infty:=\lim_{n\to\infty} x_n=\lim_{n\to\infty} y_n,\qquad A_\infty:=A(x_\infty),\qquad a_\infty:=a(x_\infty).
\end{equation}
The contradiction argument is based on two blow-up sequences
\begin{equation*}
    v_{n,i}(x) := \eta(x_n) \frac{u_{n,i}(x_n + r_n x)}{L_n r_n^{\alpha}} \quad \text{and} \quad \bar{v}_{n,i}(x) := \frac{(\eta u_{n,i})(x_n + r_n x)}{L_n r_n^{\alpha}},
\end{equation*}
both defined in the scaled domain $\Omega_n:= (B_3(0)-x_n)/r_n$; see \cite{sz, sttz, tvz, w}. Here $r_n\in(0,1)$, $r_n\to 0$, will be conveniently chosen later. Observe that $B_{1/r_n}(0)\subset\Omega_n $, therefore $\Omega_n$ approaches $\R^m$ as $n\to \infty$. Since $\eta$ is positive in $B_2(0)$, the functions $v_{n,i}$ and $\bar{v}_{n,i}$ are nonnegative and nontrivial in $\Omega_n':=(B_{2}(0)-x_n)/r_n$. Note that, as $x_n \in B_2(0)$, $\Omega_n'$ approaches a limit domain $\Omega_\infty$ which is either a half-space or the whole $\R^m$, as $n\to \infty$. 

\begin{lemma}\label{lem: basic prop} Under the assumptions of \emph{Theorem} \ref{thm: holder bounds}, $v_{n,i}$ and $\bar{v}_{n,i}$ have the following properties:
\begin{enumerate}
\item The sequence of $\alpha$-Hölder seminorms $([\bar{v}_{n,i}]_{\cC^{0,\alpha}(\overline{\Omega_n'})})$ of $\bar{v}_{n,i}$ in $\overline{\Omega_n'}$ is uniformly bounded. Furthermore, for every $n\in\n$,
\[
\max_{i=1,\dots,\ell} \sup_{\substack{x \neq y \\ x,y \in \overline{\Omega_n'}}}    \frac{ \left|\bar v_{n,i}(x)-\bar v_{n,i}(y)\right|}{|x-y|^{\alpha}} = \frac{ \left|\bar v_{n,1}(0)-\bar v_{n,1}\left(\frac{y_n-x_n}{r_n}\right)\right|}{\left|\frac{y_n-x_n}{r_n}\right|^{\alpha}} = 1.
\]
\item $v_{n,i}$ solves
\begin{equation*}
-\div(A_n(x)\nabla v_{n,i}) = g_{n,i}(x)+ a_n(x)\sum_{j \neq i}  \Lambda_n|v_{n,j}|^{\gamma+1}|v_{n,i}|^{\gamma-1}v_{n,i}\quad\text{in }\Omega_n,
\end{equation*}
where \ $a_n(x):=a(x_n+r_n x),\quad  A_n(x):=A(x_n+r_nx),$
\[
\ g_{n,i}(x):=\frac{\eta(x_n)r_n^{2-\alpha}}{L_n}f_{n,i}(x_n+r_nx,u_{n,i}(x_n+r_n x)),
\] 
\[
\quad  \Lambda_n:=  \lambda_n r_n^{2(\alpha\gamma+1)} \left(\tfrac{L_n}{\eta(x_n)}\right)^{2\gamma}.
\]
\item There exist $a_0,a_1,a_2,\theta,\Theta>0$ such that, for every $n\in\n$, 
$$a_1\leq a_n(x)\leq a_2\quad\text{and}\quad\theta |\xi|^2\leq\langle A_n(x) \xi,\xi\rangle\leq \Theta |\xi|^2\quad\forall x\in \Omega_n, \ \xi\in \mathbb{R}^m,$$ 
$$\|A_n\|_{\cC^1(\Omega_n,\mathrm{Sym}_m)}\leq a_0.$$
\item  $\|g_{n,i}\|_{L^\infty(\Omega_n)}\to 0$, and there is $c_0>0$ such that $|g_{n,i}(x)|\leq c_0r_n^2|v_{n,i}(x)|$ for every \ $x\in\Omega'_n, \ n\in\n$.
\item $\|v_{n,i} - \bar{v}_{n,i}\|_{L^\infty(\cK)} \to 0$ for any compact set $\cK \subset \R^m$ and every $i=1,\ldots, \ell$. 
\item For any compact set $\cK \subset \R^m$ there exists $C>0$ such that 
\[
|v_{n,i}(x) - v_{n,i}(y)| \le C + |x-y|^\alpha\qquad\forall x,y\in\cK, \ i=1,\dots,\ell.
\]
In particular, $(v_{n,i})$ has uniformly bounded oscillation in any compact set.
\end{enumerate}
\end{lemma}
\begin{proof} The first two statements are proved by direct computation. The third one follows from $(H1)-(H2)$ with $a_1:=\min_{x\in\overline{B_3(0)}}a(x)$, $a_2:=\max_{x\in\overline{B_3(0)}}a(x)$, $\Theta:=\|A\|_{\cC^0(\overline{B_3(0)},\mathrm{Sym}_m)}$ and $a_0:=\|A\|_{\cC^1(\overline{B_3(0)},\mathrm{Sym}_m)}$, while the fourth one is a consequence of $(H3)$, \eqref{eq:Linfty_bound}, \eqref{absurd assumption} and $r_n\to 0$. The last two statements are proved exactly as those in \cite[Lemma 2.2-(4),(5)]{sttz}.
\end{proof}

\begin{lemma}\label{lem: bound in 0}
Take $r_n \to 0^+$ such that
\begin{equation}\label{conditions r_n}
\liminf_{n\to\infty} |\Lambda_n|>0\quad \text{ and } \quad \limsup_{n\to\infty} \frac{|x_n-y_n|}{r_n}<\infty.
\end{equation}
Then the sequence $(v_{n}(0))$ is bounded in $\r^\ell$, where $v_n:=(v_{n,1},\ldots,v_{n,\ell})$. 
\end{lemma}
\begin{proof}
We follow \cite[Lemma 2.3]{sttz}, to which we refer for further details. 

Assume by contradiction that $|v_{n,\bar i}(0)|\to \infty$ for some $\bar i\in\{1,\ldots,\ell\}$. Take $R\geq |y_n-x_n| /r_n$ for all $n\in\n$. From Lemma \ref{lem: basic prop}-1 we get that $|v_{n,\bar i}(0)|=|\bar v_{n,\bar i}(0)|\leq |\bar v_{n,\bar i}(x)|+(2R)^\alpha$ if $x\in\Omega'_n\cap B_{2R}(0)$. So $\inf_{\Omega'_n\cap B_{2R}(0)}|\bar v_{n,\bar i}(x)|\to\infty$ and, as $\bar{v}_{n,\bar i}|_{\rm\smallsetminus\Omega_n'}=0$, we conclude that $B_{2R}(0)\subset\Omega'_n$ for $n$ large enough. Since $R$ is arbitrary, it follows that $\Omega_n'$ approaches $\R^m$ as $n\to \infty$.

Let $\varphi \in \mathcal{C}^\infty_c(B_{2R}(0))$ be a nonnegative cut-off function such that $\varphi=1$ in $B_R(0)$. Take $i=1,\ldots, \ell$.  Testing the equation for $v_{i,n}$ in Lemma \ref{lem: basic prop}-2 against $v_{i,n}\varphi^2$, we obtain
\begin{align*}
&\theta\int_{B_{2R}(0)} |\nabla v_{n,i}|^2\varphi^2 + \int_{B_{2R}(0)} a_n\sum_{j\neq i} |\Lambda_n|  |v_{n,j}|^{\gamma+1}|v_{n,i}|^{\gamma+1}\varphi^2 \\
&\qquad\leq\int_{B_{2R}(0)} \langle A_n\nabla v_{n,i},\nabla v_{n,i} \rangle\varphi^2 + \int_{B_{2R}(0)} a_n\sum_{j\neq i} |\Lambda_n|  |v_{n,j}|^{\gamma+1}|v_{n,i}|^{\gamma+1}\varphi^2 \\
	&\qquad= -2\int_{B_{2R}(0)} \langle A_n\nabla v_{n,i},\nabla \varphi \rangle v_{n,i}\varphi + \int_{B_{2R}(0)} g_{n,i}v_{n,i} \varphi^2\\
	&\qquad\leq \frac{\theta}{2}\int_{B_{2R}(0)} |\nabla v_{n,i}|^2\varphi^2 + C\int_{B_{2R}(0)}(v_{n,i}^2 + 1)
\end{align*}
where in the last inequality we used Lemma \ref{lem: basic prop}-4, $(H2)$ and Young's inequality. By Lemma \ref{lem: basic prop}-3 we have
\[
\sum_{j\neq i} \int_{B_R(0)} |\Lambda_n| |v_{n,j}|^{\gamma+1}|v_{n,i}|^{\gamma+1}\leq  C\int_{B_{2R}(0)}(v_{n,i}^2 + 1).
\]
Combining this inequality with $\liminf |\Lambda_n|>0$ and Lemma \ref{lem: basic prop}-6 we deduce that
\[
 |v_{n,j}(x)|^{2(p+1)}|v_{n,i}(x)|^{2(p+1)} \leq C( |v_{n,i}(x)|^2 +1)( |v_{n,j}(x)|^2 +1)\quad\forall x\in B_R(0),
\]
for every for $i\neq j$. Using again Lemma \ref{lem: basic prop}-6 and our assumption that $|v_{n,\bar i}(0)|\to \infty$ we derive
\[
\inf_{B_{2R}(0)} |v_{n,\bar i}|\to \infty,\qquad \sup_{B_{2R}(0)} |v_{n,i}|\to 0 \quad\forall i\neq \bar i.
\]
Now we consider two cases.

Assume first that $\gamma\in (0,1]$ (as when $2(\gamma+1)=2^*$ and $m\geq 5$). There are again two possibilites:

\noindent \textit{Case 1.} If $\bar i=1$, take $I_n:=a_1|\Lambda_n| \inf_{B_{2R}(0)} |v_{n,1}|^{\gamma+1}\to \infty$. Then, since $v_{n,i}\to 0$ in $B_{2R}(0)$ and $\gamma\leq 1$, from Lemma \ref{lem: basic prop} we get that
\begin{align*}
-\div(A_n\nabla v_{n,i})\leq C r_n^2 v_{n,i} - I_n v_{n,i}^\gamma \leq -\frac{I_n}{2}v_{n,i}\quad\text{in \ }B_{2R}(0),\quad\forall i>1,
\end{align*}
Since $\|A_n\|_{\cC^1(\Omega_n,\mathrm{Sym}_m)}\leq a_0$ for all $n\in\n$ (by Lemma \ref{lem: basic prop}-5), Lemma \ref{lemma:exp}-1 yields 
\[
0\leq v_{n,i}\leq c_1 e^{-c_2\sqrt{I_n}}\quad \text{ in }B_{R}(0), 
\]
and therefore
$$|\lambda_n| M_n v_{n,i}^{\gamma+1}v_{n,1}^\gamma \leq 2I_n c_1 e^{-c_2(\gamma+1) \sqrt{I_n}}\to 0 \qquad\text{ in } B_R(0)$$
for $n$ large, and
\[
\div(A_n\nabla v_{n,i})\to 0 \quad \text{ in } L^\infty(B_{R}(0)).
\]
Then, setting $w_n(x):=v_{n,1}(x)-v_{n,1}(0)$, by the Ascoli-Arzel\`a theorem we have that $w_n\to w_\infty$ in $L^\infty(B_R(0))$. Moreover, $A_n(x)\to A(x_\infty)=:A_\infty$ with $x_\infty$ as in \eqref{eq:xinfty}. Observing that $R$ may be taken arbitrarily large, we conclude that
\[
\div(A_\infty \nabla w_\infty)=0 \quad \text{ in } \R^m.
\]
Arguing as in as in \cite[p. 401--402]{sttz} and using \eqref{eq:osc_eta}, we see that $[w_\infty]_{C^{0,\alpha}(\R^m)}=1$, which contradicts Lemma \ref{lem:liouville2}-3.

\noindent \textit{Case 2.} If $\bar i>1$, take $I_n:= \sum_{j>1} |\Lambda_n| \inf_{B_{2R}(0)} |v_{n,j}|^{\gamma+1}\to \infty$. Then
\begin{align*}
-\div(A_n\nabla v_{n,1})\leq -\frac{I_n}{2}|v_{1,n}|^\gamma \leq -\frac{I_n}{2}|v_{n,1}| \quad \text{ in } B_{2R}(0).
\end{align*}
Therefore, again by Lemma \ref{lemma:exp}-1, $v_{n,1}\leq c_1 e^{-c_2\sqrt{I_n}}$ in $B_R(0)$, which gives again 
\[
\div(A_n(x)\nabla v_{n,1}(x))\to 0 \quad \text{ uniformly in } B_R(0),
\]
a contradiction.

Finally, if $\gamma>1$, one may argue exactly as in Case 1 of the proof of \cite[Lemma 2.3]{sttz}, using this time the decay estimate Lemma \ref{lemma:exp}-2. In both cases, $\bar i=1$ and $\bar i>1$, we end up with $\div(A_n(x)\nabla v_{n,1}(x))\to 0$ locally uniformly in $\R^m$, leading as before to a contradiction.
\end{proof}

\begin{lemma}\label{M_n illimitato}
Up to a subsequence, we have that
\[
|\lambda_n| \left( \frac{L_n}{\eta(x_n)}\right)^{2\gamma}|x_n-y_n|^{2(\alpha\gamma +1)} \to \infty
 \]
\end{lemma}

\begin{proof}
We follow \cite[Lemma 2.5]{sttz}. Arguing by contradiction, assume that the sequence considered in the statement is bounded and take
\[
r_n:=\left( |\lambda_n| \left( \frac{L_n}{\eta(x_n)}\right)^{2\gamma} \right)^{-1/(2(\alpha\gamma +1))}\to 0.
\]
With this choice we have that \eqref{conditions r_n} is satisfied and from Lemma \ref{lem: bound in 0} we deduce that $(\bar{v}_n(0))$ is bounded in $\r^\ell$. Combining this fact with Lemma \ref{lem: basic prop}-1, we deduce the existence of $(v_1,\ldots, v_\ell)\in \mathcal{C}^{0,\alpha}(\rm,\r^\ell)$ such that $\bar v_{n,i}\to v_i$ in the $\alpha$-Hölder norm as $n\to \infty$. Under the previous choice of $r_n$ one has $\Lambda_n=-1$. Hence, by elliptic regularity, the convergence of $\bar v_{n,i}$  to $v_i$ is actually in $\mathcal{C}^{1,\alpha}$, and
\[
-\div(A_\infty \nabla v_i)=-a_\infty v_i^\gamma \sum_{j\neq i} v_j^{\gamma+1} \quad \text{ in } \Omega_\infty,
\]
where $\Omega_\infty$ is the limit domain of $\Omega'_n$ and $A_\infty$, $a_\infty$ are defined in \eqref{eq:xinfty}. In particular, for any $i\neq j$, the pair $(v_i,v_j)$ is a nonnegative solution of
\[
-\div(A_\infty \nabla v_i)\leq -a_\infty v_i^\gamma  v_j^{\gamma+1},\quad -\div(A_\infty \nabla v_j)\leq -a_\infty v_j^\gamma  v_i^{\gamma+1}\quad \text{ in } \Omega_\infty,
\]
having bounded $\alpha$-H\"older seminorm. Using Lemma \ref{lem:liouville2}-2. and reasoning from this point on word by word as in \cite[Lemma 2.5]{sttz}, we obtain a contradiction.
\end{proof}

\begin{lemma}\label{lem 3.6 notateve}
Let $r_n:=|x_n-y_n|$. Then there exists $(v_1,\ldots, v_\ell) \in \mathcal{C}^{0,\alpha}(\R^m, \R^\ell)$ such that, up to a subsequence,
\begin{itemize}
\item[$(i)$] ${v}_{n,i} \to {v}_i$ in $L^\infty_{loc}(\R^m)\cap H^1_{loc}(\R^m)$ \ for all \ $i=1,\ldots, \ell$;
\item[$(ii)$] for any $r>0$,
\begin{equation*}
\lim_{n \to \infty} \int_{B_r(0)} |\Lambda_n| |v_{n,i}|^{\gamma+1} |v_{n,j}|^{\gamma+1}  = 0 \qquad \forall i,j=1,\ldots,\ell\ \text{with } i\neq j.
\end{equation*}
In particular, $v_iv_j \equiv 0$ in $\R^m$ for every $i\neq j$.
\end{itemize}
\end{lemma}
\begin{proof} 
Using   Lemmas \ref{lem: basic prop}, \ref{lem: bound in 0} and \ref{M_n illimitato}, in particular the smoothness, boundedness and uniform ellipticity of $A_n$, the proof is obtained from a straightforward adaptation of that of \cite[Lemma 2.6]{sttz} (which, in turn, is based on \cite[Lemma 3.6]{nttv})). Observe that $v_iv_j\equiv 0$ is a direct consequence consequence of the strong convergence of $v_n$, the convergence in $(ii)$ and the fact that 
\[
|\Lambda_n| =|\lambda_n| \left( \frac{L_n}{\eta(x_n)}\right)^{2\gamma}   |x_n-y_n|^{2(\alpha\gamma +1)} \to\infty.\qedhere
\]
\end{proof}

\begin{lemma}\label{lem: end blow-up}
Let $(v_1,\ldots, v_\ell)$ be as in \emph{Lemma} \ref{lem 3.6 notateve} and $A_\infty$ be as in \eqref{eq:xinfty}. Then,
\begin{itemize}
\item[($i$)] $\max_{x \in \partial B_1} |v_1(x)-v_1(0)|=1$ and 
$$\div(A_\infty \nabla v_1)=0 \text{ \ in \ }\Omega_1:=\{x\in\rm:v_1(x)>0\}.$$
The set $\Omega_1$ is open and connected, and $\Omega_1\neq\R^m$.
\item[($ii$)] $v_i \equiv 0$ in $\R^m$ for every $i>1$.
\end{itemize}
\end{lemma}
\begin{proof}
Using the previous lemma together with Lemma \ref{lem:liouville2}, the proof follows exactly as the one of \cite[Lemma 2.7]{sttz}.
\end{proof}

\subsection{The domain variation formula. End of the proof}

\begin{lemma}
Let $(v_1,\ldots, v_\ell)$ be as in Lemma \ref{lem 3.6 notateve} and $A_\infty \in \mathrm{Sym}_m$ as in \eqref{eq:xinfty}. Then, for any vector field $Y\in \cC^1_c(\R^m,\R^m)$, we have
\begin{equation}\label{eq:domainvariation}
\int_{\R^m} \left(\langle\mathrm dY A_\infty\nabla v_1,\nabla v_1\rangle-\frac{1}{2}\langle A_\infty \nabla v_1,\nabla v_1\rangle \div Y \right)=0.
\end{equation}
\end{lemma}

\begin{proof}
We test the $i$-th equation in Lemma \ref{lem: basic prop}-2 against $\langle \nabla v_{n,i},Y\rangle $, integrate by parts and take the sum for all $i=1,\ldots, \ell$ to obtain
\begin{align*}
&\sum_{i=1}^\ell\int_{\Omega_n} \langle A_n\nabla v_{n,i},\nabla \langle \nabla v_{n,i},Y\rangle \rangle+ \mathop{\sum_{i,j=1}^\ell}_{j\neq i} \int_{\Omega_n} |\Lambda_n| a_n v_{n,j}^{\gamma+1}v_{n,i}^\gamma \langle \nabla v_{n,i},Y\rangle \\
&\qquad = \sum_{i=1}^\ell \int_{\Omega_n} g_{n,i}(x) \langle \nabla v_{n,i},Y\rangle.
\end{align*}
Observe that
\begin{align*}
&\int_{\Omega_n} \langle A_n \nabla v_{n,i},\nabla \langle \nabla v_{n,i},Y\rangle \rangle=\int_{\Omega_n} \langle A_n\nabla v_{n,i},\,D^2v_{n,i} Y+(\mathrm dY)^t\nabla v_{n,i} \rangle\\
&\qquad=\int_{\Omega_n} \left(\langle D^2 v_{n,i} A_n\nabla v_{n,i},Y\rangle + \langle\mathrm dY A_n\nabla v_{n,i},\nabla v_{n,i}\rangle \right)\\
&\qquad=\int_{\Omega_n}\Big(-\tfrac{1}{2} \langle A_n\nabla v_{n,i},\nabla v_{n,i}\rangle\div Y\\
&\qquad\qquad- \tfrac{1}{2}\sum_{j,k,l}  \frac{\partial (A_n)_{jk}}{\partial x_l}Y_l \frac{\partial v_{n,i}}{\partial x_k}\frac{\partial v_{n,i}}{\partial x_j}+ \langle\mathrm dY A_n\nabla v_{n,i},\nabla v_{n,i}\rangle \Big)\\
&\qquad\longrightarrow\int_{\R^m}  \left(-\tfrac{1}{2}\langle A_\infty \nabla v_1,\nabla v_1\rangle \div Y +\langle\mathrm dY A_\infty\nabla v_1,\nabla v_1\rangle\right)
\end{align*}
because $Y$ has compact support, $v_{n,i}\to v_i$ strongly in $H^1_{loc}(\R^m)$ and $\frac{\partial (A_n)_{jk}}{\partial x_l}=r_n^2  \frac{\partial A_{jk}}{\partial x_l}(x_n+r_n \,\cdot\,)\to 0$ in $L^\infty_{loc}(\R^m)$. Moreover,
\begin{align*}
&\mathop{\sum_{i,j=1}^\ell}_{j\neq i}\int_{\Omega_n} |\Lambda_n| a_nv_{n,j}^{\gamma+1}v_{n,i}^\gamma \langle \nabla v_{n,i},Y\rangle=\frac{1}{\gamma+1} \mathop{\sum_{i,j=1}^\ell}_{i< j} \int_{\Omega_n} |\Lambda_n| a_n \langle  \nabla (v_{n,j}^{\gamma+1}v_{n,i}^{\gamma+1}),Y\rangle\\
&=-\mathop{\sum_{i,j=1}^\ell}_{i< j} \int_{\Omega_n} |\Lambda_n|a_n(x) v_{n,j}^{\gamma+1}v_{n,i}^{\gamma+1} \div Y \\
&\qquad- \mathop{\sum_{i,j=1}^\ell}_{i< j} \int_{\Omega_n} |\Lambda_n| r_n^2 v_{n,j}^{\gamma+1}v_{n,i}^{\gamma+1} \langle \nabla a(x_n+r_n x),Y\rangle\longrightarrow 0
\end{align*}
by Lemma \ref{lem 3.6 notateve}\,$(ii)$. The statement follows from this facts.
\end{proof}
\medskip

\begin{proof}[End of the proof of Theorem \ref{thm: holder bounds}] Since $A_\infty\in \text{Sym}_m$ is positive definite, there exist an orthogonal matrix $O$ and a diagonal matrix $D=\text{diag}(d_1,\ldots, d_\ell)$ with $d_i>0$ such that $O^t A_\infty O=D$. Let 
\[
u_1(x):=v_1(O\sqrt{D}x),\quad \text{ so that } \quad \nabla u_1(x)=\sqrt{D} O^t v_1(O\sqrt{D}x).
\]
Then, from Lemma \ref{lem: end blow-up} we get
\begin{itemize}
\item $\displaystyle \max_{|\sqrt{D}x|=1}|u_1(x)-u_1(0)|=1$,
\item $\Delta u_1=0$ in $\{u_1>0\}$, which is an open connected set, that does not coincide with $\R^m$;
\end{itemize}
while \eqref{eq:domainvariation} turns into
\begin{equation}\label{eq:domainvariation2}
\int_{\R^m} \left(\langle\mathrm dZ \nabla u_1,\nabla u_1 \rangle-\frac{1}{2}  |\nabla u_1|^2\div Z\right)=0
\end{equation}
for $Z(x):=\sqrt{D} O^t Y(O\sqrt{D} x)$. Since $Y$ is an arbitrary vector field with compact support, then \eqref{eq:domainvariation2} holds true for every $Z\in \cC^1_c(\R^m,\R^m)$. Given $x_0\in \R^m$ and $r>0$, let $\eta_\delta\in \cC^\infty_c(B_{r+\delta}(x_0))$ be a cut-off function such that $0\leq \eta_\delta \leq 1$ and $\eta_\delta=1$ in $B_r(x_0)$. Then, taking $Z(x)=Z_\delta(x):=(x-x_0)\eta_\delta$ in \eqref{eq:domainvariation2} and letting $\delta\to 0$, we derive the local Pohozaev identity
\[
(2-m)\int_{B_r(x_0)}|\nabla u_1|^2=\int_{\partial B_r(x_0)} r(2(\partial_\nu u_1)^2-|\nabla u_1|^2)
\]
(see for instance \cite[Corollary 3.16]{rtt} for the details). From this, it is now classical to deduce an Almgren monotonicity formula, namely, if we set
\begin{align*}
&E(x_0,r):=\frac{1}{r^{m-2}}\int_{B_r(x_0)}|\nabla u_1|^2,\qquad H(x_0,r):=\frac{1}{r^{m-1}}\int_{\partial B_r(x_0)} u_1^2,\\
&N(x_0,r):=\frac{E(x_0,r)}{H(x_0,r)},
\end{align*}
then we have that $H(x_0,r) \neq 0$ for every $r>0$, the function $r\mapsto N(x_0,r)$ is absolutely continuous and monotone nondecreasing, and
\[
\frac{\de}{\de r} \log H(x_0,r) = \frac{2}{r}N(x_0,r)
\]
(see for instance \cite[Theorem 3.21]{rtt} for a proof). Moreover, if $N(x_0,r)= \varrho$ for every $r \in [r_1,r_2]$, then $u_1=r^\varrho\hat{u}_1(\vartheta)$ in $\{r_1 < r< r_2\}$, where $(r,\vartheta)$ denotes a system of polar coordinates centered at $x_0$. Therefore we have obtained precisely the statements contained in \cite[Lemma 2.7 and Proposition 2.9]{sttz}. From this point on we argue \emph{exactly} as in \cite[Section 2.3]{sttz} to obtain a contradiction.
\end{proof}

\section{Lipschitz continuity of the limiting profiles and regularity of the free boundaries}\label{app:generaltheorem_Lip_Reg}

Staying within the framework of Appendix \ref{app:A} we continue our study of system \eqref{eq:s2}. Our aim now is to prove the following result.

\begin{theorem}\label{thm:generaltheorem_Lip_Reg}
Let $\Omega$ be an open subset of $\R^m$ and $\gamma>0$. Assume that
\begin{itemize}
\item[$(H1')$] $a\in \cC^1(\Omega)$ and $a>0$ in $\Omega$,
\end{itemize}
and that $A$ and $f_i$ satisfy the assumptions $(H2)$ and $(H3)$ stated in \emph{Theorem \ref{thm: holder bounds}}.
For each $\lambda<0$, let $(u_{\lambda,1},\ldots,u_{\lambda,\ell})$ be a nonnegative solution to the system \eqref{eq:s2} satisfying
\begin{itemize}
\item[$(H4)$] $u_{\lambda,i}\to u_{i}$ strongly in $H^1(\Omega)\cap \cC^{0,\alpha}(\Omega)$ for every $\alpha\in (0,1)$, as $\lambda\to -\infty$, where $u_i\not \equiv 0$. 
\item[$(H5)$] $\displaystyle \int_\Omega \lambda u_{\lambda,i}^\gamma u_{\lambda,j}^\gamma\to 0$ whenever $i\neq j$; in particular, $u_iu_j\equiv 0$ if $i\neq j$.
\item[$(H6)$] $-\div(A(x)\nabla u_i)=f_i(x,u_i)$ in the open set $\{x\in \Omega:\ u_i(x)>0\}$.
\end{itemize}
Then, the following statements hold true:
\begin{itemize}
\item[$(a)$] $u_i$ is Lipschitz continuous for every $i=1,\ldots, \ell$.
\item[$(b)$] the nodal set $\Gamma:=\{x\in \Omega: \ u_i(x)=0\ \forall i=1,\ldots, \ell\}$ is the union of two sets, $\Gamma =\mathscr{R}\cup\mathscr S$ with $\mathscr{R}\cap\mathscr S=\emptyset$, where $\mathscr{R}$ is an $(m-1)$-dimensional $\cC^{1,\alpha}$-submanifold of $\Omega$ and $\mathscr S$ is a relatively closed subset of $\Omega$ whose Hausdorff measure is smaller than or equal to $m-2$.  Moreover,
\begin{itemize}
\item given $x_0\in \mathscr{R}$, there exist $i, j$ such that
\[
\lim_{x\to x_0^+} \langle A(x)\nabla u_i(x),\nabla u_i(x)\rangle=\lim_{x\to x_0^-} \langle A(x)\nabla u_j(x),\nabla u_j(x)\rangle\neq 0,
\]
where $x\to x_0^\pm$ are the limits taken from opposite sides of $\mathscr{R}$,
\item and, if $x_0\in\mathscr S$, then
\[
\lim_{x\to x_0}\langle A(x)\nabla u_i(x),\nabla u_i(x)\rangle= 0 \quad \text{ for every } i=1,\ldots, \ell.
\]
\end{itemize}
\end{itemize}
\end{theorem}
The proof of the Lipschitz continuity of the limiting profiles goes along the lines of \cite[Theorem 1.2 and Section 4]{nttv} (see also \cite[Theorem 1.5-(4)]{sttz}, while the regularity of the nodal set follows \cite[Theorem 1.1]{TavaresTerracini1} (see also \cite[Theorem 1.7]{sttz}), where the differential operator is the Laplacian. The proof requires a careful blow-up analysis and is mainly based on the use of an Almgren-type monotonicity formula. Adapting it to divergence form operators with nonconstant matrices is not completely straghtforward, and for that we use ideas from \cite{Kukavica,GarofaloGarciaAdv2014,GPGJMPA2016,SWsublinear}.

\subsection{Almgren's monotonicity formula: the case $A(0)=Id$}

Assume that $0\in\Omega$ and, for now, that $A(0)=Id$. Our goal is to prove an Almgren monotonicity formula centered at the origin. Later on we shall see how to reduce the case where $A(0)$ is any matrix to the one where $A(0)=Id$, and in which way this affects the formulas. The advantage of making this assumption stems from the fact that, near the origin, the problem looks like the one for the Laplacian, for which formulas are easier to derive. As in \cite{Kukavica,SWsublinear,GarofaloGarciaAdv2014}, we define
\[
\mu(x):=\left\langle A(x) \frac{x}{|x|},\frac{x}{|x|}\right\rangle,\qquad x\in \R^m\setminus \{0\}.
\]
The next lemma quantifies the behavior of various functions involving $A$ as $x\to 0$, in terms of 
\[
\|DA\|_{\infty}:=\max \{ \| \partial_{x_k} a_{ij}(x)\|_{L^\infty(\Omega)},\ i,j=1,\ldots, \ell,\ k=1,\ldots, m\}
\] 
(which we assume to be finite, eventually by taking a smaller $\Omega$ from the start). The proof follows computations in \cite{GarofaloGarciaAdv2014}. Here, however, we need to keep track of the dependencies of the constants involved in the monotonicity formula. This is a key factor in passing from a general $A$ to one with $A(0)=Id$ (see the proof of Theorem \ref{thm:Almgren_general} below, and its relation with Theorem \ref{thm:Almgren}).

\begin{lemma}\label{lemma:severalestimates}
There exists a constant $C$, depending only on the dimension $m$ and on an upper bound for $\|DA\|_\infty$, such that, as $|x|\to 0$, 
\begin{enumerate}
\item $\|A(x)-Id\|\leq C|x|$,
\item $|\mu(x)-1|\leq C |x|$,
\item $|\frac{1}{\mu(x)}-1|\leq \frac{C}{1-C |x|}|x|$,
\item $|\frac{1}{\mu^2(x)}-1|\leq \frac{C}{(1-C |x|)^2}|x|$,
\item $|\nabla \mu(x)|\leq C$,
\item $|\div(A(x)\nabla |x|)-\frac{m-1}{|x|}|\leq C$,
\item $|\div(\frac{A(x)x}{\mu(x)})-m|\leq C|x|$.
\end{enumerate}
\end{lemma}
\begin{proof}
The first statement is a direct consequence of the mean value theorem and the fact that the coefficients of $A$ are of class $\cC^1$, which yields $\|A(x)-Id\|\leq \sqrt{m} \|DA\|_\infty|x|$. The second one follows from the identity
\[
\mu(x)-1=\left\langle \frac{x}{|x|},\frac{x}{|x|}\right\rangle+\left\langle (A(x)-Id) \frac{x}{|x|},\frac{x}{|x|}\right\rangle,
\]
combined with the Cauchy-Schwarz inequality and item \emph{1.}, which allows to conclude that $|\mu(x)-1|\leq \sqrt{m}\|DA\|_\infty |x|$.  Items \emph{3.} and \emph{4.} are direct consequences of \emph{2.}, namely,
\begin{align*}
&\left|\frac{1}{\mu(x)}-1\right|\leq \frac{\sqrt{m}\|DA\|_\infty }{1-\sqrt{m}\|DA\|_\infty |x|}|x|,\\
&\left|\frac{1}{\mu^2(x)}-1\right|\leq \frac{\sqrt{m}\|DA\|_\infty (2+\sqrt{m}\|DA\|_\infty)}{(1-\|DA\|_\infty |x|)^2}|x|.
\end{align*}
Regarding \emph{5.}, from the proof of \cite[Lemma 4.2]{GarofaloGarciaAdv2014} we get
\begin{align*}
|\partial_{x_k} \mu(x)|\leq &\left|\sum_{i,j=1}^m \frac{\partial_{x_k}a_{ij}(x)x_ix_j}{|x|^2}+\sum_{j=1}^m\frac{2(b_{kj}(x)-\delta_{kj})x_j}{|x|^2}-\sum_{i,j=1}^m\frac{2(a_{ij}(x)-\delta_{ij})x_ix_jx_k}{|x|^4}\right|\\
								\leq & \|DA\|_\infty \left( \sum_{i,j=1}^m \frac{|x_i| |x_j|}{|x|^2}+\sum_{j=1}^m \frac{2|x_j| |x|}{|x|^2}+\sum_{i,j=1}^m \frac{2|x_i| |x_j||x_k||x|}{|x|^4}\right)\\
								\leq & \|DA\|_\infty(3m^2+2m).
\end{align*}
As for item \emph{6.}, following the computations in \cite[Lemma 4.1]{GarofaloGarciaAdv2014}, we see that
\[
\div (A(x)\nabla |x|)=\frac{m-1}{|x|}+\div((A(x)-Id)\nabla |x|),
\]
and
\begin{multline*}
|\div(A(x)-Id)\nabla |x|)|=\left|\sum_{i,j=1}^m \partial_{x_i} a_{ij}(x)\frac{x_j}{|x|}+\left(\frac{\delta_{ij}}{|x|}-\frac{x_ix_j}{|x|^3}\right)(a_{ij}(x)-\delta_{ij}) \right|\\
							\leq \|DA\|_\infty \sum_{i,j=1}^m \left( \frac{|x_j|}{|x|}+\left( \frac{1}{|x|}+\frac{|x_ix_j|}{|x|}\right)|x| \right)\leq 3m^2 \|DA\|_\infty.
\end{multline*}
Finally, following \cite[Lemma A.5]{GarofaloGarciaAdv2014}, we have
\begin{multline*}
\div \left(\frac{|x|A(x)\nabla |x|}{\mu(x)}-m\right)=\frac{|x| \div(A(x)\nabla |x|)}{\mu(x)}+1-m- \frac{|x| \langle A(x) \nabla |x|,\nabla \mu(x)\rangle}{\mu^2(x)}  \\
			=\frac{|x|}{\mu(x)}\left(\div(A(x)\nabla|x|)-\frac{m-1}{|x|}\right)+(m-1) \left(\frac{1}{\mu(x)}-1\right)-\frac{|x|\langle A(x)\nabla |x|,\nabla \mu(x)\rangle }{\mu^2(x)}.
\end{multline*}
Since 
\[
|\langle A(x)\nabla |x|,\nabla \mu(x)\rangle|\leq \left(|A(x)-Id|+1\right) |\nabla \mu(x)|\leq C.
\]
(by item \emph{5.}), we have that \emph{7.} follows from \emph{3., 4.} and \emph{6.}
\end{proof}

Set $B_r:=B_r(0)$, $u=(u_1,\ldots, u_\ell)$, $|u|^2:=\sum_{i=1}^\ell u_i^2$, $\langle A(x)\nabla u,\nabla u\rangle :=\sum_{i=1}^\ell \langle A(x)\nabla u_i,\nabla u_i \rangle $ and $f(x,u):=(f_1(x,u_1),\ldots, f_\ell(x,u_\ell))$. Define
\begin{align*}
E(r)&:=\frac{1}{r^{m-2}}\int_{B_r} \left(\langle A(x) \nabla u,\nabla u \rangle - \langle f(x,u),u\rangle\right) dx= \frac{1}{r^{m-2}}\sum_{i=1}^\ell \int_{B_r}\left(\langle A(x)\nabla u_i,\nabla u_i \rangle -f_i(x,u_i)u_i\right) dx,\\
H(r)&:=\frac{1}{r^{m-1}}\int_{\partial B_r} \mu(x) |u|^2\, d\sigma = \frac{1}{r^{m-1}}\sum_{i=1}^\ell \int_{\partial B_r} \mu(x) u_i^2\, d\sigma,
\end{align*}
and the Almgren quotient
\[
N(r):=\frac{E(r)}{H(r)}\qquad \text{ whenever } H(r)\neq 0.
\]
The main purpose of this section is to prove the following result.

\begin{theorem}[Monotonicity formula, case $A(0)=Id$]\label{thm:Almgren}
Take $\omega \Subset \Omega$ such that $0\in \omega$. There exist $C,\bar r>0$ (depending only on the dimension $m$, the ellipticity constant $\theta$, the domain $\omega$, and on an upper bound for $\|DA\|_\infty$ and $\|u\|_\infty$) such that, whenever $r\in (0,\bar r)$, we have that $H(r)\neq 0$, the function $N$ is absolutely continuous, and
$$
N'(r)\geq -C(N(r)+1).
$$
In particular $e^{C r}(N(r)+1)$ is a non decreasing function and the limit $N(0^+):=\lim_{r\to 0^+} N(r)$ exists and is finite. Moreover,
\begin{equation}\label{eq:doubling1}
\left| \left(\log H(r)\right)' -\frac{2}{r}N(r) \right | \leq C \qquad \text{ for every } r\in (0,\bar r).
\end{equation}
\end{theorem}

We present here a sketch of the proof of this result, which is based in ideas from \cite{TavaresTerracini1,nttv,sttz}, with adaptations obtained from \cite{Kukavica,GarofaloGarciaAdv2014,GPGJMPA2016,SWsublinear} allowing us to deal with variable coefficients. Our main goal in carrying out this proof and in repeating some computations is to focus on the dependence of the constants $C$ and $\bar r$, specially on the matrix $A$, something that was not needed in previous papers.

\begin{lemma}\label{lemma:MFormula_aux1}
We have
\begin{equation*}
E(r)=\frac{1}{r^{m-2}}\sum_{i=1}^\ell \int_{\partial B_r} u_{i} \langle A(x) \nabla u_{i}, \nu(x)\rangle \ d\sigma.
\end{equation*}
\end{lemma}

\begin{proof}
For $\lambda<0$, we have
\begin{align*}
E_\lambda(r):=&\frac{1}{r^{m-2}}\int_{B_r} \left(\langle A(x)\nabla u_\lambda,\nabla u_\lambda \rangle- \langle f(x,u_\lambda ),u_\lambda \rangle\right) dx - \frac{1}{r^{m-2}}\int_{B_r}a(x)\mathop{\sum_{ i,j=1}^\ell}_{i\neq j} \lambda  |u_{\lambda,j}|^{\gamma+1}|u_{\lambda,i}|^{\gamma+1} dx\\
					=&\frac{1}{r^{m-2}}\sum_{i=1}^\ell \int_{\partial B_r} u_{\lambda,i} \langle A(x) \nabla u_{\lambda,i}, \nu(x)\rangle \ d\sigma,
\end{align*}
where the last identity is a consequence of testing the $i$--th equation in \eqref{eq:s2} by $u_{\lambda,i}$, integrating by parts, and taking the sum over $i$. Passing to the limit as $\lambda\to -\infty$ and using assumption $(H4)$, yields the claim.
\end{proof}

\begin{lemma}\label{lemma:MFormula_aux2}
Let $\omega\Subset \Omega$ be such that $0\in \omega$. There exist constants $C, \bar r>0$, depending only on the dimension $m$, on $\omega$ and on an upper bound for $\|DA\|_\infty$, such that
\[
\left| H'(r)-\frac{2}{r}E(r) \right| \leq CH(r) \qquad \text{ for every  } r\in (0,\bar r).
\]
In particular, we get \eqref{eq:doubling1}.
\end{lemma}

\begin{proof}
We combine the proof of \cite[Lemma 3.3]{SWsublinear} with the estimates from Lemma \ref{lemma:severalestimates}. By \cite[eq. (3.3)]{SWsublinear},
\begin{align*}
\frac{d}{dr}\int_{\partial B_r} \mu(x) u_i^2 =   2\int_{\partial B_r} u_i \langle A(x)\nabla u_i,\nabla |x|\rangle + \int_{\partial B_r} u_i^2 \div(A(x)\nabla |x|),
\end{align*}
which, together with Lemma \ref{lemma:MFormula_aux1}, yields
\begin{align*}
H'(r)=&\frac{1-m}{r} H(r)+\frac{2}{r^{m-1}}\sum_{i=1}^\ell \int_{\partial B_r} u_{i} \langle A(x) \nabla u_{i}, \nu(x)\rangle+ \frac{2}{r^{m-1}}\sum_{i=1}^\ell\int_{\partial B_r} u_i^2 \div(A(x)\nabla |x|)\\
			=&\frac{1-m}{r}H(r) + \frac{2}{r}E(r) + \frac{1}{r^{m-1}}\sum_{i=1}^\ell \int_{\partial B_r} u_i^2 \div(A(x)\nabla |x|)
\end{align*}
The conclusion now follows from the estimate
\begin{multline*}
\left| \frac{1-m}{r}H(r)+ \frac{1}{r^{m-1}}\sum_{i=1}^\ell \int_{\partial B_r} u_i^2 \div(A(x)\nabla |x|) \right| \\
\leq \frac{1}{r^{m-1}}\sum_{i=1}^\ell\int_{\partial B_r} \mu(x)u_i^2 \left(\frac{\div(A(x)\nabla |x|)}{\mu(x)}- \frac{1-m}{r}\right)\leq CH(r),
\end{multline*}
where the constant $C>0$ arises from items \emph{3.} and \emph{6.} in Lemma \ref{lemma:severalestimates}.
\end{proof}

Define
\[
Z(x):=\frac{A(x)x}{\mu(x)}\sim x \text{ as } x\to 0.
\]
From now on we use the summation convention for repeated indices, unless stated otherwise.

\begin{lemma}[Local Pohozaev-type identities]\label{lemma:localPohozaev_A}  For every $r>0$ such that $B_r\subset \Omega$, we have the following identity (where $A=(a_{ij})$)
\begin{align}
r \int_{\partial B_r} \langle A \nabla u_i,\nabla u_i \rangle &= \int_{B_r} \div Z \langle A\nabla u_i,\nabla u_i\rangle +2\int_{B_r} f_i(x,u_i)\langle \nabla u_i,Z\rangle \nonumber \\					  
&+2\int_{\partial B_r} \langle Z,\nabla u_i\rangle \langle A\nabla u_i,\nu\rangle + \int_{B_r} \langle Z,\nabla a_{hl}\rangle\frac{\partial u_i}{\partial x_h} \frac{\partial u_i}{\partial x_l}-2\int_{B_r}a_{hl}\frac{\partial Z_j}{\partial x_h}\frac{\partial u_i}{\partial x_j} \frac{\partial u_i}{\partial x_l}.   \label{eq:localPoho}
\end{align}
\end{lemma}

\begin{proof}
From system \eqref{eq:s2} we derive an identity for the $u_{\lambda,i}$'s, and then pass to the limit as $\lambda\to -\infty$. For each $i$, from the divergence theorem and the definition of $\mu(x)$ and $Z(x)$, we derive
\begin{align*}
r \int_{\partial B_r} \langle A \nabla u_{\lambda,i},\nabla u_{\lambda,i} \rangle &=\int_{\partial B_r} \langle A\nabla u_{\lambda,i},\nabla u_{\lambda,i}\rangle\langle Z,\nu\rangle=\int_{B_r} \div(\langle A\nabla u_{\lambda,i},\nabla u_{\lambda,i}\rangle Z)\\
						&=\int_{B_r} \div Z \langle A\nabla u_{\lambda,i},\nabla u_{\lambda,i}\rangle +\int_{B_r} \langle Z,  \nabla (\langle A\nabla u_{\lambda,i},\nabla u_{\lambda,i} \rangle ) \rangle.
\end{align*}
Following now \cite[Lemma A.1]{SWsublinear} (see also \cite[Lemma A.9]{GarofaloGarciaAdv2014}), we obtain
\begin{multline*}
\int_{B_r} \langle Z,  \nabla (\langle A\nabla u_{\lambda,i},\nabla u_{\lambda,i} \rangle ) \rangle=\int_{B_r} \langle Z,\nabla a_{hl}\rangle \frac{\partial u_{\lambda,i}}{\partial x_h} \frac{\partial u_{\lambda,i}}{\partial x_l}
		+2\int_{\partial B_r} \langle Z,\nabla u_{\lambda,i}\rangle \langle A\nabla u_{\lambda,i},\nu\rangle \\ -2\int_{B_r} a_{hl} \frac{\partial Z_j}{\partial x_h}\frac{\partial u_{\lambda,i}}{\partial x_j}\frac{\partial u_{\lambda,i}}{\partial x_l}
		+2\int_{B_r} \langle Z,\nabla u_i\rangle \Big(f_i(x,u_{\lambda,i})+a(x)\lambda \mathop{\sum_{j=1}^\ell}_{j\neq i} |u_{\lambda,j}|^{\gamma+1}|u_{\lambda,i}|^{\gamma-1}u_{\lambda,i}\Big).
\end{multline*}
Passing to limit as $\lambda\to -\infty$, the conclusion will follow once we prove the following claim
\begin{equation}\label{eq:claim_Pohozaev}
\sum_{i=1}^\ell \int_{B_r} \langle Z,\nabla u_{\lambda,i}\rangle a(x)\lambda \mathop{\sum_{j=1}^\ell}_{j\neq i} |u_{\lambda,j}|^{\gamma+1}|u_{\lambda,i}|^{\gamma-1}u_{\lambda,i}\to 0 \qquad \text{ as $\lambda\to -\infty$.}
\end{equation}
This statement is true due to the variational character of the coupling term; in fact, as $\lambda\to -\infty$, 
\begin{multline*}
\sum_{i=1}^\ell \int_{B_r} \langle Z,\nabla u_{\lambda,i}\rangle a(x)\lambda \mathop{\sum_{j=1}^\ell}_{j\neq i} |u_{\lambda,j}|^{\gamma+1}|u_{\lambda,i}|^{\gamma-1}u_{\lambda,i}
			=\sum_{i<j}\int_{B_r} \langle Z,\nabla (|u_{\lambda,i}|^{\gamma+1}|u_{\lambda,j}|^{\gamma+1}) \rangle \frac{a(x)\lambda}{\gamma+1}  \\
			=-\sum_{i<j}\int_{B_r} \div Z \frac{ a(x) \lambda}{\gamma+1} |u_{\lambda,i}|^{\gamma+1}|u_{\lambda,j}|^{\gamma+1}
			- \sum_{i<j} \int_{B_r} \langle Z,\nabla a(x)\rangle \lambda |u_{\lambda,i}|^{\gamma+1}|u_{\lambda,j}|^{\gamma+1} \\
			+\sum_{i<j} \int_{\partial B_r} \langle Z,\nu \rangle a(x) \lambda |u_{\lambda,i}|^{\gamma+1}|u_{\lambda,j}|^{\gamma+1} \to 0
\end{multline*}
by assumption $(H5)$ and because $\div Z=m+O(r)$ (see item \emph{7.} in Lemma \ref{lemma:severalestimates}). This proves the claim \eqref{eq:claim_Pohozaev} and completes the proof of Lemma \ref{lemma:localPohozaev_A}.
\end{proof}

Let
\[
\widetilde E(r):=\frac{1}{r^{m-2}}\int_{B_r} \langle A(x) \nabla u,\nabla u \rangle=\frac{1}{r^{m-2}}  \int_{B_r} \langle A(x) \nabla u_i,\nabla u_i \rangle.
\]

\begin{lemma}\label{lemma:MFormula_aux3}
We have
\begin{align}
\widetilde E'(r)=&\frac{2}{r^{m-2}}\sum_{i=1}^\ell\int_{\partial B_r} \frac{\langle A\nabla u_i,\nu\rangle^2}{\mu}+\frac{2}{r^{m-1}} \int_{B_r} f_i(x,u_i)\langle Z,\nabla u_i \rangle \nonumber \\
&+\frac{1}{r^{m-1}} \int_{B_r} \langle Z,\nabla a_{hl}\rangle \frac{\partial u_i}{\partial x_h}\frac{\partial u_i}{\partial x_l} +\frac{1}{r^{m-1}}\int_{B_r} \div (Z-x) \langle A\nabla u_i,\nabla u_i\rangle \nonumber \\
&-\frac{2}{r^{m-1}}  \int_{B_r} a_{hl} \frac{\partial (Z_j-x_j)}{\partial x_h} \frac{\partial u_i}{\partial x_j}\frac{\partial u_i}{\partial x_l}. \label{eq:derivative_of_E}
\end{align}
In particular, given $\omega\Subset \Omega$ with $0\in \omega$, there exist constants $C, \bar r>0$ (depending only on $m$, $\theta$, $\omega$ and on an upper bound for $\|DA\|_\infty$) such that, for every $r\in (0,\bar r)$,
\[						
\left| \widetilde E'(r)-\frac{2}{r^{m-2}}\sum_{i=1}^\ell \int_{\partial B_r} \frac{\langle A\nabla u_i,\nu\rangle^2}{\mu}-\frac{2}{r^{m-1}}  \int_{B_r} f_i(x,u_i)\langle Z,\nabla u_i \rangle\right| \leq C \widetilde E(r).
\]
\end{lemma}

\begin{proof}
From Lemma \ref{lemma:localPohozaev_A} and since $\displaystyle \frac{2}{r}\int_{\partial B_r} \langle Z,\nabla u_i\rangle \langle A\nabla u_i,\nu\rangle=2\int_{\partial B_r}\frac{\langle A\nabla u_i,\nu\rangle^2}{\mu}$, 
we have
\begin{multline*}
\left(\int_{B_r} \langle A(x)\nabla u_i, \nabla u_i\rangle\right)'=\frac{m-2}{r}\int_{B_r} \langle A\nabla u_i,\nabla u_i\rangle +\frac{2}{r}\int_{B_r} f_i(x,u_i)\langle \nabla u_i,Z\rangle 						+2\int_{\partial B_r}\frac{\langle A\nabla u_i,\nu\rangle^2}{\mu}  \\ 
+\frac{1}{r}\int_{B_r} \div (Z-x) \langle A\nabla u_i,\nabla u_i\rangle +\frac{1}{r}\int_{B_r} \langle Z,\nabla a_{hl}\rangle\frac{\partial u_i}{\partial x_h} \frac{\partial u_i}{\partial x_l}-\frac{2}{r}\int_{B_r}a_{hl}\frac{\partial (Z_j-x_j)}{\partial x_h}\frac{\partial u_i}{\partial x_j} \frac{\partial u_i}{\partial x_l}.
\end{multline*}
As 
\[
\widetilde E'(r)= \frac{2-m}{r}\widetilde E(r) + \frac{1}{r^{m-2}} \left(\int_{B_r} \langle A\nabla u_i, \nabla u_i\rangle\right)',
\]
we conclude that identity \eqref{eq:derivative_of_E} is true.	

Now, by Lemma \ref{lemma:severalestimates}-\emph{3.,4.,7.}, we have
\[
\Big\|\frac{1}{\mu}-1\Big\|_{L^\infty(B_r)},\;\Big\|\frac{1}{\mu^2}-1\Big\|_{L^\infty(B_r)},\; \|\div(Z-x)\|_{L^\infty(B_r)}\, \leq Cr 
\]
for some constant $C>0$ depending only on the dimension $m$, on $\omega$ and on an upper bound for $\|DA\|_\infty$. Then, using also $(H2)$, we obtain
\begin{multline*}
\left| \frac{1}{r^{m-1}} \int_{B_r} \langle Z,\nabla a_{hl}\rangle \frac{\partial u_i}{\partial x_h}\frac{\partial u_i}{\partial x_l} +\frac{1}{r^{m-1}}\int_{B_r} \div (Z-x) \langle A\nabla u_i,\nabla u_i\rangle-\frac{2}{r^{m-1}}  \int_{B_r} a_{hl} \frac{\partial (Z_j-x_j)}{\partial x_h} \frac{\partial u_i}{\partial x_j}\frac{\partial u_i}{\partial x_l}\right|\\ \leq C\frac{1}{r^{m-2}}\int_{B_r}\langle A\nabla u_i,\nabla u_i\rangle=C\widetilde E(r)
\end{multline*}
(see equations (A.3)--(A.12) in \cite{SWsublinear} for more details). This completes the proof.
\end{proof}

\begin{remark}
\emph{
Observe that identities \eqref{eq:localPoho} and \eqref{lemma:MFormula_aux3}, which can be seen as local Pohozaev-type identities, are equivalent. They correspond to the condition (G3) for the Laplacian stated in \cite{sttz} and \cite{TavaresTerracini1} respectively.}
\end{remark}
\smallskip

\begin{proof}[Proof of Theorem \ref{thm:Almgren}]  This result now follows from standard arguments. Here, as before, we mainly verify the dependence of the constants. Within this proof, $O(1)$ will represent a bounded function of $r$ depending only on $m$, $\theta$, $\omega$ and on an upper bound for $\|DA\|_\infty$ (but which is independent of $\|u\|_\infty$). We have, by Lemma \ref{lemma:MFormula_aux3},
\begin{align}
E'(r)&=\widetilde E'(r)-\frac{2-m}{r^{m-1}}\int_{B_r} f_i(x,u_i)u_i- \frac{1}{r^{m-2}}\int_{\partial B_r} f_i(x,u_i)u_i \nonumber\\
		&=\frac{2}{r^{m-2}}\sum_{i=1}^\ell \int_{\partial B_r} \frac{\langle A\nabla u_i,\nu \rangle^2 }{\mu} +O(1) E(r)+ R(r), \label{eq:bounds1}
\end{align}
where
\[
R(r):= \frac{2}{r^{m-1}}\int_{B_r}f_i(x,u_i)\langle Z,\nabla u_i \rangle+\frac{O(1)}{r^{m-1}}\int_{B_r} f_i(x,u_i)u_i- \frac{1}{r^{m-2}}\int_{\partial B_r} f_i(x,u_i)u_i.
\]
By $(H3)$, there exists $\bar d$ depending on an upper bound for $\|u\|_\infty$ such that $|f_i(x,u_i)|\leq\bar d u_i$. This together with assumption $(H2)$ and Lemma \ref{lemma:severalestimates}-\emph{2.-3.}, yields $Z(x)=A(x)x/\mu(x)=O(1)|x|$ as $x\to 0$ and
\begin{align}
|R(r)| &\leq  O(1)\,\bar d \sum_{i=1}^\ell \left(\frac{1}{r^{m-2}}\int_{B_r} |u_i| |\nabla u_i|  + \frac{1}{r^{m-1}}\int_{B_r} u_i^2 + \frac{1}{r^{m-2}}\int_{\partial B_r} u_i^2\right) \nonumber\\
		&\leq O(1)\,\bar d \sum_{i=1}^\ell \left(\frac{1}{r^{m-2}}\int_{B_r} \langle A\nabla u_i,\nabla u_i\rangle  + \frac{1}{r^{m}}\int_{B_r} u_i^2 + \frac{1}{r^{m-1}}\int_{\partial B_r} u_i^2\right) \nonumber\\
		&\leq O(1)\,\bar d \left(E(r)+H(r)+\frac{1}{r^m}\sum_{i=1}^\ell \int_{B_r} u_i^2\right). \label{eq:bounds2}
\end{align}
Using Poincar\'e's inequality (see \cite[pp. 279--280]{TavaresTerracini1} for the details), we conclude that
\begin{equation}\label{eq:bounds3}
\frac{1}{r^m}\sum_{i=1}^\ell \int_{B_r} u_i^2 \leq O(1)\,\bar d\,(E(r)+H(r))
\end{equation}
for every $r\in (0,\bar r)$ sufficiently small. Combining \eqref{eq:bounds1}--\eqref{eq:bounds2}--\eqref{eq:bounds3}, we arrive at 
\[
E'(r)=\frac{2}{r^{m-2}}\sum_{i=1}^\ell \int_{\partial B_r} \frac{\langle A\nabla u_i,\nu \rangle^2 }{\mu} +O(1)\,\bar d\,(E(r)+H(r)).
\]
Recalling from Lemmas \ref{lemma:MFormula_aux1} and \ref{lemma:MFormula_aux2} that
\begin{equation*}
E(r)=\frac{1}{r^{m-2}}\sum_{i=1}^\ell \int_{\partial B_r} u_{i} \langle A(x) \nabla u_{i}, \nu(x)\rangle \qquad \text{ and } \qquad  H'(r)=\frac{2}{r}E(r)+O(1)H(r),
\end{equation*}
we finally deduce the existence of a constant $C$ with the required properties such that
\begin{align*}
N'(r)&=\frac{E'(r)H(r)-E(r)H'(r)}{H^2(r)}\\
		&=\frac{2}{H^2(r)r^{2m-3}}\left(\sum_{i=1}^\ell \int_{B_r}\frac{\langle A\nabla u_i,\nu \rangle}{\mu}\sum_{j=1}^\ell \int_{\partial B_r} \mu u_j^2 - \Big(\sum_{i=1}^\ell \int_{\partial B_r} u_i \langle A\nabla u_i,\nu\rangle \Big)^2 \right)\\
		&\qquad+\frac{1}{H^2(r)}\left(O(1)\,\bar d\,H(r)(E(r)+H(r))+O(1) E(r)H(r)\right)			\\
&\geq -C(N(r)+1),
\end{align*}
and $e^{Cr}(N(r)+1))$ is nondecreasing whenever $H(r)\neq 0$. Now observe that $H$ solves $H'(r)=a(r)H(r)$ with $a(r)=\frac{2}{r}N(r)+O(1)r$, and by the existence and uniqueness theorem for this ODE we have that $H>0$ for sufficiently small $r>0$. Finally, the validity of \eqref{eq:doubling1} is given by Lemma \ref{lemma:MFormula_aux2}.
\end{proof}

\subsection{Almgren's monotonicity formula: the general case}

We have proved a monotonicity formula under the assumption that $A(0)=Id$. The general case can be reduced to this case in the following way: let $A(x_0)^\frac{1}{2}$ be the square root of the (positive definite) matrix $A(x_0)$, that is, the unique positive definite matrix whose square is $A(x_0)$. We recall that $A(x_0)^\frac{1}{2}$ is also symmetric, it commutes with $A(x_0)$, it has real entries and that the map $x_0 \mapsto A(x_0)^\frac{1}{2}$ is continuous (see for instance \cite{matrixanalysis}). Following \cite{GPGJMPA2016,SWsublinear}, we set
\begin{align*}
&T_{x_0}x:=x_0+ {A}(x_0)^\frac{1}{2}x,\\
&A_{x_0}(x):=A(x_0)^{-\frac{1}{2}}A(T_{x_0}x)A(x_0)^{-\frac{1}{2}},\\
&\mu_{x_0}(x):=\left\langle A_{x_0} \frac{x}{|x|},\frac{x}{|x|}\right\rangle,\\
&f_{x_0}(x,s):=f(T_{x_0}x,s),\\
&v_{i,x_0}:=u_i(T_{x_0}x).
\end{align*}
Observe that $A_{x_0}(0)=Id$. 
Let now
\[
N(x_0,u,r):=\frac{E(x_0,u,r)}{H(x_0,u,r)},
\]
where 
\begin{align*}
&E(x_0,u,r):=\frac{1}{r^{N-2}}\int_{B_r(0)} \Big(\langle A_{x_0}\nabla v_{x_0},\nabla v_{x_0} \rangle-\langle f_{x_0}(x,v_{x_0}),v_{x_0}\rangle \Big) dx\\
&H(x_0,u,r):=\frac{1}{r^{N-1}}\int_{\partial B_r(0)} \mu_{x_0}(x)|v_{x_0}|^2.
\end{align*}
These quantities can be expresed in terms of the original function $u$ in the ellipsoidal set
\[
\cE_{r}(x_0):=\{x\in \R^m:\ |A(x_0)^{-\frac{1}{2}}(x-x_0)|<r\}. 
\]
Namely, by a change of variables one has
\begin{align*}
&\int_{B_r(0)} \langle A_{x_0}\nabla v_{x_0},\nabla v_{x_0}\rangle =\det ( {A(x_0)}^{-\frac{1}{2}})\int_{\cE_{r}(x_0)} \langle A\nabla u,\nabla u\rangle,\\
&\int_{B_r(0)} \langle f_{x_0}(x,v_{x_0}),v_{x_0}\rangle=\det ( {A(x_0)}^{-\frac{1}{2}})\int_{\cE_r(x_0)} \langle f(x,u),u\rangle, \\
&\int_{\partial B_r(0)} \mu_{x_0}(x) |v_{x_0}(x)|^2\, d\sigma(x)=  \int_{\partial \cE_{r}(x_0)} b_{x_0}(y) |u(y)|^2\, d\sigma(y),
\end{align*}
where $b_{x_0}(y):= c(x_0,y)|A(x_0)^{-\frac{1}{2}}(y-x_0)|^{-2} \langle A(x_0)^{-1}A(y)A(x_0)^{-1}y,y\rangle$, \ $c(x_0,y)$ being the dilation coefficient/tangential Jacobian (see for instance \cite[Chapter 11]{maggi}), which is continuous and positive. 

\begin{theorem}[Monotonicity formula, general case]\label{thm:Almgren_general}
Take $\omega \Subset \Omega$ and let $u$ be as before. Then there exist $C,\bar r>0$ (depending on the dimension $m$, the ellipticity constant $\theta$ and the domain $\omega$, but \emph{independent} from $x_0$) such that, whenever $r\in (0,\bar r)$ and $x_0\in \omega$, we have that $v_{x_0}$ satisfies identities  
\eqref{eq:localPoho} and \eqref{eq:derivative_of_E},   $H(x_0,u,r)\neq 0$, the function $r\mapsto N(x_0,u,r)$ is absolutely continuous, and
$$
\frac{\partial}{\partial r}N(x_0,u,r)\geq -C(N(x_0,u,r)+1).
$$
In particular, $e^{C r}(N(x_0,u,r)+1)$ is nondecreasing and the limit $N(x_0,u,0^+):=\lim_{r\to 0^+} N(x_0,u,r)$ exists and is finite. Moreover,
\begin{equation}\label{eq:doubling}
\left| \frac{\partial}{\partial r}\log H(x_0,u,r)-\frac{2}{r}N(r) \right | \leq C \qquad \text{ for every } r\in (0,\bar r).
\end{equation}
\end{theorem}

\begin{proof}
This is basically a direct consequence of Theorem \ref{thm:Almgren}. The only thing left to check is the dependence of the constants. But this is straightforward by observing that $\|v_{i,x_0}\|_\infty=\|u_i\|_\infty$ for every $i=1,\ldots, \ell$, and that
\begin{align*}
\|DA_{x_0}(x)\|_\infty=\|DA(T_{x_0}x)A(x_0)^\frac{1}{2}\|_\infty,
\end{align*}
which is uniformly bounded for $x\in \omega$, because of $(H2)$ and the continuity of the map $x_0\mapsto A(x_0)^\frac{1}{2}$. This allows to take $C$ and $\bar r$ which are independent of $x_0$.
\end{proof}

Now that we have shown an Almgren's monotonicity formula with constants independent of $x_0$ in any compactly contained subset of $\Omega$, we have all tools required to conclude the proof of the main result of this appendix.

\subsection{Proof of the regularity result}

\begin{proof}[End of the proof of Theorem \ref{thm:generaltheorem_Lip_Reg}]

$(a):$ To prove that the functions $u_i$ are Lipschitz continuous for any $i = 1,...,\ell$ we argue as in the proof of \cite[Proposition 3.4]{sttz}, with minimal adaptations at this point: 
\begin{itemize}
\item we have an elliptic divergence type operator instead of the pure Laplacian operator, therefore the estimates will depend on the ellipticity constant $\theta$;
\item the identity \eqref{eq:localPoho} plays the role of the identity in the last assumption of \cite[Proposition 3.4]{sttz}, while the monotonicity formula (Theorem \ref{thm:Almgren_general}) plays the role of the monotonicity formula \cite[Theorem 3.3]{sttz}.
\end{itemize}
For related proofs of Lipschitz continuity in similar contexts, see also \cite[Section 4.1]{nttv} or \cite[Section 2.4]{Tavares}, the latter being a more detailed version of the former.

\smallbreak

$(b):$ Regarding the regularity properties of $\Gamma:=\{x\in \Omega: \ u_i(x)=0\ \forall i=1,\ldots, \ell\}$, we argue as in the proof of \cite[Theorem 1.1]{TavaresTerracini1}:
\begin{itemize}
\item again, here we have an elliptic divergence type operator instead of the Laplacian;
\item formula \eqref{eq:derivative_of_E} plays the role of the expression for the derivative of $\widetilde E(x_0,U,r)$ in the statement of \cite[Theorem 1.1]{TavaresTerracini1} (see condition (G3) therein), while our Theorem \ref{thm:Almgren_general} plays the role of the monotonicity formula \cite[Theorem 2.2]{TavaresTerracini1}.
\end{itemize}
At a regular point $x_0\in \Gamma$, identity \eqref{eq:derivative_of_E} (or, equivalently, the local Pohozaev identities \eqref{eq:localPoho}) together with the equations
\[
-\div(A(x)\nabla u_i)=f_i(x,s)\quad \text{ in the open set  } \{x\in \Omega:\ u_i(x)>0\},\qquad i=1,\ldots, \ell,
\]
given by assumption $(H6)$, provide the free boundary condition
\[
\lim_{x\to x_0^+} \langle A(x)\nabla u_i,\nabla u_i\rangle=\lim_{x\to x_0^-} \langle A(x)\nabla u_j,\nabla u_j\rangle\neq 0,
\]
where $x\to x_0^\pm$ are the limits taken from opposite sides of $\Gamma$; see \cite[Section 2]{TavaresTerracini1} for the details.

For related proofs of regularity in similar contexts, see also \cite[Theorem 1.7]{sttz} or \cite[Chapter 3]{Tavares}.
\end{proof}

\begin{remark}
\emph{
We remark that Theorem \ref{thm:generaltheorem_Lip_Reg} can be seen as a direct consequence of Theorem 7.1 in \cite{TavaresTerracini1}. However, since the latter result is presented without proof, we have decided to write this appendix and give all the necessary details.}
\end{remark}

\subsection*{Acknowledgement}

The authors would like to thank N. Soave for pointing out some references regarding monotonicity formulas for equations with variable coefficients.

\phantomsection 	

\bigskip

\begin{flushleft}
\textbf{M\'onica Clapp}\\
Instituto de Matemáticas\\
Universidad Nacional Autónoma de México\\
Circuito Exterior, Ciudad Universitaria\\
04510 Coyoacán, Ciudad de México, Mexico\\
\texttt{monica.clapp@im.unam.mx} 
\medskip

\textbf{Angela Pistoia}\\
Dipartimento di Metodi e Modelli Matematici\\
La Sapienza Università di Roma\\
Via Antonio Scarpa 16 \\
00161 Roma, Italy\\
\texttt{angela.pistoia@uniroma1.it} 
\medskip

\textbf{Hugo Tavares}\\
Departamento de Matemática do Instituto Superior Técnico\\
Universidade de Lisboa\\
Av. Rovisco Pais\\
1049-001 Lisboa, Portugal\\
\texttt{hugo.n.tavares@tecnico.ulisboa.pt}
\end{flushleft}

\end{document}